\documentclass{amsart}
\usepackage{amsmath}
\usepackage{amsthm}
\usepackage{graphicx}
\usepackage{amsfonts}
\usepackage{amssymb}
\usepackage{textcomp}
\usepackage{bm}
\usepackage{color}
\usepackage{mathrsfs}
\usepackage{epsfig}
\usepackage{url}
\usepackage{mathrsfs,a4wide}
\usepackage{calligra}
\usepackage{xcolor}
\usepackage[normalem]{ulem}
\usepackage{mathtools}
\usepackage{esint}

\theoremstyle{plain}
\newtheorem*{theorem*}{Theorem}
\newtheorem{theorem}{Theorem}[section]

\newtheorem{lemma}[theorem]{Lemma}
\newtheorem{proposition}[theorem]{Proposition}
\newtheorem{corollary}[theorem]{Corollary}
\newtheorem{remark}[theorem]{Remark}
\newtheorem{definition}[theorem]{Definition}
\theoremstyle{definition}
\theoremstyle{remark}
\numberwithin{equation}{section}

\newcommand{\Qscr}{{\mathscr{Q}}}

\newcommand{\Ad}{\mathscr{A}}
\newcommand{\mad}{\mathrm{D}}
\newcommand{\Td}{T^{D}}
\newcommand{\sal}{S}
\newcommand{\tsal}{\widehat{S}}

\newcommand{\D}{\mathrm{D}}

\newcommand{\ol}{\overline}
\newcommand{\dist}{\mathrm{dist}}
\newcommand{\nnu}{\bm{\nu}}
\newcommand{\ttau}{\bm{\tau}}

\newcommand{\ep}{\varepsilon}

\newcommand{\ffi}{\varphi}

\newcommand{\R}{\mathbb{R}}
\newcommand{\N}{\mathbb{N}}
\newcommand{\Z}{\mathbb{Z}}

\newcommand{\F}{\mathcal{F}}

\newcommand{\di}{\textrm{dist}}
\newcommand{\om}{\omega}

\newcommand{\ud}{\;\mathrm{d}}

\newcommand{\Om}{\Omega}
\newcommand{\supp}{\mathrm{supp}\,}
\newcommand{\M}{\mathcal{M}}

\newcommand{\Div}{\mathrm{Div}}

\newcommand{\Huno}{\mathcal H^1}

\newcommand{\weakly}{\rightharpoonup}           
\newcommand{\weakstar}{\stackrel{*}{\weakly}}   
\newcommand{\fla}{\stackrel{\mathrm{flat}}{\rightarrow}}
\newcommand{\flt}{\mathrm{flat}}

\newcommand{\rad}{\mathcal Rad}

\newcommand{\loc}{\mathrm{loc}}

\newcommand{\Qcal}{{\mathcal{Q}}}
\newcommand{\Cc}{C_{\mathrm{c}}}
\newcommand{\cu}{\mathrm{curl}\;}

\newcommand{\Ss}{\mathbb{S}}
\newcommand{\Su}{\mathbb{S}^1}

\newcommand{\sm}{\setminus}
\newcommand{\Ld}{\mathcal{L}^2}
\newcommand{\iii}{\mathrm{int}}
\newcommand{\fff}{\mathrm{frac}}

\newcommand{\MMM}{\color{black}}
\newcommand{\EEE}{\color{black}}
\newcommand{\tdelta}{\widehat{\delta}}

\newcommand{\strictly}{\overset{\mathrm{strict}}\weakly}
\newcommand{\Ccal}{{\mathscr{C}}}
\newcommand{\B}{\mathscr{B}}

\def\XXint#1#2#3{{\setbox0=\hbox{$#1{#2#3}{\int}$}
     \vcenter{\hbox{$#2#3$}}\kern-.5\wd0}}

\newcommand{\res}{\mathop{\hbox{\vrule height 7pt width .5pt depth 0pt
\vrule height .5pt width 6pt depth 0pt}}\nolimits}
\usepackage{latexsym}
\newcommand{\newatop}{\genfrac{}{}{0pt}{1}}

\makeatletter
\def\@splitop#1#2\@nil{$\mathscr{#1}\!\!$\calligra#2\,\,}
\newcommand*\DeclareCursiveOperator[2]{%

 \newcommand#1{\mathop{\mbox{\@splitop#2\@nil}}\nolimits}}
\makeatother
\DeclareCursiveOperator{\Anew}{A}
\DeclareCursiveOperator{\Bnew}{B}
\DeclareCursiveOperator{\Cnew}{C}
\DeclareCursiveOperator{\Dnew}{D}
\DeclareCursiveOperator{\Enew}{E}
\DeclareCursiveOperator{\Qnew}{Q}

\title[Jacobian for $BV$ maps]
{ Approximation of topological singularities 
\\ through free discontinuity functionals: \\
the critical and super-critical regimes 
}
\author[V. Crismale]
{V. Crismale}
\address[Vito Crismale]{Dipartimento di Matematica ``Guido Castelnuovo'', Sapienza Universit\`a di Roma, Piazzale Aldo Moro 2, I-00185 Roma, Italy}
\email[V. Crismale]{crismale@mat.uniroma1.it}
\author[L. De Luca]
{L. De Luca}
\address[Lucia De Luca]{Istituto per le Applicazioni del Calcolo ``M. Picone'', IAC-CNR, 00185 Rome, Italy}
\email[L. De Luca]{lucia.deluca@cnr.it}
\author[R. Scala]
{R. Scala}
\address[Riccardo Scala]{Dipartimento di Ingegneria dell'Informazione e Scienze Matematiche, Universit\`a di Siena, 53100 Siena, Italy.}
\email[R. Scala]{riccardo.scala@unisi.it}
\usepackage{hyperref}
\begin{document}
\begin{abstract} 
We further investigate the properties of an approach to topological singularities through free discontinuity functionals of Mumford-Shah type proposed in \cite{DLSVG}. We prove the variational equivalence between such energies, Ginzburg-Landau, and Core-Radius for anti-plane screw dislocations energies in dimension two, in the relevant energetic regimes $|\log \varepsilon|^a$, $a\geq 1$, where $\varepsilon$ denotes the linear size of the process zone near the defects. 

Further, we remove the \emph{a priori} restrictive assumptions that the approximating order parameters have compact jump set. This is obtained by proving a new density result for $\Ss^1$-valued $SBV^p$ functions, approximated through functions with essentially closed jump set, 
 in the strong $BV$ norm. 
%

\end{abstract}
\maketitle
\tableofcontents
\section*{Introduction}
This paper concerns the analysis of topological singularities, which is a central topic in models arising in Physics and Materials Science. Vortices in superconductivity and superfluidity and (screw end edge) dislocations in single crystal  plasticity are the main examples of such phenomenon \cite{AO, HL, HB, Lo1, Lo2, Mermin}.

In the last decades several models have been introduced to describe the emergency of these objects. Among them,
the most celebrated is the Ginzburg-Landau (GL) model, mainly studied in the context of superconductivity. 
In such a model, the order parameter is a function $u\in H^1(\Omega;\R^2)$ and the energy functional (in its simplest form) reads as
\begin{equation}\label{def:gl_intro}
	\mathcal{E}^{\mathrm{GL}}_{\ep}(u):=\frac{1}{2}\int_{\Omega}|\nabla u|^2\ud x+\frac{1}{\ep^2}\int_{\Omega}\big(1-|u|^2\big)^2\ud x,
\end{equation}
where the parameter $\ep>0$ is referred to as {\it coherence length}. 
Here and below $\Omega\subset\R^2$ is a bounded open set with Lipschitz continuous boundary.
A topological singularity is nothing but a point around which $u$ has non-trivial winding number and hence the main object to look at is the Jacobian determinant (of $u$) $Ju:=\det \nabla u$. 
Denoting by $\Ss^1$ the set of unitary vectors in the plane, we notice that close to a topological singularity, $u$ cannot be $\Ss^1$-valued
 (a singularity can be seen somehow as a zero of the order parameter); therefore, the parameter $\ep$ can be interpreted as the size of the region where 
 $u$ fails to take values in $\Ss^1$ and hence as the core-radius of the topological singularity.

 The variational analysis of the (GL) functional has been first systematized in the monography \cite{BBH} (see also \cite{SS2} and the references therein for the asymptotic analysis in terms of $\Gamma$-convergence), where the (GL) model is compared with (and shown somehow to be ``equivalent'' to) the so-called core-radius (CR) approach,  in antiplane elasticity. 
 Within this framework, the main variable is represented by the distribution of topological singularities $\mu=\sum \alpha_i\delta_{\xi_i}$ (with integer weights $\alpha_i$) but  the energy functional depends both on $\mu$ - which here plays the role of the Jacobian in (GL) - and on a map $u\in H^1(\Omega_\ep(\mu);\Ss^1)$ ``compatible with $\mu$''. Here, $\ep$ is the core-radius, $\Omega_\ep(\mu):=\Omega\setminus\bigcup_{i}\overline{B}_\ep(\xi_i)$ and the notion of compatibility is given by the fact that $\deg(u,\partial B_\ep(\xi_i))=\alpha_i$ (assuming that the balls $B_\ep(\xi_i)$ are pairwise disjoint). The energy of the system thus writes
\begin{equation}\label{def:cra_intro}
	\mathcal{E}^{\mathrm{CR}}_{\ep}(\mu,u):=\frac{1}{2}\int_{\Omega_\ep(\mu)}|\nabla u|^2\ud x+|\mu|(\Omega).
\end{equation}
Here, the quantity $|\mu|(\Omega)$ plays the same role of the potential term in $\mathcal{E}^{\mathrm{CR}}_{\ep}$, namely, avoids that the cores do not cover the whole domain; in other words, it serves only to guarantee compactness and does not provide any energy contribution in the asymptotics as $\ep\to 0$.

The (CR) approach is mostly used to model screw dislocations in semi-discrete theories. Loosely speaking, in pure (anti-plane) elasticity the bulk energy is determined by the Hooke's law, and 
reads as $\frac 1 2\int_{\Omega}|\nabla w|^2\ud x$, where the displacement $w$ lies in $H^1(\Omega)$.
In presence of a finite distribution $\mu=\sum_i\alpha_i\delta_{x_i}$ of (scalar) defects, the material has a purely plastic behavior in the cores $B_\ep(\xi_i)$ and, oversimplifying, such a plastic contribution can be expressed by $|\mu|(\Omega)$.
Moreover, along a closed circuit enclosing the singularity $\xi_i$, a displacement $w$ compatible with $\mu$ should have a jump $[w]$ equal to $\alpha_i$. Therefore, the displacement $w$ is only in $SBV^2(\Omega_\ep(\mu))$ with $[w]\in\Z$ and its elastic energy should be given by $\frac{1}{2}\int_{\Omega_\ep(\mu)}|\nabla w|^2\ud x$, where $\nabla w$ is the absolutely continuous part of $\mathrm{D}w$. 
Setting 
\begin{equation}\label{fase}
u=e^{2\pi\imath w},
\end{equation}
one obtains that the total energy associated to the pair $(\mu,u)$ is given by $\mathcal{E}^{\mathrm{CR}}_\ep$.
\medskip
 
In this paper, we adopt a different viewpoint, following the approach proposed in \cite{DLSVG}. 
The main feature is that the order parameter is now an $\Ss^1$-valued map, as in the (CR) approach, defined on the whole $\Omega$, as in the (GL) approach. Clearly, in presence of topological singularities, such a map cannot be in $H^1(\Omega;\Ss^1)$. But, instead of removing small disks around the singularities (as in (CR)) or to weaken the $\Ss^1$-constraint (as in (GL)), the map $u$ is now allowed to jump. More precisely, $u$ is a {\it special function of bounded variation}  with square-integrable approximate gradient (i.e., $u\in SBV^2(\Omega; \Ss^1)$).
The energy functional we consider is 
\begin{equation}\label{defEneG:intro}
	\F_\ep(u):= \int_{\Omega}\frac 1 2 |\nabla u|^2\ud x+\frac 1 \ep\Huno(S_u),
\end{equation}
where $\ep>0$ is a small parameter determining the size of the jump set $S_u$ of $u$. Here and throughout the paper $\Huno$ denotes the (one-dimensional) Hausdorff measure.

Formally, the functional $\mathcal F_\ep$ 
has the structure of the Mumford-Shah functional \cite{MS89}, but the $\Ss^1$-constraint makes the analysis completely different.
Indeed, having in mind the identity \eqref{fase} and the (CR) approach for screw dislocations, jumps of the map $u$ correspond to non-integer jumps of the displacement $w$ and should pay energy.
 In other words, the (amplitude of the) jump $[w]$ of the displacement exhibits  a transition between integers in a little portion  of $S_w$. 
The transition is assumed to have length of the order of $\ep$, and corresponds to the presence of singularities. 
In this respect, 
$\Huno(S_u)$ is the analogue of the potential term in (GL) and of the plastic term in (CR) and
the parameter $\ep$ can be understood also in this case as the core-radius of the singularity.

We highlight that for $SBV$ maps the definition of topological degree as well as that of Jacobian determinant are not so standard so that the notion of topological singularity is not so clear as in (GL) and in (CR). Nevertheless, in \cite{DLSVG}, using the {\it minimal lifting} in \cite{Jerrard}, a notion of  Jacobian determinant is provided also for $SBV$ functions; we recall such a definition in Section \ref{sec:prelim} (see \cite{M} where in a more restrictive setting this notion was first introduce, and see also \cite{BCD} for a different definition of Jacobian determinant in dimension 2).
In a nutshell, given a map $u\in SBV(\Omega;\Ss^1)$, 
the Jacobian $Ju$ of $u$ is defined as the boundary of the $1$-current $T_u$
defined by 
\begin{equation}\label{corrente_intro}
\begin{aligned}
T_u:=&\frac12(-u^1\partial_{x_2}u^2+u^2\partial_{x_2}u^1;u^1\partial_{x_1}u^2-u^2\partial_{x_1}u^1)+\frac12(u^+\wedge u^-)\bm\tau\res S_{u}\\
=:&\,T_u^{D}+T_u^S\,.
\end{aligned}
\end{equation}
We stress that dealing with the energy considered in  \cite{DLSVG},  that is 
\begin{equation}\label{defEne_intro}
	\mathcal G_\ep(u) := \int_{\Omega}\frac 1 2 |\nabla u|^2\ud x+\frac 1 \ep\Huno(\overline{S}_u),
\end{equation}
rules out, for instance, $SBV^2$ functions with jump set dense in the whole $\Omega$; such a fact is crucial to prove compactness, rendering the analysis more simpler.

In \cite[Theorem 3.1]{DLSVG} the $\Gamma$-convergence analysis of the  functional  $\mathcal G_\ep$ at the energy regime $|\log\ep|$ has been developed.
As one may expect,  such an analysis reveals that the functional $\mathcal G_\ep$ shares the same compactness and $\Gamma$-convergence properties of the functionals $\mathcal{E}_{\ep}^{\mathrm{CR}}$ and $\mathcal{E}_\ep^{\mathrm{GL}}$.
Specifically, as $\ep\to 0$, the Jacobian determinant tends to concentrate around a finite number of effective singularities and the $\Gamma$-limit of
the functionals $\frac{\mathcal G_\ep}{|\log\ep|}$ is given (up to multiplicative constants) by the total variation of 
the limiting measure of the Jacobians. 
Since, in view of the possibile presence of short dipoles, a uniform bound on the total variation of the dislocations' distributions is not available, the natural setting for such an asymptotic analysis is the (strong) flat convergence for Jacobian determinants, the flat topology being the strong topology in the dual of Lipschitz continuous functions with compact support in $\Omega$.
\medskip

In the present paper we generalize the analysis done in \cite{DLSVG} along two directions.
On the one hand, we show that the penalization term can be ``weakened'' considering only the length of the jump set instead of its closure, i.e., working with the functional $\F_\ep$ rather than with $\mathcal G_\ep$.  Therefore it is not needed to assume \emph{a priori} that the jump set is compact. \EEE On the other hand, we show that the functional $\F_\ep$ shares the same asymptotic behavior of the functionals $\mathcal{E}_{\ep}^{\mathrm{CR}}$ and $\mathcal{E}_\ep^{\mathrm{GL}}$ also in other energy regimes.

The first improvement is obtained by means of a density result in $SBV^p(\Omega; \Ss^1)$, $p>1$, with respect to energies $\F_\ep$ for fixed $\varepsilon>0$, through functions in $SBV^p(\Omega;\Ss^1)$ with (essentially) closed jump set, converging in the strong $BV$ norm and such that also the two unilateral traces of the approximants along the jump set converge; in particular, by using the characterization of $Ju$ as the boundary of $T_u$ in \eqref{corrente_intro}, the strong convergence of Jacobian determinants with respect to flat norm follows. 
 
 Our result hinges on tools developed in a slightly different setting, that is when only the symmetric part of the diffuse gradient is controlled in some $L^p$, for $p>1$, rather than the whole diffuse gradient. Mechanically, this corresponds to consider fracture models for general linearized elasticity without the anti-plane assumption, described by the Griffith functional \cite{Griffith} instead of the Mumford-Shah one.
 
In fact, the main tool for density results developed in the context of Mumford-Shah functional is an approximated Poincar\'e-Wirtinger inequality for $SBV^p$ functions with small ($\mathcal{H}^{d-1}$-measure of the) jump set, due to De Giorgi-Carriero-Leaci (\cite{DeGCarLea}): given $u\in SBV^p$ there exists a truncation in $W^{1,p}$ such that $u$ differs from $w$ on an exceptional set $\omega$ whose volume is controlled by $(\mathcal{H}^{d-1}(\sal_u))^{1^*}$, $d\geq 2$ being the space dimension 
and $1^*:=d/d{-}1$. In the same paper, this result has been used to prove that the jump set of Mumford-Shah minimizers is essentially closed, namely the $\mathcal{H}^{d-1}$ measure of the jump set equals that of its closure. After short time a generalization for $SBV^p(\Omega; \Ss^{k-1})$ maps has been proven in \cite{CarLea91}; by combining such a generalization with an argument in \cite{BraChP96} (cf.\ Lemma~5.2 therein), one can show that the Mumford-Shah functional in $SBV^p(\Omega; \Ss^1)$ can be approximated through $SBV^p(\Omega; \Ss^1)$ functions having essentially closed jump set and converging pointwise. 
Unfortunately, pointwise convergence of a sequence of functions does not guarantee  convergence of the corresponding Jacobians; since we need convergence in the flat norm such an approach is not satisfactory for our purposes. 
For functions with finite Griffith energy (with exponent $p$), that is in the space $GSBD^p$ (\cite{DM13}), a fundamental tool is the approximated Poincar\'e-Korn inequality in \cite{ChaConFra14}, stating that for any $u \in GSBD^p$ with small jump set there exists an infinitesimal rigid motion $a$ (i.e.\ an affine function with null symmetrized gradient) such that the $L^{\frac{dp}{d-1}}$-norm of $u-a$ is estimated by the $L^p$ norm of $e(u)$, the symmetrized diffuse gradient of $u$, outside an exceptional set $\omega$  whose volume is controlled by $(\mathcal{H}^{d-1}(J_u))^{{1^*}}$. Moreover, a convolution of $u \chi_{\omega^c}+a \chi_\omega$ at the same scale of the domain provides a function with $L^p$-norm of the symmetrized diffuse gradient controlled 
by those of $u$. 

This result, on which other contribution in this direction rest (see, e.g.,the approximation in $GSBD^p$ through functions with essentially closed jump set \cite{CFI17Density, ChaCri17} and the analogue of \cite{DeGCarLea} for the Griffith functional \cite{CFI19, CCI19, ChaCri19}), has been generalized by \cite{CFI18} and \cite{CCS22}: here, any $u\in GSBD^p$ is approximated, in terms of the Griffith energy, by functions $W^{1,p}$ on a slightly smaller domain, with essentially closed jump set, which differ from $u$ on a set $\omega$ whose boundary is controlled by $\mathcal{H}^{d-1}(J_u)$; further, in \cite{CFI18} it is shown that in dimension two it is possible to guarantee that the approximants coincide with $u$ in the boundary neighborhood where they are not in $W^{1,p}$. In \cite{CFI18} such an approximation is used to prove an integral representation result, while in \cite{CCS22} the main result is the approximation of any $u\in GSBD^p$,  with respect to the Griffith energy, through functions with essentially closed jump set differing from $u$ on sets of vanishing perimeter.

Moreover, Friedrich \cite{Fri17} proved a \emph{piecewise Korn inequality} in dimension two, showing that up to subtracting piecewise rigid functions (finite sums of infinitesimal rigid motions multiplied by characteristic functions), any $u$ in $GSBD^p$ can be approximated by functions in $SBV^q\cap L^\infty$, for $q<p$, in particular the diffuse gradient of the approximants is estimated on the whole domain by $e(u)$; this is a very powerful tool allowing to overcome the lack of a Coarea Formula in $GSBD$ and then to show, e.g., existence of quasi-static evolutions for Brittle Fracture models (see, for instance, \cite{FriSol18}). In the same spirit, in \cite{Fri19} a similar result has been shown in the Mumford Shah setting, namely dealing only with full diffuse gradients.

Eventually, we refer to \cite{DiPStr23} for the two dimensional analogue of \cite{CarLea91} for maps in $SBV^{p(\cdot)}(\Omega; \Ss^{k-1})$, (with $\Omega\subset \R^2$) whose approximate gradient is integrable with respect to the variable exponent $p(\cdot)$ over $\Omega$ and whose jump set has finite $\Huno$-measure (see also \cite{LeoSciSolVer23} for the variable exponents analogue of \cite{DeGCarLea}), obtained under the assumption that the function $p(\cdot)$ is regular enough and takes values in $(1,2)$. This uses the analogue of the approximation of \cite{CFI18}, proven by employing 
retractions $\mathcal{P}\colon \R^k\sm \mathcal{X}\to \Ss^{k-1}$ with locally $q$-integrable gradient for $q \in [1,2)$, where $\mathcal{X}$ is a smooth complex of codimension two (cf.\ e.g.\ \cite{CanOrl19}).

We then compare our main density result
Corollary~\ref{thm:densitySBV} with \cite[Theorem~5.1]{CCS22}: we are in two dimensions and consider the full diffuse gradient instead of its symmetrized part, however we keep in the approximation the constraint of being $\Ss^1$-valued. We notice that also a version with symmetrized diffuse gradient is readily shown with essentially the same proof, see Theorem~\ref{thm:densityGSBD}.

Since in our application the case $p=2$ is the relevant one, we cannot follow a strategy based on retractions. Moreover, the proof of \cite{CFI18}, \cite{CCS22}, and \cite{Fri19} is not compatible with a non-convex target space such as $\Ss^1$.

Our approach is based on the existence of a \emph{lifting} $\varphi\in SBV^p(\Omega)$ (i.e., such that $u=e^{2\pi \imath \varphi}$) with $\pi \|\varphi\|_{BV}\leq \|u\|_{BV}$ (\cite{DI}), for which we provide a suitable approximation (Theorem~\ref{thm:densitySBVaux}) and then compose the approximants with $e^{2\pi \imath \cdot}$. We observe that, since in \cite[Theorem~5.1]{CCS22} the set on which the traces of the approximants differ from those of the given function has only finite $\mathcal{H}^{d-1}$ measure and then could be dense in the original jump set, after an application of \cite[Theorem~5.1]{CCS22} to $\varphi$ and the composition with $e^{2\pi \imath \cdot}$, one could obtain approximating functions with jump set dense in the integer jump set of $\varphi$, whose $\mathcal{H}^{1}$-measure could be of order $\|u\|_{BV}$.

Therefore we need a more refined density result, which allows to approximate integer jumps with integer jumps as well. The strategy is to work locally near points with integer jump at a scale for which the jump is almost flat and assumes a constant integer value. Then locally most of the jump set of $\varphi$ can be transferred (up to a small error) into a flat segment $S$ on which the jump has the same value; such $S$ can be chosen in such a way that the remaining  small jump set of $\varphi$, on any square with arbitrarily small sidelength with a side contained in $S$, is small compared to the sidelength: this follows by a two dimensional argument drawn from \cite{CFI18}. 

In this way, the approximation for functions with small jump set of \cite{CFI18} may be applied at every scale; therefore a Whitney-type argument combined with the fact that the approximants coincide with the original function on the boundary of any square ensures that the traces are the same on both sides of $S$, so the new jump is still integer.
This strategy may be replicated for different target manifolds, provided a lifting with good $BV$ bounds as in \cite{DI} exists (in this respect see e.g.\ \cite{CanOrlJFA19}), but limiting to two dimensional domains. We notice that in the present context several hard issues arise when considering space dimension $d\geq 3$, such as the lack of the analogue of the Ball Construction.



The density result described above allows to develop the $\Gamma$-convergence analysis for the functional $\F_\ep$ also in different energetic regimes (as it applies for fixed $\varepsilon>0$), thus generalizing the setting of \cite{DLSVG}. 

First, we develop the $\Gamma$-convergence analysis in the so-called critical regime, that is $|\log\ep|^2$. Loosely speaking, since $|\log\ep|$ is the energy cost of an isolated singularity, the fact that $\mathcal F_\ep(u_\ep) \sim|\log\ep|^2$ implies that, for $\ep>0$, the number of singularities of the Jacobians $Ju_\ep$ is of order $|\log\ep|$; therefore, the Jacobians $Ju_\ep$, once rescaled by $|\log\ep|$, should converge (in the flat norm) to a measure $\mu$ that is not anymore atomic but diffuse. Furthermore, we prove that such a measure $\mu$ lies also in $H^{-1}(\Omega)$.
Indeed,
by standard compactness results in $L^2(\Omega)$, also the fields $T^{\mathrm D}_{u_\ep}$, once scaled by $|\log\ep|$, should converge (weakly in $L^2(\Omega)$) to a field $T^{\mathrm{D}}$, whose distributional divergence is shown to be given by $-\pi\mu$. As one may expect, the $\Gamma$-limit accounts both for the plastic contribution of $\mu$ as well as for the elastic energy of $T^{\mathrm{D}}$. That is the reason why the $|\log\ep|^2$ regime is called critical, since in such a case the elastic and plastic effects are of the same order.
The $\Gamma$-convergence analysis for the functional $\mathcal F_\ep$ is provided in Theorems \ref{mainthm} and \ref{mainthm_super} which are proved in Section \ref{sec:5}. 
Second, the proofs of the compactness and of the lower bound are obtained combining the corresponding results for the core-radius approach together with the refined ball construction machinery introduced in \cite{DLSVG} to analyze the $|\log\ep|$ regime.

Finally, adopting the same strategy, in Theorem \ref{mainthm_super} we analyze also the super-critical regimes $|\log\ep|^2\ll N_\ep\ll \frac{1}{\ep}$. In such a case, the interaction elastic energy is larger and larger than the core energy, so that (unless scaling differently the two quantities $Ju_\ep$ and $T^{\mathrm D}_{u_\ep}$) one keeps track of the only $T^{\mathrm D}$ and the Jacobian determinants do not play any role when computing the effective energy.

We highlight that the $\Gamma$-convergence analysis for the functional $\mathcal E_\ep^{\mathrm{GL}}$ in the regime $|\log\ep|^2$  has been developed in \cite{JS2, SS3, SS4}, where the authors consider also the case with magnetic field. The analysis for $\mathcal E_\ep^{\mathrm{CR}}$ is provided in Section \ref{cra_section} and is somehow a short self-contained resume of the results above, along the lines of \cite{AP}.

However, a similar result in the context of edge dislocations within the (CR) approach is proven in \cite{GLP} under the well-separation assumption for the singularities (see also \cite{MSZ} for such an analysis in the nonlinear elasticity framework); such an assumption has been removed in \cite{Ginster}.
 In view of the asymptotic equivalence result \cite{ACP} between the Ginzburg-Landau model and the purely discrete models of $XY$ spin systems and screw dislocations, we have that the analysis in the (GL) context extends also to such discrete models.

\medskip

The paper is organized as follows: After recalling some notations and preliminary results in Section \ref{sec:prelim}, we prove in Section \ref{sec:density} a general density result for $SBV^p(\Om)$ functions in Theorem \ref{thm:densitySBVaux} which implies, as a consequence, Corollary \ref{thm:densitySBV}. This is the result we employ to obtain energy density in the  $\Gamma$-convergence results of Section \ref{sec:model}, actually allowing us to restrict such analysis to $\Ss^1$-valued functions with essentially closed jump set. The latter results are Theorems \ref{mainthm} and \ref{mainthm_super}, stated in Section \ref{sec:model} after recalling the main features of our model. In order to prove them we recall in Section \ref{cra_section} the classical core radius approach, which is the starting point of our analysis. Finally, the proofs of Theorems \ref{mainthm} and \ref{mainthm_super} are given, respectively, in Sections \ref{sec:5} and \ref{sec:6}.
\medskip

\textsc{Acknowledgements:}
The authors are members of the Gruppo Nazionale per l'Analisi Matematica, la Probabilit\`a e le loro Applicazioni (GNAMPA) of the Istituto Nazionale di Alta Matematica (INdAM).

VC acknowledges the financial support of PRIN 2022J4FYNJ
 ``Variational methods for stationary and evolution problems with singularities and interfaces", PNRR Italia Domani, funded by the European Union under NextGenerationEU, CUP B53D23009320006.
 
LDL acknowledges the financial support of PRIN 2022HKBF5C
 ``Variational Analysis of complex systems in Materials Science, Physics and Biology", PNRR Italia Domani, funded
by the European Union via the program NextGenerationEU, CUP B53D23009290006.
 
 RS also acknowledges the partial financial support of the F-cur project number 2262-2022-SR-CONRICMIUR$_{-}$PC-FCUR2022$_{-}002$ of the University of Siena, and the of the PRIN project 2022PJ9EFL "Geometric Measure Theory: Structure of Singular
 Measures, Regularity Theory and Applications in the Calculus of Variations'', PNRR Italia Domani, funded
 by the European Union via the program NextGenerationEU, CUP B53D23009400006.

\section{{Preliminary results}}\label{sec:prelim}
In this section we collect some preliminary notions on the flat norm of measures and currents, as well as some properties of $BV$ functions that will be used throughout the paper.
\medskip
\paragraph{\bf {Flat norm of Radon measures}}  
{Let $n\ge 1$ be an integer and let $U\subset\R^n$ be a bounded and open set.} We denote by $\mathcal M_b(U)$ the space of Radon measures on $U$  {with finite total variation}. If $\mu\in \mathcal M_b(U)$, we denote by $|\mu|(U)$ the total variation of $\mu$\,. We recall that a sequence $\mu_k\in \mathcal M_b(U)$ converges tightly to $\mu\in \mathcal M_b(U)$ if $\mu_k$ converges to $\mu$ weakly* as measure, and $|\mu_k|(U)\rightarrow|\mu|(U)$.
We also introduce the concept of flat norm  of a measure $\mu$, {denoted} by $\|\mu\|_{\flt}$\,, {as}
\begin{align}\label{flat_norm_mu}
	{\|\mu\|_{\flt}:=\sup_{{\newatop{\ffi\in \Cc^{0,1}(U)}{\|\ffi\|_{C^{0,1}(U)}\le 1}}}\int_U\varphi \ud\mu\,.}
\end{align}
{Here and below, the Lipschitz norm $\|\ffi\|_{C^{0,1}(U)}$ is defined by
$$
\|\ffi\|_{C^{0,1}(U)}:=\|\ffi\|_{L^\infty(U)}+\sup_{\newatop{x,y\in U}{x\neq y}}\frac{|\varphi(x)-\varphi(y)|}{|x-y|}\,.
$$}
By a density argument we easily see that the supremum in  \eqref{flat_norm_mu} can be equivalently computed among smooth and compactly supported (in $U$) functions $\varphi$ with {$\|\ffi\|_{C^{0,1}(U)}\le 1$}\,.
\medskip
\paragraph{{\bf Flat norm of $k$-currents.}} {Let $n\ge 2$ be an integer and let $U\subset\R^n$ be an open set}.  For every  $k\in\N$ with $0\leq k\leq n$\,, we denote by $\mathcal D^k(U)$ the topological vector space of smooth and compactly supported $k$-forms on $U$, and by $\mathcal D_k(U)$ its dual, i.e., the space of $k$-currents on $U$.

The mass $|T|$ of a current $T\in \mathcal D_k(U)$ is defined as
$$
|T|=\sup\{\langle T,\omega\rangle:\; \omega\in \mathcal D^k(U),\|\omega\|_{L^\infty}\leq1\}\,.
$$
As {done in \eqref{flat_norm_mu} for measures}, we define the {\it flat norm} of a current $T\in\mathcal D_k(U)$ in $U$ by
\begin{equation}\label{flatcurr}
	\|T\|_{\flt,U}:=\sup_{\newatop{\omega\in\mathcal{D}^k(U)}{\|\omega\|_{F,U}\le 1}}\langle T,\omega\rangle,
\end{equation}
where
$$
\|\omega\|_{F,U}:=\|\omega\|_{L^\infty(U)}+\|\mathrm{d}\omega\|_{L^\infty(U)}\,.
$$
In the special case that $T$ is a $0$-current and has finite mass, then it can be standardly {identified} with a measure, and  the flat norm of $T$ coincides with the flat norm of the measure $T$ {defined in \eqref{flat_norm_mu}}.
\medskip
\paragraph{\bf Jacobian for $\Ss^1$-valued Sobolev maps} 
{Let $U\subset\R^2$ be a bounded and open set.}
 Given a map $u\in W^{1,1}(U;\Ss^1)$ we recall that  the distributional Jacobian $Ju=\textrm{Det}(\nabla u)$ of $u$ is defined by 
\begin{equation}\label{Det}
\langle Ju,\ffi\rangle_{U}:=\int_{U}\nabla\ffi\cdot \lambda_u\,\ud x,\qquad\qquad\textrm{for every }\ffi\in \Cc^\infty(U),
\end{equation} 
where \begin{align}\label{lambda_u}
	\lambda_u:=\frac12\Big(-u^1\frac{\partial u^2}{\partial x_2}+u^2\frac{\partial u^1}{\partial x_2};u^1\frac{\partial u^2}{\partial x_1}-u^2\frac{\partial u^1}{\partial x_1}\Big)\,;
\end{align}
notice that
 $\lambda_u\in L^1(U;\R^2)$. 
 
 Moreover, denoting by $j(u)\in L^1(U;\R^2)$ the {\it current} associated to $u$\,, i.e.,  
 \begin{equation}\label{ju}
 j(u):=\frac12(u^1\nabla u^2-u^2\nabla u^1)\,, 
 \end{equation}
 one has $j^\perp(u)=\lambda_u$ and $j(u)=\pi\nabla w$\,, where $w$ is a generic lifting of $u$\,, i.e., a map in $SBV^2(U)$ satisfying \eqref{fase}
and 
  $\nabla w$ is the approximate gradient of $w$.  Furthermore, it is easy to check that
 $$
Ju=-\mathrm{Div}\lambda_u=\mathrm{curl}\,  j(u)=\pi\mathrm{curl}(\nabla w),
 $$
 holds in the sense of distributions.

In the sequel we will use the fact that a function $u\in H^1(U;\Ss^1)$ satisfies $\textrm{Det}(\nabla u)=0$ in the sense of distributions. 
Moreover, if $u\in H^1(U\setminus\overline B;\Ss^1)$, where $B\subset U$ is a ball, then, integrating by parts,
$$\int_{U\setminus \overline B}\lambda_u\cdot \nabla \varphi\ud x=\int_{\partial B}\lambda_u\cdot {\nnu} \varphi \ud\mathcal H^1=\int_{\partial B}j(u)\cdot {\ttau} \varphi \ud\mathcal H^1,\qquad \textrm{for every } \varphi\in \Cc^\infty(U),$$
where ${\nnu}$ is the inner normal vector to $\partial B$, ${\ttau} =-{\nnu}^\perp$ is the counter-clockwise tangent vector to $\partial B$. 
Notice that $j(u)\cdot {\ttau} =\frac12(u^1\frac{\partial u^2}{\partial {\ttau} }-u^2\frac{\partial u^1}{\partial {\ttau} })$ on $\partial B$.

We recall that $\deg(u,\partial B)\in \mathbb Z$ is defined as
\begin{align}
\deg(u,\partial B):=\frac1\pi\int_{\partial B}j(u)\cdot {\ttau}  \ud\mathcal H^1=\frac1\pi\int_{\partial B}\lambda_u\cdot {\nnu} \ud\mathcal H^1\,,\label{def_grado}
\end{align}
whenever $u\in H^{\frac12}(\partial B;\Ss^1)$.
\medskip
\paragraph{\bf Jacobian for $\Ss^1$-valued $SBV$ maps} 
Let $U\subset\R^2$ be a bounded and open set.   
For any $p\in [1,+\infty)$ the symbol $SBV^p(U;\R^2)$ denotes the space of functions $u\in BV(A;\R^2)$ such that the Cantor part  $\mad^c u\equiv 0$\,, and  $\nabla u\in L^p(A;\R^{2\times 2})$, where $\nabla u$ is the density of $\mad^a u$\,, i.e.,  $\mad^a u:=\nabla u\mathcal L^2$\,.
The  space $SBV^p(U;\Ss^1)$ denotes the set of the functions $u\in SBV^p(U;\R^2)$ such that $|u|=1$ a.e. in $U$.

The following result, proven in \cite[Corollary 2.1]{DLSVG}, is specialized here to maps taking values in $\R^2$.
\begin{proposition}\label{JJ}
Let $u\in SBV(U;\R^2)\cap L^\infty(U;\R^2)$\,; then there exists a unique measure $\nu_u \in \mathcal M_b(U;\R^{2\times 2\times 2})$ such that, whenever $\{v_k\}_{k\in\N}\subset C^1(U;\R^2)\cap W^{1,1}(U;\R^2)\cap L^\infty(U;\R^2)$
	satisfies $\|v_k\|_{{L^\infty(U;\R^2)}}\leq C<+\infty$ for all $k\geq1$ and 
	$v_k\strictly u$ in $BV(U;\R^2)$\,, then {$v_k\otimes\nabla v_k\rightarrow  \nu_u$}\,, where $(\nu_u)_{j}^{i,h}$ is defined (for all $\ffi\in \Cc(U)$)
	by 
	\begin{equation*}
	\begin{aligned}
	\int_U\ffi(x)\ud(\nu_u)^{i,h}_j=&\int_{U\setminus \sal_u}\phi(x)u^h(x)\partial_{x_j}u^i(x)\ud x\\
	&+\frac 1 2\int_{\sal_u}\phi(x)(u^{h,+}(x)+u^{h,-}(x))(u^{i,+}(x)-u^{i,-}(x))\bm{\nu}_j(x)\ud\Huno(x)\,,
\end{aligned}
\end{equation*}
	 for every $i,j,h\in\{1,2\}$\,. 
	Finally, if $\{u_k\}_{k\in\N}\subset SBV(U;\R^2)\cap L^\infty(U;\R^2)$ with 
	\begin{equation}\label{Linftybound}
	\|u_k\|_{L^\infty(U;\R^2)}\leq C,
	\end{equation}
	 for some constant $C>0$, and $u_k\strictly u$ in $BV(U;\R^2)$\,, then 
	\begin{equation}\label{convdebstar}
	\nu_{u_k}\weakstar\nu_u\qquad\textrm{in }\mathcal M_b(U;\R^{2\times 2\times 2})\,.
	\end{equation}
\end{proposition}
In the following, for every map $u\in SBV(U;\R^2)\cap L^\infty(U;\R^2)$\,, we set
\begin{equation*}
[u^h\mad_{j}u^i]:=(\nu_u)^{i,h}_j,\;\;\;\;i,j,h\in\{1,2\}\,,
\end{equation*}
For any map $u\in SBV(U;\R^2)\cap L^{\infty}(U;\R^2)$ we introduce the $1$-current $T_u$ defined by 
\begin{equation}\label{unocorr}
\begin{aligned}
T_u:=&\frac12(-[u^1\mad_2u^2]+[u^2\mad_2u^1];[u^1\mad_1u^2]-[u^2\mad_1u^1])\\
=&\frac12(-u^1\partial_{x_2}u^2+u^2\partial_{x_2}u^1;u^1\partial_{x_1}u^2-u^2\partial_{x_1}u^1)+\frac12(u^+\wedge u^-)\bm\tau\res S_{u}\\
=:&\,T_u^{D}+T_u^S\,.
\end{aligned}
\end{equation}
where we have noted $\alpha\wedge \beta=-\alpha\cdot\beta^\perp=\det(\alpha,\beta).$
Notice that $T_u^{D}\in L^1(U;\R^2)$ and that, if $u\in W^{1,1}(U;\Ss^1)$\,, then $T_u=T_u^{D}=\lambda_u$ with $\lambda_u$ defined in \eqref{lambda_u}.
Finally, we highlight that if $u\in SBV(U;\Ss^1)$\,, for any lifting $w\in SBV(U)$ of $u$, i.e., satisfying \eqref{fase}, it holds that 
\begin{equation}\label{perlift}
\Td_u:=\pi\nabla^{\perp}w;\qquad\qquad T^S_u:=\frac 1 2\sin(2\pi(w^{-}-w^+))\bm\tau\res S_{u}\,.
\end{equation}
The distributional Jacobian $Ju\in \mathcal D_0(U)$ of $u$ is defined as the  boundary of $T_u$, namely
\begin{align}
Ju:=\partial T_u\qquad \qquad \text{in }\mathcal D_0(U).
\end{align}
Essentially by definition, it easily follows that 
$$\|Ju\|_{\text{{\rm flat}},U}\leq C\|u\|_{BV},$$
for all $u\in SBV(\Om;\Su)$, for a universal constant $C>0$. 
\begin{remark}
\rm{
We point out that in  general $Ju$ is not a Radon measure. This notion of Jacobian determinant was first introduced in \cite{M} under some special hypotheses on $u$. Under these hypotheses it turns out that $Ju$ is also a Radon measure. 
}
\end{remark}
\section{Density results in $SBV^p(\Omega;\Su)$}\label{sec:density}
In this section we prove that any function $u\in SBV^p(\Omega;\Su)$ can be approximated - in the strong $BV$ norm - by (sequence of) functions in $SBV^p(\Omega;\Su)$ with closed jump set. As a consequence (see Corollary \ref{cor:21041335}), we deduce that the corresponding currents can be approximated in the flat norm.

In what follows for every function $\ffi\in SBV^p(\Omega)$ (with $p\ge 1$) we denote by $\sal_\ffi^\fff$ the {\it fractional jump set} of $\ffi$\,, i.e., $\sal^\fff_\ffi:=\{x \in \sal_\ffi \colon [\ffi] \notin \Z\}$\,, and by  $\sal_\ffi^\iii$ the {\it integer jump set}\,, namely, $\sal_\ffi^\iii:=\sal_\ffi \sm \sal_\ffi^\fff$\,.
The main result of this section is the following.
\begin{theorem}\label{thm:densitySBVaux}
Let $\Omega\subset \R^2$ be a bounded open set with finite perimeter, $p\in (1,+\infty)$, and $\varepsilon>0$. Then for every $\ffi \in SBV^p(\Omega)$  there exist:
\begin{itemize}
\item closed sets $\Gamma^\iii=\Gamma^\iii_\ep$, $\Gamma^\fff=\Gamma^\fff_\ep$, finite unions of disjoint $C^1$ curves;
\item a set $\widetilde{\omega}=\widetilde{\omega}_\ep$, finite union of cubes;
\item a set of finite perimeter $\widehat{\omega}=\widehat{\omega}_\ep$;
\item a function $\theta=\theta_\ep \in SBV^p(\Omega) \cap W^{1,p}(\Omega \sm ({\Gamma^\iii\cup\Gamma^\fff} \cup \overline{\widetilde{\omega}}))$
;
\end{itemize}
such that 
\begin{equation}\label{basic1}
\{\nabla \ffi\neq \nabla \theta\} \subset \widetilde{\omega}\cup \widehat{\omega}, \quad \Ld(\{\ffi\neq \theta\})<\varepsilon, \quad \theta=0\textrm{ in }\widetilde{\omega}, 
\end{equation}
$[\theta](x)\in \Z$ for $\Huno$-a.e.\ $x\in \Gamma^\iii$\,,
\begin{equation}\label{basic2}
|\Huno(\sal^\iii_\ffi)-\Huno(\Gamma^\iii)| + \Huno(\sal^\fff_\ffi \triangle \Gamma^\fff)+\Huno(\partial \widetilde{\omega}) + \Huno(\partial^* \widehat{\omega}) \leq \varepsilon,
\end{equation}
and
\begin{equation}\label{basic3}
 \int_{\Omega} |\nabla \theta|^p \ud x  \leq (1+\varepsilon)  \int_\Omega |\nabla \ffi|^p \ud x.
\end{equation}
Moreover, $\Huno(\Gamma^\fff \cap \{\theta^+ \neq \ffi^+\})+\Huno(\Gamma^\fff \cap \{\theta^- \neq \ffi^-\}) \leq \varepsilon$, where $\theta^\pm$ and $\ffi^\pm$ denote the traces of $\theta$ and $\ffi$ on the two sides of $\Gamma$.
\end{theorem}
Before providing the proof of Theorem \ref{thm:densitySBVaux} we state and prove our desired approximation results for maps in $SBV^p(\Omega;\Su)$.
\begin{corollary}\label{thm:densitySBV}
Let $\Omega\subset \R^2$ be a bounded open set of finite perimeter, $p\in (1,+\infty)$, and $\varepsilon>0$. Then for every $u \in SBV^p(\Omega;\Su)$ there exist:
\begin{itemize}
\item a closed set $\Gamma=\Gamma_\ep$, finite union of disjoint $C^1$ curves;
\item a set $\widetilde{\omega}=\widetilde{\omega}_\ep$, finite union of cubes;
\item a set of finite perimeter $\widehat{\omega}=\widehat{\omega}_\ep$;
\item a function $v=v_\ep \in SBV^p(\Omega; \Su) \cap W^{1,p}(\Omega \sm (\Gamma \cup \overline{\widetilde{\omega}}); \Su)$;
\end{itemize}
such that 
\begin{equation}\label{basic0}
\{\nabla u\neq \nabla v\} \subset  \widetilde{\omega}\cup \widehat{\omega},\quad \Ld(\{u \neq v\}) <\ep, \quad \nabla v=0\textrm{ $\Ld$-a.e. in } \widetilde{\omega}, 
\end{equation}
and
\begin{equation}\label{basic00}
\begin{split}
\Huno(\sal_u \triangle \Gamma)+\Huno(\partial \widetilde{\omega}) + \Huno(\partial^* \widehat{\omega}) 
\leq \varepsilon, \quad \int_{{\Omega}} |\nabla v|^p \ud x  \leq (1+\varepsilon)  \int_\Omega |\nabla u|^p \ud x.
\end{split}
\end{equation}
Moreover, $\Huno(\Gamma \cap \{v^+ \neq u^+\})+\Huno(\Gamma \cap \{v^- \neq u^-\}) \leq \varepsilon$, where $v^\pm$ and $u^\pm$ denote the traces of $v$ and $u$ on the two sides of $\Gamma$.
\end{corollary}
\begin{proof}
Let $u\in SBV^p(\Omega;\Su)$ and let $\ep>0$\,. Then, by \cite[Theorem 1.1]{DI}, there exists $\ffi\in SBV^p(\Omega)$ such that $u=e^{2\pi\imath\ffi}$ with $\pi\|\ffi\|_{BV}\le \|u\|_{BV}$\,. Let $\Gamma^\iii,\,\Gamma^\fff,\, \widetilde\omega,\, \widehat\omega$ be the sets and let $\theta$ be the function provided by Theorem \ref{thm:densitySBVaux}\,. We set $v:=e^{2\pi\imath\theta}$\,. Then, 
$\{\nabla \ffi=\nabla\theta\}\equiv\{\nabla u=\nabla v\}$ $\Ld$-a.e. and
$\{\ffi=\theta\}\subset \{u=v\}$\,, so that, by \eqref{basic1} we immediately deduce \eqref{basic0}.
Furthermore, since $\sal_\ffi^\fff\equiv \sal_u$\,, taking $\Gamma=\Gamma^\fff$, by Theorem \ref{thm:densitySBVaux} we deduce also the last part of the claim.
\end{proof}
\begin{proposition}\label{cor:21041335}
Let $\Omega\subset \R^2$ be a bounded open set with finite perimeter, let $p\in (1,+\infty)$, and let $u \in SBV^p(\Omega;\Su)$. Then there exists $\{u_n\}_{n\in\N} \subset SBV^p(\Omega;\Su)$ with $\Huno(\overline{\sal}_{u_n}\sm \sal_{u_n})=0$ for all $n$ such that
\begin{equation}\label{21041335}
\begin{aligned}
&\|u_n - u\|_{BV(\Omega;\R^2)} \to 0,\\
&\|\nabla u_n\|_{L^p(\Omega;\R^{2\times 2})}\rightarrow \|\nabla u\|_{L^p(\Omega;\R^{2\times 2})},\\
&\mathcal H^1(S_{u_n})\rightarrow \mathcal H^1(S_{u}).
\end{aligned}
\end{equation}
Furthermore
\begin{equation}\label{2104231340}
|T_{u_n} - T_u|(\Om) \to 0,
\end{equation}
where $T_u$ and $T_{u_n}$ are the measures provided by \eqref{unocorr}; in particular
$$\|Ju-Ju_n\|_{\mathrm{flat},\Omega}\to 0,$$ 
with $\|\cdot\|_{\flt, \Omega}$ is the norm defined in \eqref{flatcurr}.
\end{proposition}
\begin{proof}
Let $u\in SBV^p(\Omega;\Su)$ be fixed. For every $n\in\N$ let $u_n$ be the function provided by Corollary \ref{thm:densitySBV} for $\ep=\ep_n=\frac{1}{n}$\,.
%
By \eqref{basic0}
, we have that, for all $s\in [1,+\infty)$,
\begin{equation}\label{10112023_0}
\mathcal L^2(\{u\neq u_n\})\leq \frac1n,\qquad \|u-u_n\|_{L^s(\Omega;\R^2)}\le \frac{1}{n^{\frac1s}},
\end{equation}
and
\begin{equation}\label{ragione_vito}
\begin{aligned}
\|\nabla u-\nabla u_n\|_{L^1(\Omega;\R^{2\times 2})}\le&\, \big(\|\nabla u\|_{L^p(\Omega;\R^{2\times 2})}+\|\nabla u_n\|_{L^p(\Omega;\R^{2\times 2})}\big) \big(|\widetilde\omega_n|+|\widehat\omega_n|\big)^{\frac{1}{p'}}\\
\le&\, \frac{C}{n^2}\|\nabla u\|_{L^p(\Omega;\R^{2\times 2})}\,.
\end{aligned}
\end{equation}
Furthermore, since
\begin{equation*}
\D^{\mathrm{S}}u=(u^+-u^-)\otimes\nnu_u\res\sal_u\ud\Huno,\qquad \qquad \D^{\mathrm{S}}u_n=(u_n^+-u_n^-)\otimes\nnu_{u_n}\res\sal_{u_n}\ud\Huno\,,
\end{equation*}
by triangle inequality, using that $\nnu_u=\nnu_{u_n}$ on $\sal_{u}\cap\sal_{u_n}$, we get
\begin{equation}\label{partesalto}
\begin{aligned}
|\D^{\mathrm{S}}u-\D^{\mathrm{S}}u_n|(\Omega)\le&\, \big(|u^+-u^+_n|+|u^--u^-_n|\big)\Huno(\sal_u\cap\sal_{u_n})\\
&\,+|u^+-u^-|\Huno(\sal_u\setminus\sal_{u_n})+|u_n^+-u_n^-|\Huno(\sal_{u_n}\setminus\sal_u)\\
\le&\,4\Huno(\Gamma_n\cap\{u_n^{+}\neq u^{+}\})+4\Huno(\Gamma_n\cap\{u_n^{-}\neq u^{-}\})+2\Huno(\sal_u\triangle\sal_{u_n})\\
\le&\,\frac{C}{n},
\end{aligned}
\end{equation}
where in the last inequality we have used \eqref{basic00} and the fact that $\sal_{u_n}\subset\Gamma_n\cup \partial\widetilde\omega_n$, to deduce that
\begin{equation*}
\begin{aligned}
\Huno(\sal_u\triangle\sal_{u_n})\le&\, \Huno(\Gamma_n\cap(\sal_u\setminus\sal_{u_n}))+\Huno(\sal_u\triangle\Gamma_n)+\Huno(\partial\widetilde\omega_n)\\
\le&\,\Huno(\Gamma_n\cap\{u_n^{+}\neq u^{+}\})+\Huno(\Gamma_n\cap\{u_n^{-}\neq u^{-}\})+\Huno(\sal_u\triangle\Gamma_n)+\Huno(\partial\widetilde\omega_n)
\le \frac{2}{n}.
\end{aligned}
\end{equation*}
By \eqref{10112023_0}, \eqref{ragione_vito}, \eqref{partesalto}, we immediately deduce the first condition in \eqref{21041335}, whereas the other two easily  follow from \eqref{basic00}.
{Notice that \eqref{21041335} together with the fact that $u_n,u\in SBV^p(\Omega;\Su)$ implies that
\begin{equation}\label{nuova_15122023}
\int_{\sal_u\cap\sal_{u_n}}|[u]-[u_n]|\otimes\nnu_u\ud\Huno\to 0\quad\textrm{and}\quad\Huno(\sal_u\triangle\sal_{u_n})\to 0\qquad\textrm{as }n\to +\infty.
\end{equation}
}
Furthermore, using triangle inequality, \eqref{10112023_0} and \eqref{ragione_vito}, it is easy to check that
\begin{equation*}
|T^{D}_u-T^{D}_{u_n}|(\Om) \le 4\|\nabla u\|_{L^1(\{u\neq u_n\};\R^{2\times 2})}\|u-u_n\|_{L^\infty(\Om;\R^2)} +4\|u_n\|_{L^\infty(\Om;\R^2)}\|\nabla u-\nabla u_n\|_{L^1(\Omega;\R^{2\times 2})}\to 0
\end{equation*}
and that
\begin{equation*}
\begin{aligned}
|T^{S}_u-T^{S}_{u_n}|(\Om)\le&
\frac12 \int_{\sal_u\cap \sal_{u_n}} |u^+_n\wedge u^-_n-u^+\wedge u^-|\ud\Huno\\
&\,
+\frac 1 2\int_{\sal_u\setminus \sal_{u_n}}|u^+\wedge u^-|\ud\Huno+\frac 1 2\int_{\sal_{u_n}\setminus \sal_{u}}|u_n^+\wedge u_n^-|\ud\Huno
 \leq \Huno(\{[u]\neq [u_n]\}) \leq \frac3n ,
\end{aligned}
\end{equation*}
since $S_{u_n}\subset \Gamma_n \cup \partial\widetilde{\omega}_n \cup \partial^* \widehat{\omega}_n$ and $\{[u]\neq [u_n]\}=(S_u\triangle S_{u_n}) \cup (S_u\cap S_{u_n} \cap \{[u]\neq [u_n]\})$.
\end{proof}
We now turn to the proof of Theorem \ref{thm:densitySBVaux}.
\begin{proof}[Proof of Theorem \ref{thm:densitySBVaux}]
Let $\varrho$, $\alpha_1$, $\alpha_2$ be small positive constants to be determined later. We divide the proof into three steps. 

\medskip
\textit{Step 1: Covering the jump set.} Since the sets $$\widehat\sal_u^0:=\sal_u^\fff\qquad \text{and}\qquad\widehat\sal_u^z:=\{x \in \sal_u \colon [u](x)=z\}$$ for all $z\in\Z^*:=\Z\setminus\{0\}$ are countably $(\Huno, 1)$ rectifiable with finite $\Huno$ measure, by \cite[Theorem, 3.2.29]{Federer} for every $z\in \Z$ there exists a countable family $\{M_k^z\}_{k\in \N}$ such that
\begin{equation}\label{2304231952}
\Huno\Big(\tsal_u^z \sm \bigcup_{k=1}^\infty M_k^z\Big)=0
\end{equation}
and, by \cite[Theorem~2.76]{AFP} we may assume that for each $z \in \Z$ and $k\in \N$ the manifold $M_k^z$ is a graph of a $C^1$ and Lipschitz function with Lipschitz constant less than $\alpha_1$\,.
Let us fix $z\in \mathbb{Z}$ such that $\Huno(\widehat{S}_u^z)>0$ (in particular, for the application of the present theorem in this paper, this holds for $z=0$). Then, since for $\Huno$-a.e.\ $x \in \widehat{S}_u^z \cap M_k^z$, $x$ has $\Huno$-density 1 both for $\widehat{\sal}_{u}^z$ and $M_k^z$, for every $k\in \N$ and every such $x \in \widehat{S}_u^z \cap M_k^z$ there exists $\eta(\alpha_2, x) \in (0,\varrho)$ such that
\begin{equation}\label{2304232000}
\begin{aligned}
|\Huno(\ol Q_r(x)\cap \widehat{S}_u^z)-2r| \leq 2\alpha_2 r,\quad |\Huno(\ol Q_r(x)\cap M_k^z)-2r| \leq 2\alpha_2 r, \\
 |\Huno(\ol Q_r(x)\cap (\widehat{S}_u^z\cap M_k^z))-2r| \leq 2\alpha_2 r.
 \end{aligned}
\end{equation}
for every $r \leq \eta(\alpha_2, x)$, and moreover 
\begin{equation}\label{2304232002}
\Huno(\ol Q_r(x)\cap (\widehat{S}_u^z\triangle M_k^z)) \leq \alpha_2 \Huno (\ol Q_r(x) \cap \widehat{S}_u^z),
\end{equation}
for every $r \leq \eta(\alpha_2, x)$; here we recall that  $Q_r(x)$ denotes the (open) cube $Q^{\nu(x)}_r(x)$, centered at $x$, with sidelength $2r$  and with a side normal to $\nu(x)$, the approximate normal to $\sal_u$ (or $\widehat{\sal}_{u}^z$) at $x$. We notice that \eqref{2304232000} holds also for $S_u$ in place of $\widehat{S}_u^z$ or $\widehat{S}_u^z \cap M_k^z$, that is, for $\Huno$-a.e.\ $x \in S_u$, 
we may also assume
\begin{equation}\label{2304232000'}
|\Huno(\ol Q_r(x)\cap S_u)-2r| \leq 2\alpha_2 r
\end{equation}
for every $r \leq \eta(\alpha_2, x)$. Now we introduce $$M:= \sal_u \cap \bigcup_{z\in\Z, \, k\in \N} M_k^z.$$
We also denote by $\widetilde M\subset M$, the set of points $x$ satisfying \eqref{2304232000}, \eqref{2304232002}, and \eqref{2304232002}. From what observed, $\mathcal H^1(M\setminus \widetilde M)=0$; so,
 since the family $\{ \ol Q_r(x)\colon x\in M,\, r \leq \eta(\alpha_2, x)\}$ is a fine cover of $\widetilde M$,  Vitali-Besicovitch's Covering Theorem (see \cite[Theorem~1.10]{Fal86} for its version for cubes) ensures the existence of a disjoint subfamily $\{ \ol Q_{r(\alpha_2,x)}(x)\colon x\in M'\}$, for a countable set $M'=\{x_j\}_{j\in \N}\subset \widetilde M\subset M$ and $r(\alpha_2,x)\leq \eta(\alpha_2,x)$, such that
\begin{equation*}
\Huno\Big(\sal_u \sm \bigcup_{j\in \N} \ol Q_j\Big)=0,
\end{equation*}
where we have noted $Q_j:=Q_{r_j}(x_j)$ and $r_j:=r(\alpha_2,x_j)$ for every $j\in\N$.
Then there exists $J=J(\alpha_2)\in\N$ such that
\begin{equation}\label{2304232338}
\Huno\Big(\sal_u \sm \bigcup_{j=1}^J \ol Q_j \Big)< \alpha_2.
\end{equation}
For every $j \in \{1,\dots,J\}$, let $k_j\in\N$, $z_j\in\Z$ be the indeces such that, $x_j\in  \tsal_u^{z_j}  \cap M_{k_j}^{z_j}$, and \eqref{2304232000}, \eqref{2304232002}, \eqref{2304232000'} hold (for all $j$, such indeces are unique).
Then we  set $$\Gamma_j:= \ol Q_j \cap M_{k_j}^{z_j}\;.$$ We now see that for every $j\in \{1,\dots,J\}$, there hold
\begin{align}\label{aprop}\tag{{a}}
&\text{$\Gamma_j$ is the graph of a $C^1$ and Lipschitz function with Lipschitz constant less than $\alpha_1$;}\\ \label{bprop}{\tag{b}}
&\text{$|\Huno(\ol Q_r(x_j) \cap \Gamma_j)-2r|\leq 2 \alpha_2 r$ for all $0<r\leq r_j$;}\\ \label{cprop}{\tag{c}}
&\text{$\Huno(\ol Q_j\cap (\tsal^{z_j}_u \triangle \Gamma_j)) + \Huno(\ol Q_j \cap (\sal_u \sm {\tsal}^{z_j}_{u})) \leq  8\alpha_2 r_j$;}\\ \label{dprop}{\tag{d}}
&\text{$\Huno\Big(\sal_u  \triangle \bigcup_{j=1}^J \Gamma_j\Big) \leq \alpha_2(1+5 \Huno(\sal_u))$;}\\ \label{eprop}{\tag{e}}
&\text{$\Ld\Big(\bigcup_{j=1}^J \ol Q_j\Big) \leq 3 \varrho  \, \Huno(\sal_u)$.}
\end{align}
Property \eqref{aprop} follows by definition, since $M^z_k$ are graphs of Lipschitz maps with constant less than $\alpha_1$. 
Property \eqref{bprop} follows immediately from  \eqref{2304232000}.
As for the proof of \eqref{cprop}, by \eqref{2304232002} and \eqref{2304232000}, choosing 
\begin{align}\label{alpha2_1}
\alpha_2\le 1\;,
\end{align} we have that
\begin{equation}\label{primalucia}
\Huno(\ol Q_j\cap (\tsal^{z_j}_u \triangle \Gamma_j))\le \alpha_2\Huno(\ol Q_j\cap \tsal^{z_j}_u)\le 2\alpha_2(1+\alpha_2) r_j\le 4\alpha_2 r_j\,;
\end{equation}
moreover, by \eqref{2304232000'} and \eqref{2304232000} we have that
\begin{equation*}
\begin{aligned}
\Huno(\ol Q_j\cap (\sal_u\setminus \tsal^{z_j}_u))=\,&\Huno(\ol Q_j\cap S_u)-\Huno(\ol Q_j\cap \tsal_u^{z_j})\\
\le&\, 2r_{j}(1+\alpha_2)-2r_j(1-\alpha_2)=4\alpha_2r_j\,,
\end{aligned}
\end{equation*}
which, combined with \eqref{primalucia}, yields property \eqref{cprop}.
Property \eqref{dprop} follows from the decomposition
\[
\sal_u \sm \bigcup_{j=1}^J \Gamma_j= \left( \sal_u \sm \bigcup_{j=1}^J \ol Q_j \right) \cup  \bigcup_{j=1}^{J} \big(\ol Q_j \cap (\sal_u \sm\Gamma_j) \big) ,
\]
combined with \eqref{2304232338} and the estimate
\begin{equation*}
\begin{aligned}
\Huno\big(\ol Q_j \cap (\sal_u \triangle\Gamma_j) \big) 
=&\,
\Huno\big(\ol Q_j \cap(\tsal_u^{z_j}\setminus\Gamma_j)\big)+\Huno\big(\ol Q_j \cap((\sal_u\setminus\tsal_u^{z_j})\setminus\Gamma_j)\big)+\Huno\big(\ol Q_j \cap(\Gamma_j\setminus\sal_u) \big)\\
\leq&\,
\Huno\big(\ol Q_j \cap (\tsal_u^{z_j} \triangle\Gamma_j) \big)+\Huno\big(\ol Q_j\cap(\sal_u\setminus\tsal_u^{z_j})\big)
\\
\le&\, 8\alpha_2 r_j \leq 4 \alpha_2 \left( \Huno(\ol Q_j \cap \sal_u) + 2 \alpha_2 r_j) \right)\\
 \leq&\, 5 \alpha_2 \Huno(\ol Q_j \cap \sal_u)
\end{aligned}
\end{equation*}
recalling that the cubes $\ol Q_j$ are pairwise disjoint.
Here the first inequality follows from the fact that $\tsal_u^{z_j}\subset\sal_u$\,, the second one
 from \eqref{cprop}, the third one from \eqref{2304232000'}, the fourth again from  \eqref{2304232000'} choosing \begin{align}\label{alpha2_2}
\alpha_2<\frac 1 5\;.
 \end{align} 
Eventually, recalling that $r_j \leq \varrho$ for every $j$, 
by \eqref{2304232000'}, since $\alpha_2<\frac15<\frac 1 3$, and using again that the cubes $Q_j$ are pairwise disjoint,
we have 
\[
\Ld\big(\bigcup_{j=1}^J \ol Q_j\big) = \sum_{j=1}^J 4 r_j^2 \leq 2 \varrho \sum_{j=1}^J \left( \Huno(\ol Q_j \cap \sal_u) + 2 \alpha_2 r_j \right) \leq 3\varrho\sum_{j=1}^J \Huno(\ol Q_j \cap \sal_u) \leq 3 \varrho \Huno(\sal_u),
\]
from which \eqref{eprop} follows.
Moreover, using \eqref{2304232338} with \eqref{cprop}, and arguing as done to prove \eqref{dprop}, we obtain that
\begin{equation*}
 \widehat{\Gamma}^\iii	:= \bigcup_{j \colon z_j\neq 0} \Gamma_j, \qquad\quad \widehat{\Gamma}^\fff:= \bigcup_{j \colon z_j=0} \Gamma_j, \quad\qquad \widehat{\Gamma}:=\widehat\Gamma^\iii \cup \widehat\Gamma^\fff= \bigcup_{j=1}^J \Gamma_j
\end{equation*}
are finite unions of disjoint $C^1$ curves and
\begin{equation}\label{2404230943}
\Huno(\sal^\iii_u \triangle \widehat{\Gamma}^\iii) + \Huno(\sal^\fff_u \triangle \widehat{\Gamma}^\fff)\leq \alpha_2(1+  5  \Huno(\sal_u)).
\end{equation}

\medskip
\textit{Step 2: Approximation in the cubes $Q_j$.} 
We perform two different approximations depending on whether the cube $Q_j=Q_{r_j}(x_j)$ is such that $z_j=0$ or not. {To shorten the notation, we set $\nu_j:=\nu(x_j)$, where $\nu(x_j)$ is the approximate normal to $\sal_u$ at $x_j$.}

\medskip
\textit{Step 2.1: The case $z_j\neq0$}. This implies that $[u](x_j)=z_j\in\Z^*$\,. 
In this case we first show that there exists a ``big'' set of segments (parallel to $\nu_j^\perp$) in the cube $\ol Q_j$ that do not intersect the jump set $\sal_u$ of $u$ and such that (small) stripes centered at each of this segment contain a ``few'' portion of $\sal_u\setminus\Gamma_j$\,.
To this end, denoting by $x_j+\nu_j^{\perp}$ the straight line orthogonal to $\nu_j$ and passing through $x_j$\,, we define the (signed) distance from such a line as $\mathrm{dist}(x,x_j+\nu_j^{\perp}):=(x-x_j)\cdot \nu_j$\,.
Moreover, for every $\gamma\in(-r_j/2,r_j/2)$ we define 
$$
T_j^\gamma:= \ol Q_j \cap \{\dist(\cdot, x_j + \nu_j^\perp)=\gamma\}
$$ 
and, for every $k\in\N$\,, we set
\begin{equation*}
C^{\gamma,k}_j:=\ol Q_j \cap \{\dist(\cdot, x_j + \nu_j^\perp) \in [\gamma-2^{-  k }r_j, \gamma+2^{- k }r_j]\}. 
\end{equation*}
Let 
\begin{equation}\label{sceltadieta}
\tilde{\eta}:=\frac{1}{160} \eta,
\end{equation}
where $\eta$ is the constant from Proposition~\ref{prop:analoga32CFI} below. 
We set $S_j:=\ol Q_j\cap (\sal_u\setminus\Gamma_j)$\,.
We claim that there exists a set $I^{\tilde\eta}_j \subset (-r_j/2, r_j/2)$ with 
\begin{equation}\label{2405231721}
\mathcal{L}^1(I^{\tilde\eta}_j) \leq \frac{ 160\, \alpha_2  }{\tilde{\eta}}r_j
\end{equation}
such that, 
for every $\gamma \in (-r_j/2, r_j/2)\setminus I^{\tilde\eta}_j$, it holds
\begin{equation}\label{2405231725}
\Huno(S_j \cap C^{\gamma,k}_j)< \tilde{\eta} 2^{-(k+1)} r_j \quad \text{for all }k\in \N
\end{equation}
 and
\begin{equation}\label{1609231124}
\Huno(T^{\gamma}_j\cap \sal_u)=0.
\end{equation} 
Indeed, for $\delta_j^k:= 2^{-k}r_j$, we argue as in the proof of \cite[Theorem~2.1]{CFI18}, considering the family
\begin{equation}\label{2405231842}
\mathscr{I}^{\tilde\eta}_j:=\left\{ [\gamma-\delta^k_j \EEE , \gamma+ \delta^k_j \EEE] \quad\colon\quad \Huno(S_j\cap C^{\gamma,k}_j)\geq \frac{\tilde{\eta}}{2} \delta^k_j\,,\quad k\in\N\,,\quad\gamma\in  (-r_j/2, r_j/2)\right\}
\end{equation}
and $I^{\tilde\eta}_j:=\bigcup_{I\in \mathscr{I}^{\tilde\eta}_j}I$\,. By Vitali's covering theorem, there exists a countable set $\{(\gamma^l, k^l)\}_{l\in \N}$ such that the intervals $[\gamma^l-\delta^{k^l}_j\EEE, \gamma^l+\delta^{k^l}_j\EEE]$ in $\mathscr{I}^{\tilde\eta}_j$ are pairwise disjoint and 
$$
I^{\tilde \eta}_j\subset\bigcup_{l\in\N}[\gamma^l-5\delta^{k^l}_j\EEE, \gamma^l+5\delta^{k^l}_j\EEE]\,.
$$
By property \eqref{cprop} of $\Gamma_j$ we have
\begin{equation*}
8 \alpha_2 r_j \geq \Huno (S_j\cap Q_j)\geq \sum_{l\in \N} \Huno(S_j \cap C_j^{\gamma^l, k^l}) \geq \sum_{l\in \N} \frac{\tilde{\eta}}{2} \delta_j^{k^l} \geq \frac{\tilde{\eta}}{ 20 } \mathcal{L}^1({I}_j^{\tilde\eta}).
\end{equation*}
Then \eqref{2405231721} follows.
By definition of $I^{\tilde\eta}_j$,  every 
$\gamma \in (-r_j/2,r_j/2) \sm I^{\tilde\eta}_j$  does not belong to any interval of the family \eqref{2405231842} and then  satisfies \eqref{2405231725}. 
 Since $\Huno(T_j^{\gamma} \cap \sal_u)=0$ for every $\gamma \in (-r_j/2, r_j/2)$ except at most countable many, we may enforce also \eqref{1609231124}. 
Then the claim is confirmed. 

Let us choose $\gamma_j \in (-r_j/2, r_j/2)\setminus I^{\tilde{\eta}}_j$ satisfying \eqref{2405231725}, \eqref{1609231124} with
\begin{equation}\label{2701240937}
\gamma_j \in \left(0, \frac{161 \alpha_2}{\tilde{\eta}}r_j\right).
\end{equation}
We consider the function
\begin{equation}\label{2405231920}
\widehat{u}_j:= u+ z_j \chi_{H_j} \quad \text{in }\ol Q_j,
\end{equation}
where
$H_j \subset Q_j$ is the closed region delimited by $\Gamma_j$, $T_j^{\gamma_j}$, and the two segments $\Sigma_j^1$, $\Sigma_j^2 \subset \partial Q_j$ joining the two couples of intersection points of $\Gamma_j$ and $T_j^{\gamma_j}$ with the two boundary segments 
\[
B_j^\pm:=\{x_j\pm r_j \tau_j + t \nu_j \colon |t|<r_j \} \subset \partial Q_j,
\] 
being the `tangent vector' $\tau_j$ the unit vector forming an angle of $\pi/2$ with $\nu_j$. \EEE
By \MMM property \eqref{cprop} 
we deduce that
\begin{equation}\label{2405231232}
\begin{split}
\Huno(Q_j & \cap (\sal_{\widehat{u}_j} \sm T_j^{\gamma_j})) 
 \MMM \leq \EEE \Huno(Q_j \cap (\sal_{u} \sm \MMM \Gamma_j \EEE)) + \Huno(\Gamma_j \sm \tsal_u^{z_j})  \leq 8\alpha_2  r_j;
\end{split}
\end{equation}
\MMM by \eqref{1609231124} it holds that \EEE
\begin{equation}\label{0106231130}
 \Huno(T_j^{\gamma_j} \cap \{[\widehat{u}_j] \neq \MMM z_j \EEE \}) =0. 
\end{equation}
Further, for every $s \in (0, r_j)$, denoting \MMM $Q_{j,s}:=x_j + Q_s$ and $\Sigma_{j,s}^1$, $\Sigma_{j,s}^2 \subset \partial Q_{j,s}$ the two segments joining the two couples of intersection points  of $\Gamma_i$ and $T_j^{\gamma_j}$ with the boundary segments
\[
B_{j,s}^\pm:=\{x_j\pm s \tau_j + t \nu_j \colon |t|<r_j \} \subset \partial Q_{j,s},
\]  
in view of \eqref{aprop}  and \eqref{2701240937} \EEE
\begin{equation}\label{0106230928}
\Huno(\Sigma_{j,s}^1 \MMM \cup \EEE \Sigma_{j,s}^2)\leq 2\left( \alpha_1 s \MMM + \frac{161 \alpha_2}{\tilde{\eta}}r_j \right). \EEE 
\end{equation}
Arguing as done before to ensure \eqref{2405231725}, \eqref{1609231124} for $\gamma$ outside a small set, it is possible to find $\widehat{r}_j \in ((1- \sqrt{\alpha_2}) r_j, r_j)$ (for $\alpha_2$ small enough) such that, denoting
\begin{align}\label{27012400949}
&Q^+_{j,\widehat{r}_j}:=\{ x \in Q_j\colon (x-x_j)\cdot \nu_j \in (\gamma_j, \gamma_j+ \widehat{r}_j) ,\, (x-x_j) \cdot \tau_j \in (-\widehat{r}_j, \widehat{r}_j)\},\nonumber\\
&Q^-_{j,\widehat{r}_j}:=\{ x \in Q_j\colon (x-x_j)\cdot \nu_j \in (\gamma_j- \widehat{r}_j, \gamma_j) ,\, (x-x_j) \cdot \tau_j \in (-\widehat{r}_j, \widehat{r}_j)\},
\end{align}
it holds that 
\begin{equation}\label{2405231948}
\begin{split}
&\Huno(\sal_{\widehat{u}_j} \cap \partial Q^\pm_{j,\widehat{r}_j})=0, \\ &\Huno\Big(\sal_{\widehat{u}_j} \cap Q^\pm_{j, \widehat{r}_j} \cap \big(\partial Q^\pm_{j, \widehat{r}_j}+ B_{2^{-k}r_j}(0)\big)\Big) < \tilde{\eta}2^{-(k+1)}r_j \text{ for every }k \in \N. 
\end{split}
\end{equation}
In fact, the same argument as above shows that there exist sets $I_{\mathrm{hor}}^\pm$, $I_{\mathrm{ver}}^\pm \subset (-r_j, r_j)$ with $\mathcal{L}^1(I_{\mathrm{hor}}^\pm)\leq \frac{320 \alpha_2}{\tilde{\eta}}r_j$, $\mathcal{L}^1(I_{\mathrm{ver}}^\pm)\leq \frac{320 \alpha_2}{\tilde{\eta}}r_j$ such that, for 
\begin{equation*}
\begin{split}
& C_{\mathrm{hor},-}^{\widetilde{\gamma},k}:=\{x\in Q_j \colon (x-x_j)\cdot \nu_j \in [\widetilde{\gamma}, \widetilde{\gamma}+ 2^{-k}r_j]\}, \quad C_{\mathrm{hor},+}^{\widetilde{\gamma},k}:=\{x\in Q_j \colon (x-x_j)\cdot \nu_j \in [\widetilde{\gamma}- 2^{-k}r_j, \widetilde{\gamma}]\},\\
& C_{\mathrm{ver},-}^{\widetilde{\gamma},k}:=\{x\in Q_j \colon (x-x_j)\cdot \tau_j \in [\widetilde{\gamma}, \widetilde{\gamma}+ 2^{-k}r_j]\}, \quad C_{\mathrm{ver},+}^{\widetilde{\gamma},k}:=\{x\in Q_j \colon (x-x_j)\cdot \tau_j \in [\widetilde{\gamma}- 2^{-k}r_j,\widetilde{\gamma}] \},
\end{split}
\end{equation*}
it holds that for every $\widetilde{\gamma}\notin I_{\mathrm{hor}}^\pm$, $I_{\mathrm{ver}}^\pm$ \EEE
\begin{equation*}
\Huno(S_j \cap C_{\mathrm{hor},\pm}^{\widetilde{\gamma},k})\leq \tilde{\eta}2^{-(k+3)}r_j, \quad \Huno(S_j \cap C_{\mathrm{ver},\pm}^{\widetilde{\gamma},k})\leq \tilde{\eta}2^{-(k+3)}r_j \quad \text{ for every }k \in \N.
\end{equation*}
Therefore, for $\alpha_2$ small enough 
it is possible to find 
\[
\widehat{r}_j \in ((1- \sqrt{\alpha_2}) r_j, r_j) \quad \text{ such that }\gamma_j\pm \widehat{r}_j \notin I_{\mathrm{hor}}^\pm, \, \pm \widehat{r}_j \notin I_{\mathrm{ver}}^\pm,\] 
so the second condition in \eqref{2405231948} is satisfied. Then \eqref{2405231948} follows, since the first condition holds true for every $\widehat{r}_j$ except at most countably many.

Let\footnote{The intersection can be assumed to be nonempty, up to choosing $\alpha_2$ smaller if necessary.} 
\begin{equation}\label{1609231727}
\tdelta_j \in \Big(8\frac{\alpha_2}{\tilde\eta}, 16\frac{\alpha_2}{\tilde\eta}\Big) \cap \frac{\widehat{r}_j}{ r_j \N},
\end{equation}
so that the rectangles $Q^\pm_{i,\widehat{r}_j}$
are partitioned into cubes of sidelength $ \tdelta_j r_j$. Moreover, let $\widehat{k}_j \in \N$ be such that 
\begin{equation}\label{2701241920}
\tdelta_j \in  [2^{-(\widehat{k}_j+1)}, 2^{-\widehat{k}_j}).
\end{equation} \EEE
By property \eqref{aprop}, we have that
\begin{equation}\label{0106231042}
\Huno((Q_j \sm Q_{j, \widehat{r}_j}) \cap \Gamma_j)\leq 2 \alpha_1  \sqrt{\alpha_2} r_j.
\end{equation}
 We now subdivide $Q^\pm_{j,\widehat{r}_j}$ into cubes whose sidelength vanishes in a dyadic way towards the boundaries such that in any of them there is a small amount of jump of $\widehat{u}_j$ compared to the sidelength, in the sense of Proposition~\ref{prop:analoga32CFI}.

Let us assume, for simplicity of notation, that $x_j=0$ and $\nu(x_j)=e_2$. We introduce two sets $\Qcal_i^\pm$ of dyadic squares of sidelength $\widetilde{\delta}_k:= 2^{-k} \tdelta_j \EEE\, r_j$, $k \in \N$, which refine towards $\partial Q^\pm_{j,\widehat{r}_j}$, as follows: let $\Qcal^\pm_{j,0}$ be the family of squares $q \in \{ z + (0, \tdelta_j \EEE r_j]^2 \colon z \in \tdelta_j \EEE r_j \Z^2\}$, $q \subset Q^\pm_{j,\widehat{r}_j}$ such that $\dist(q, \partial Q^\pm_{j,\widehat{r}_j})> \tdelta_j \EEE r_j$; recursively, for $k\geq 1$, let $\Qcal^\pm_{j,k}$ be the family of squares $q \in \{ z+(0, \widetilde{\delta}_k]^2 \colon x \in \widetilde{\delta}_k \Z^2\}$, $q \subset Q^\pm_{j,\widehat{r}_j}$ such that $\dist(q, \partial Q^\pm_{j,\widehat{r}_j})> \widetilde{\delta}_k$ and $q$ does not intersect any cube in $\Qcal_{j,l}^\pm$, for $l<k$; we define
\[
\Qcal^\pm_j:=\bigcup_{k=0}^\infty \Qcal^\pm_{j,k}.
\]
For each $q \in \Qcal_j^\pm$ let $q'$ and $q''$ denote squares concentric with $q$ with sidelength 10\% and 20\% longer, respectively, so that 
$l(q')=\frac{11}{12} l(q'')$ and $l(q'')=\frac 6 5l(q)$; here and below, $l(\tilde{q})$ denotes the sidelength of a cube $\tilde{q}$.
By \eqref{sceltadieta}, \eqref{2405231232}, and \eqref{1609231727}, for any $q\in \Qcal_{j,0}^{\pm}$, we get that
\begin{equation}\label{quelli0}
\Huno(q''\cap \sal_{\widehat u_j})\le 8\alpha_2r_j\le \tilde\eta \tdelta_j \EEE r_j<\eta\frac{1}{40}l(q)=\eta\frac{1}{10}\frac{5}{6}\frac{l(q'')}{4}=\eta\Big(1-\frac{11}{12}\Big)\frac{l(q'')}{4}\,,
\end{equation}
so that all the squares $q''$ ``coming from'' squares $q\in \Qcal_{j,0}^{\pm}$ satisfy the hypotheses of Proposition~\ref{prop:analoga32CFI} for $s=\frac{11}{12}$.

Moreover, let $k\in\N$\,.
By \eqref{sceltadieta}, \eqref{2405231948}, \eqref{1609231727},  and \eqref{2701241920}, for any $q\in \Qcal_{j,k}^{\pm}$ we have (since $q''\subset \partial Q^\pm_{j,\widehat{r}_j}+ B_{2^{-(k+\widehat{k}_j-2)}r_j}(0)$) \EEE
\begin{equation}\label{quellik}
\Huno(q''\cap \sal_{\widehat u_j})\le\tilde\eta2^{-(k +\widehat{k}_j \EEE-1)} r_j< \frac{\eta}{40}  \widetilde{\delta}_k \EEE = \frac{\eta}{10}\frac{5}{6}\frac{l(q'')}{4}=\eta\Big(1-\frac{11}{12}\Big)\frac{l(q'')}{4}\,,
\end{equation}
so that all the squares $q''$ ``coming from'' squares $q\in \Qcal_{j,k}^{\pm}$ satisfy the hypotheses of Proposition~\ref{prop:analoga32CFI} for $s=\frac{11}{12}$\,.
By \eqref{quelli0} and \eqref{quellik} we thus deduce that all the squares $q''$ ``coming from'' squares $q\in \Qcal_{j}^{\pm}$ satisfy the hypotheses of Proposition~\ref{prop:analoga32CFI} for $s=\frac{11}{12}$.

Therefore, by Proposition~\ref{prop:analoga32CFI} applied to $\widehat{u}_j \in SBV^p(q'')$, for each $q \in \Qcal^\pm_j$ there is a set of finite perimeter $\omega_q \subset q''$, $\omega_q=\cup_{\mathcal{F}_q}B$ such that 
\begin{equation}\label{2405232347}
\begin{split}
w_q \in W^{1,p}(q'), & \quad  w_q=\widehat{u}_j \text{ in }q''\sm \omega_q\\
\Huno(\partial^* \omega_q) &\leq c/\eta \Huno(\sal_{\widehat{u}_j} \cap q''),\\
\int_{\omega_q} |\nabla w_q|^p \ud x &\leq c \int_{\omega_q} |\nabla \widehat{u}_j|^p \ud x.
\end{split}
\end{equation}
We define
\begin{equation*}
\omega_j^\pm:= \bigcup_{q \in  \Qcal_j^\pm} \omega_q.
\end{equation*}
Since the cubes $q''$ overlap at most 8 times, by the second and third property in \eqref{2405232347} we deduce that
\begin{equation}\label{2905232235}
\begin{split}
\Huno(\partial^* \omega_j^\pm) &\leq 8c/\eta \, \Huno(\sal_{\widehat{u}_j} \cap Q_{j, \widehat{r}_j}^\pm),\\
 \int_{\omega_j^\pm} |\nabla w_q|^p \ud x & \leq 8c \int_{\omega_j^\pm} |\nabla \widehat{u}_j|^p \ud x \EEE
\end{split}
\end{equation}
Following \cite[proof of Theorem~4.1]{CCS22}, we construct  regularized functions $v_j^\pm$ on $Q_{j, \widehat{r}_j}^\pm$ which are convex combinations of the functions $w_q$. We notice that in our setting all the cubes are ``good'', according to the definition in \cite{CCS22}, that is the jump inside has small $\Huno$-measure compared to the sidelength.

We set
\begin{equation}\label{2905232134}
v_j^\pm:= \sum_{q \in \Qcal_j^\pm} w_q \varphi_q,
\end{equation}
where
\begin{equation*}
\begin{split}
\varphi_q:= & \frac{\psi_q}{\sum_{\widehat{q} \in \Qcal_j^\pm} \psi_{\widehat{q}}},  \quad \psi_q(x):= \psi\Big( \frac{x-c_q}{l(q)} \Big) \text{ for }q=c_q+\Big(-\frac{l(q)}{2}, \frac{l(q)}{2}\Big)^2, \\
&  \psi \in C_c^\infty\Big( (-11/20, 11/20)^2; [0,1] \Big),\, \psi=1 \text{ on } [-1/2,1/2]^2. 
\end{split}
\end{equation*}
By construction, $\psi_q\in \Cc^{\infty}(q';[0,1])$ and $\psi_q\equiv 1$ in $q$, for any $q\in \Qcal_{j}^{\pm}$.
Since, by \eqref{2405232347},  $w_q \in W^{1,p}(q')$ for every $q \in \Qcal_j^\pm$, we deduce that
\begin{equation}\label{0106231155}
v_j^\pm \in W^{1,p}\Big(\bigcup_{q\in\Qcal_j^\pm}q'\Big).
\end{equation} 
Eventually, we define
\begin{equation}\label{2905232228}
v_j:= v_j^+ \chi_{Q^+_{j,\widehat{r}_j}}+v_j^- \chi_{Q^-_{j,\widehat{r}_j}}, \quad \omega_j:= \omega^+_j \cup \omega^-_j.
\end{equation}
By \eqref{2405231232}, \eqref{2405232347}, \eqref{2905232235}, \eqref{2905232134}, \eqref{2905232228} it follows that 
\begin{equation}\label{2905232231}
\begin{split}
\int_{\omega_j} & |\nabla v_j|^p \ud x  \leq    8c \EEE  \int_{\omega_j} |\nabla u|^p \ud x, \quad v_j=\widehat{u}_j \text{ in }Q_{j, \widehat{r}_j} \sm \omega_j, \\
& \Huno(\partial^* \omega_j) \leq 8\, c/\eta \Huno\big(Q_{j, \widehat{r}_j}\cap (\sal_{\widehat{u}_j} \sm \tsal_j^{\gamma_j})\big) \leq  64 \, c/\eta \, \, \alpha_2  r_j.
\end{split}
\end{equation}
We observe that the first estimate above is obtained arguing as in \cite[Step~3.3 in Theorem~5.1]{CCS22} with the full gradient in place of the symmetrized gradient. 

Furthermore, by construction we have that
\begin{equation}\label{2905232239}
v_j=\widehat{u}_j \text{ on } \partial Q_{j, \widehat{r}_j}, \quad [v_j]=[\widehat{u}_j] \text{ on }T_j^{\gamma_j}.
\end{equation}
We observe that the latter property above follows from the fact that we employed a Whitney-type approximation towards $T_j^{\gamma_j}$.
In view of \eqref{0106230928},
\begin{equation}\label{0106230939}
\Huno(\partial Q_{j, \widehat{r}_j} \cap \{v_j \neq u\})= \Huno(\Sigma_{j,\widehat{r}_j}^1 \cap \Sigma_{j,\widehat r_j}^2)\leq 2 \alpha_1 \widehat{r}_j +  \frac{322 \alpha_2}{\tilde{\eta}}r_j\EEE,
\end{equation}
and, by property \eqref{bprop} and the definition of $T_j^{\gamma_j}$,
\begin{equation}\label{0106230941}
|\Huno( Q_{j, \widehat{r}_j} \cap \Gamma_j) - \Huno(Q_{j, \widehat{r}_j} \cap T_j^{\gamma_j})|\leq 2 \alpha_2 \widehat{r}_j.
\end{equation}

\medskip
\noindent \textit{Step 2.2: the case $z_j=0$}. In this case, $x_j \in \sal^\fff_u$, and as done before, we find radii $\widehat{r}_j \in ((1-\sqrt{\alpha_2}) r_j, r_j)$  such that, 
denoting $Q_{j,s}:=x_j + Q_s$,   it holds that 
\begin{equation}\label{2405231948'}
\begin{split}
& \Huno(\sal_{\widehat{u}_j} \cap \partial Q_{j,\widehat{r}_j})=0, \\
& \Huno\Big(\sal_{\widehat{u}_j} \cap Q_{j, \widehat{r}_j} \cap \big(Q_{j, \widehat{r}_j}+ B_{2^{-k}r_j}(0)\big)\Big) < \tilde{\eta}2^{-(k+1)}r_j \text{ for every }k \in \N. 
\end{split}
\end{equation}
By this choice, we can slightly amend the construction in \cite[Theorem~4.1]{CCS22}, to find in both
\[
Q^\pm_{j,\widehat{r}_j} \quad \text{connected components of } Q_{j,\widehat{r}_j}\sm \Gamma_j
\]
two sets of finite perimeter $\omega_j^\pm$ and functions $v_j^\pm \in W^{1,p}(Q^\pm_{j,\widehat{r}_j})$ such that, for suitable $c_j^\pm=c_j^\pm(p)>0$,
\begin{equation}\label{3005231017}
\begin{split}
& v_j^\pm=  u \text{ in }Q^\pm_{j,\widehat{r}_j}  \sm \omega_j^\pm, \quad \int_{\omega_j^\pm} |\nabla v_j^\pm|^p \ud x \leq c_j^\pm \int_{\omega_j^\pm} |\nabla u|^p \ud x, \\
& \Huno(\partial^* \omega_j^\pm) \leq c_j^\pm \Huno(\sal_u \cap Q^\pm_{j,\widehat{r}_j}), \quad \Huno(\partial Q_{j,\widehat{r}_j} \cap (\omega_j^+\cup \omega_j^-))=0.
\end{split}
\end{equation}
We notice that the last condition is new with respect to \cite[Theorem~4.1]{CCS22}: it comes from the Whitney-type construction as in the previous substep, in turn allowed by the choice of $\widehat{r}_j$ for which \eqref{2405231948} holds, which is possible in 2d.
Moreover, as in \cite[proof of Theorem~5.1, Step 2.2]{CCS22}, one proves that the constant 
\begin{equation}\label{2302241844}
\tilde{c}:=\max\{c_j^\pm\colon j \text{ s.t.\ }x_j \in \sal_u^\fff\}
\end{equation} is bounded uniformly with respect to $\alpha_2$ 
(in particular, even if the sidelenghts of cubes decrease and the number of cubes increases; notice that increasing the number of cubes one may assume that the Lipschitz constant corresponding to $\Gamma_j$ decreases).
As above, we set 
\begin{equation*}\label{3005231930}
v_j:= v_j^+ \chi_{Q^+_{j,\widehat{r}_j}}+v_j^- \chi_{Q^-_{j,\widehat{r}_j}}, \quad \omega_j:= \omega^+_j \cup \omega^-_j.
\end{equation*}
\\
\medskip
\textit{Step 3: conclusion}.
Following the lines of \cite[proof of Theorem~5.1, Step 3]{CCS22}, let us consider $\delta \in (0, 0.4\, \sqrt{2} \alpha_2 \min_{j=1, \dots, J}r_j)$ and the families:
\begin{equation*}
\begin{split}
&\Qscr_1:=\bigg\{q_{z,\delta}=\delta z +[0,\delta]^2 \colon z \in \mathbb{Z}^2,\, q_{z,\delta} \cap \Big(\R^2 \sm \bigcup_{j=1}^J Q_{j, \widehat{r}_j}\Big) \neq \emptyset\bigg\},\\
& \Qscr_2:=\bigg\{q_{z,\delta}=\delta z +[0,\delta]^2 \colon z \in \mathbb{Z}^2,\, q_{z,\delta} \notin\Qscr_1 \text{ and intersects some cubes in }\Qscr_1\bigg\},\\
& \Qscr:=\Qscr_1 \cup \Qscr_2.
\end{split}
\end{equation*}
For each $q \in \Qscr$, let $q'$ and $q''$ be the (closed) cubes concentric with $q$ and having side length  $l(q')=\frac98 \delta$ and $l(q'')=\frac{10}{8}\delta=\frac{10}{9}l(q')$, respectively.
Let
\begin{equation}\label{3005231926}
\widehat{v}(x):=
\begin{dcases}
v_j(x), &\quad x \in Q_{j, \widehat{r}_j},\\
u(x), &\quad x \in \Omega \sm \bigcup_{j=1, \dots, J} Q_{i, \widehat{r}_j},
\end{dcases}
\end{equation}
and, recalling the definition of $c$, $\eta$ from Proposition~\ref{prop:analoga32CFI}, set
\begin{equation*}
\Qscr_g:=\Big\{ q \in \Qscr\, \colon\, \Huno(\sal_{\widehat{v}} \cap q'')\leq {\frac{1}{32}} \eta \delta  \Big\}{=\Big\{ q \in \Qscr\, \colon\, \Huno(\sal_{\widehat{v}} \cap q'')\leq \eta\Big(1-\frac{9}{10}\Big)\frac{l(q'')}{4}  \Big\}}, \quad \Qscr_b:=\Qscr \sm \Qscr_g.
\end{equation*}
For every $q \in \Qscr_g$, by Proposition~\ref{prop:analoga32CFI} applied to $\widehat{v}\in SBV^p(q'')$ (in correspondence to $s=0.9$) there exist $w_q  \in SBV^p(q'')$ and $\omega_q \subset q''$, $\omega_q=\cup_{\mathcal{F}_q}B$ such that 
\begin{equation}\label{2405232347'}
\begin{split}
w_q \in W^{1,p}(q'), & \quad  w_q=\widehat{v} \text{ in }q''\sm \omega_q\\
\Huno(\partial^* \omega_q) &\leq c/\eta \Huno(\sal_{\widehat{v}} \cap q''),\\
\int_{\omega_q} |\nabla w_q|^p \ud x &\leq c \int_{\omega_q} |\nabla {\widehat{v}}|^p \ud x.
\end{split}
\end{equation}
 Up to reducing the threshold in the definition of $\Qscr_g$, it holds that if $q' \cap \Gamma_j \neq \emptyset$ for some $j=1, \dots, J$, then $q \notin \Qscr_g$, so that if $q \in \Qscr_g$ is such that $q' \subset Q_{j, \widehat{r}_j}$ it holds that $q' \subset Q^\pm_{j, \widehat{r}_j}$ and then $w_q=\widehat{v}$ (and $\omega_q \cap q'=\emptyset$), since $\widehat{v}=v_j \in W^{1,p}(Q^\pm_{j, \widehat{r}_j})$.

We set, recalling \eqref{2905232228} (and the analogue for $j$ s.t.\ $x_j \in \sal_u^\fff$)
\begin{equation*}
G:=\bigcup_{q \in \Qscr_g} q, \quad \widetilde{\omega}:=\bigcup_{q \in \Qscr_b} q,
\end{equation*}
\begin{equation*}
\Gamma^\iii:= \bigcup_{j \colon x_j \notin \sal_u^\fff} (Q_{j, \widehat{r}_j} \cap T_j^{\gamma_j}), \quad \Gamma^\fff:=\bigcup_{j \colon x_j \in \sal_u^\fff} (Q_{j, \widehat{r}_j} \cap \Gamma_j), \qquad \widehat{\omega}:= \bigcup_{q \in \Qscr_g} \omega_q \cup \bigcup_{j=1, \dots, J} \omega_j,
\end{equation*}
and
\begin{equation}\label{3005232001}
v:=
\begin{dcases}
 \sum_{q \in \Qscr_g} w_q \varphi_q, &\quad\text{in }G,\\
 0, &\quad\text{in }\widetilde{\omega},\\
 \widehat{v}, &\quad\text{in }\Omega\sm (G\cup \widetilde{\omega}),
 \end{dcases}
\end{equation}
where
\begin{equation*}
\begin{split}
\varphi_q:= & \frac{\psi_q}{\sum_{\widehat{q} \in \Qcal_g} \psi_{\widehat{q}}},  \quad \psi_q(x):= \psi\Big( \frac{x-c_q}{l(q)} \Big) \text{ for }q=c_q+\Big(-\frac{l(q)}{2}, \frac{l(q)}{2}\Big)^2, \\
&  \psi \in \Cc^\infty\Big( (-9/16, 9/16)^2; [0,1] \Big),\, \psi=1 \text{ on } [-1/2,1/2]^2. 
\end{split}
\end{equation*}
By triangle inequality, \eqref{dprop}, \eqref{bprop}, \eqref{2304232000'} using that $0\le r_j-\widehat r_j\le \sqrt{\alpha_2}r_j$ and that the cubes $Q_j$ are pairwise disjoint, we obtain
\begin{equation}\label{persaltoint}
\begin{aligned}
|\Huno(\sal_u^{\iii})-\Huno(\Gamma^{\iii})|\le&\,\Huno(\sal_u^{\iii}\setminus\widehat\Gamma^{\iii})+|\Huno(\sal_u^{\iii}\cap\widehat\Gamma^{\iii})-\Huno(\Gamma^{\iii})|\\
\le&\,\Huno(\sal_u^{\iii}\triangle\widehat\Gamma^{\iii})+\sum_{j:x_j\notin\sal_u^{\fff}}|\Huno(\sal_u^{\iii}\cap\Gamma_j)-\Huno(T_{j}^{\gamma_j}\cap Q_{j,\widehat r_j})|\\
\le&\,\Huno(\sal_u^{\iii}\triangle\widehat\Gamma^{\iii})+\sum_{j:x_j\notin\sal_u^{\fff}}\Huno(\Gamma_j\setminus\sal_u^{\iii})
\\
&\,\phantom{\Huno(\sal_u^{\iii}\triangle\widehat\Gamma^{\iii})}
+\sum_{j:x_j\notin\sal_u^{\fff}}|\Huno(\Gamma_j)-2\widehat{r}_j
|\\
\le&\, \Huno(\sal_u^{\iii}\triangle\widehat\Gamma^{\iii})+\alpha_2(1+5\Huno(\sal_u))+(\alpha_2+\sqrt{\alpha_2})\Huno(\sal_u).
\end{aligned}
\end{equation}
Furthermore, by construction,
\begin{equation*}
\begin{aligned}
\sal_u^{\fff}\triangle \Gamma^{\fff}=&\,\Big(\sal_u^{\fff}\setminus\bigcup_{j:x_j\in\sal_u^{\fff}}\ol Q_j\Big)\cup\bigcup_{j:x_j\in\sal_u^{\fff}}\big(\ol Q_j\cap(\sal_u^{\fff}\triangle(\Gamma_j\cap Q_{j,\widehat r_j}) )\big)\\
\subseteq&\, \Big(\sal_u^{\fff}\setminus\bigcup_{j:x_j\in\sal_u^{\fff}}\ol Q_j\Big)\cup\bigcup_{j:x_j\in\sal_u^{\fff}}\big(\ol Q_j\cap(\sal_u^{\fff}\triangle\Gamma_j)\big)\\
&\,\cup\bigcup_{j=1}^{J} \big(\sal_u\cap(\ol Q_j\sm Q_{j,\widehat r_j})\big)\\
=&\, \big(\sal_u^{\fff}\triangle \widehat\Gamma^{\fff}\big)\cup \bigcup_{j=1}^{J} \big(\sal_u\cap(\ol Q_j\sm Q_{j,\widehat r_j})\big)\,,
\end{aligned}
\end{equation*}
whence, using \eqref{0106231042} and property \eqref{dprop}, we deduce that 
\begin{equation}\label{persaltofraz}
\begin{aligned}
\Huno(\sal_u^\fff \triangle \Gamma^\fff)\le&\,\Huno(\sal_u^{\fff}\triangle \widehat\Gamma^{\fff})
+\sum_{j=1}^J \Huno(\sal_u \cap (Q_j \sm Q_{j, \widehat{r}_j})) \\
\leq&\,\Huno(\sal_u^\fff \triangle \widehat\Gamma^\fff)+ \sum_{j=1}^J \Huno(\Gamma_j \cap (Q_j \sm Q_{j, \widehat{r}_j})) + \Huno\Big(\sal_u \triangle \bigcup_{j=1}^J  \Gamma_j\Big) 
\\&
\leq\, \Huno(\sal_u^\fff \triangle \widehat\Gamma^\fff)+ 2\alpha_1 \sqrt{\alpha_2} \Huno(\sal_u) + \alpha_2(1+5\Huno(\sal_ u)).
\end{aligned}
\end{equation}
By summing \eqref{persaltoint} and \eqref{persaltofraz}, using \eqref{2404230943}, we obtain
\begin{equation*}
|\Huno(\sal_u^{\iii})-\Huno(\Gamma^{\iii})|+\Huno(\sal_u^\fff \triangle \Gamma^\fff)\le \alpha_2(3+16\Huno(\sal_u))+\sqrt{\alpha_2}(1+2\alpha_1)\Huno(\sal_u).
\end{equation*}
By \eqref{0106231130} and \eqref{2905232239} it follows that $[v](x) \in \Z$ for $\Huno$-a.e.\ $x \in \Gamma^\iii$.

By definition and  \eqref{0106231155}, \eqref{3005231017}, \eqref{2405232347'} it is immediate that $v \in SBV^p(\Omega) \cap W^{1,p}(\Omega \sm (\Gamma \cup \overline{\widetilde{\omega}}))$, that $\{\nabla u\neq \nabla v\} \subset \widetilde{\omega}\cup \widehat{\omega}$ (since $\nabla(\widehat{u}_j-u)=0$ \EEE in $Q_j$, see \eqref{2405231920}), that $\{u\neq v\}\subset \{\nabla u\neq \nabla v\} \cup \bigcup_{j \colon z_j\neq0} H_j$ (whose $\Ld$-measure vanishes with $\alpha_2$ and $\varrho$ from property \eqref{eprop}), and that $v=0$ in $\widetilde{\omega}$.

Summing up \eqref{2405232347'} over $q \in \Qscr_g$ we obtain (since the cubes $q''$ may overlap at most 8 times)
\begin{equation}\label{0106231216}
\Huno(\partial^* \bigcup_{q \in \Qscr_g} \omega_q) \leq 8(c\vee \tilde{c})/\eta \, \Huno\Big(\sal_{\widehat{v}}\sm \bigcup_{j=1}^{J} Q_{j, \widehat{r}_j}\Big) \leq C(\alpha_1, \alpha_2, p)
\end{equation}
($\tilde{c}$ is the constant in \eqref{2302241844}) \EEE with $C(\alpha_1,\alpha_2, p)$ vanishes with $\alpha_2$ (for $\alpha_1 \leq 1/4$), since
\begin{equation*}
\sal_{\widehat{v}}\sm \bigcup_{j=1,\dots, J} Q_{j, \widehat{r}_j} \subset \Big(\sal_u \sm \bigcup_{j=1,\dots, J} \ol Q_{j}\Big) \cup \bigcup_{j\colon x_j \notin \sal_u^\fff} (\Sigma_{j, \widehat{r}_j}^1 \cup \Sigma_{j, \widehat{r}_j}^2)  \cup \bigcup_{j \colon x_j \notin \sal_u^\fff} ((Q_j \sm Q_{j, \widehat{r}_j}) \cap (\Gamma_j \cup T_j^{\gamma_j}) 
\end{equation*}
and from the properties of $\Gamma_j$, \eqref{0106230928}, \eqref{0106231042},  \eqref{0106230939}, \eqref{0106230941}. 
Therefore, adding the estimates of the $\Huno$-measures of $\partial^*\omega_j$ in \eqref{2905232231} over $j$  such that $x_j \notin \sal_u^\fff$   plus $\partial^* \omega_j^\pm$ in \eqref{3005231017} over  $j$ such that $x_j \in \sal_u^\fff$  together with \eqref{0106231216}, we conclude that  $\Huno(\partial^*\widehat{\omega})$  vanishes with $\alpha_2$.
In view of the definition of  $\widetilde{\omega}$  (in particular of $\Qscr_b$) we get
\begin{equation*}
\Huno(\partial  \widetilde{\omega} ) \leq \frac{40}{9} \eta \, \Huno\Big(\sal_{\widehat{v}}\sm \bigcup_{j=1,\dots, J} Q_{j, \widehat{r}_j}\Big),
\end{equation*}
where above a factor 8 accounts for the overlapping of squares $q''$; as well, $\Huno(\partial \widetilde{\omega})$ vanishes with $\alpha_2$ by \eqref{0106231216}.

Eventually, arguing as in \cite[Step~3.1 in Theorem~5.1]{CCS22} for the cubes $Q_j$ such that $x_j \in \sal_u^\fff$ one proves that $\Huno(\Gamma^\fff \cap \{ v^\pm\neq u^\pm\})$ vanishes with $\alpha_2$, while (again following \cite[Step~3.3 in Theorem~5.1]{CCS22} with the full gradient in place of the symmetrized gradient) 
one \MMM deduces from the last estimate in \eqref{2405232347'} that
\[
\int_{\bigcup_{q \in \Qscr_g} \omega_q} |\nabla v|^p \ud x \leq 8c \EEE \int_{\bigcup_{q \in \Qscr_g} \omega_q} |\nabla \widehat{v}|^p \ud x,
\]
which together with the estimates on the gradients in \eqref{2905232231} and \eqref{3005231017} gives that
\begin{equation*}
\begin{split}
\int_{\widehat{\omega}} |\nabla v|^p \ud x  \leq C'(\alpha_2, p) \int_{\widehat{\omega}}  |\nabla u|^p \ud x,
\end{split}
\end{equation*}
for $C'(\alpha_2, p)$ a positive constant vanishing with $\alpha_2$.
Being $\{\nabla u\neq \nabla v\} \subset \widetilde{\omega}\cup \widehat{\omega}$, $v=0$ in $\widetilde{\omega}$, and since the measure of $\widehat{\omega}$ vanishes with $\alpha_2$, we obtain
\begin{equation*}
\begin{split}
\int_{\Omega} |\nabla v|^p \ud x  \leq (1+C''(\alpha_2, \varrho, p))  \int_\Omega |\nabla u|^p \ud x,
\end{split}
\end{equation*}
where $C''(\alpha_2, p)>0$  vanishes with $\alpha_2$.

\MMM We conclude since $\alpha_1$, $\alpha_2$, $\varrho$ may be fixed arbitrarily small.
\end{proof}
We recall a fundamental technical tool, \cite[Proposition~3.2]{CFI18}. In \cite{CFI18} the result is stated for balls, and it holds for cubes as well. Moreover, it holds true also for $GSBD^p$ functions, in place of $SBD^p$ (see e.g.\ \cite[Proof of Proposition~3.1]{CCI19}). As usual, 
 {$Q_\rho:=(-\rho,\rho)^2$.}
\begin{proposition}\label{prop:32CFI}
For every $p \in (1,\infty)$ there exist $c>0$ and  $\eta\in (0,1)$ such that if $u \in GSBD^p(Q_{2\varrho})$, $\varrho>0$, satisfies
\[
\Huno(\sal_u \cap Q_{2\varrho})< \eta (1-s) \varrho
\] 
for  some $s \in (0,1)$, then there is a countable family $\mathscr{F}=\{B\}$ of closed balls of radius $r_B < 2(1-s) \varrho$ and center $x_B \in \ol Q_{2s{\varrho}}$ such that their union is compactly contained in $B_{2\varrho}$, and a field $w \in SBD^p(Q_{2\varrho})$ such that
\begin{itemize}
\item[(i)] $\varrho^{-1}\sum_{B\in\mathscr{F}} \Ld(B) + \sum_{B\in\mathscr{F}} \Huno(\partial B) \leq c/\eta \,\, \Huno(\sal_u \cap Q_{2\varrho})$;
\item[(ii)] $\Huno (\sal_u \cap \cup_{B\in\mathscr{F}}\partial B)= \Huno\big((\sal_u \cap Q_{2s\varrho}) \sm \cup_{B\in\mathscr{F}}B  \big)=0$;
\item[(iii)] $w=u$ $\Ld$-a.e.\ on $Q_{2 \varrho}\sm \cup_{B\in\mathscr{F}} B$;
\item[(iv)] $w \in W^{1,p}(Q_{2s\varrho}; \R^2)$ and $\Huno(S_w \sm \sal_u)=0$;
\item[(v)]
\begin{equation}\label{2405231627'}
\int_{\cup_{B\in\mathscr{F}}B} |e(w)|^p \ud x \leq c \int_{\cup_{B\in\mathscr{F}}B} |e(u)|^p \ud x .
\end{equation}
\end{itemize}
\end{proposition}
The previous result may be directly modified to obtain a $SBV^p$ version. 
\begin{proposition}\label{prop:analoga32CFI}
For every $p \in (1,\infty)$ there exist $c>0$ and $\eta\in (0,1)$ such that if $u \in SBV^p(Q_{2\varrho})$, $\varrho>0$, satisfies
\[
\Huno(\sal_u \cap Q_{2\varrho})< \eta (1-s) \varrho
\] 
for some $s \in (0,1)$, then there is a countable family $\mathscr{F}=\{B\}$ of closed balls of radius $r_B < 2(1-s) \varrho$ and center $x_B \in \ol Q_{2s{\varrho}}$ such that their union is compactly contained in $Q_{2\varrho}$, and a field $w \in SBV^p(Q_{2\varrho})$ such that
\begin{itemize}
\item[(i)] $\varrho^{-1}\sum_{B\in\mathscr{F}} \Ld(B) + \sum_{B\in\mathscr{F}} \Huno(\partial B) \leq c/\eta \,\, \Huno(\sal_u \cap Q_{2\varrho})$;
\item[(ii)] $\Huno (\sal_u \cap \cup_{B\in\mathscr{F}}\partial B)= \Huno\big((\sal_u \cap Q_{2s\varrho}) \sm \cup_{\mathscr{F}}B  \big)=0$;
\item[(iii)] $w=u$ $\Ld$-a.e.\ on $Q_{2 \varrho}\sm \cup_{B\in\mathscr{F}} B$;
\item[(iv)] $w \in W^{1,p}(Q_{2s\varrho})$ and $\Huno(\sal_w \sm \sal_u)=0$;
\item[(v)]
\begin{equation}\label{2405231627}
\int_{\cup_{B\in\mathscr{F}}B} |\nabla w|^p \ud x \leq c \int_{\cup_{B\in\mathscr{F}}B} |\nabla u|^p \ud x. 
\end{equation}
\end{itemize}
\end{proposition}
\begin{proof}
We notice that it is enough to follow the proof of \cite[Theorem~2.1]{CFI18}, from which \cite[Proposition~3.2]{CFI18} follows, and use the fact that, if $u \in SBV^p$, one can control the components of the constant matrix $\nabla \phi(u)$ in place of those of $e(\phi(u))$ (see (2.12) in \cite{CFI18} and its consequences) by
\[
\nabla \phi(u) \cdot (x-y) = \phi(u)(x)- \phi(u)(y)= \int_{S_{x,y}} (u^\nu_z)'(t) \, \mathrm{d}t, 
\] 
where $u^\nu_z(t):=u(z+t \nu)$, for $\nu:= \frac{x-y}{|x-y|}$, $z:=(\mathrm{Id}-\nu \otimes \nu)x$. Moreover, a constant in place of an infinitesimal rigid motion appears in the Poincar\'e's inequality for $u$ on $Q_{\ol x, \ol y}$.
\end{proof}
By arguing as in the proof of Theorem \ref{thm:densitySBVaux},  using Proposition~\ref{prop:32CFI} in place of Proposition~\ref{prop:analoga32CFI}, \EEE one can show that also the following result holds true.
\begin{theorem}\label{thm:densityGSBD}
Let $\Omega\subset \R^2$ be a bounded open set of finite perimeter, $p\in (1,+\infty)$, $u \in GSBD^p(\Omega)$, and $\varepsilon>0$. Then there exist:
\begin{itemize}
\item closed sets $\Gamma^\iii$, $\Gamma^\fff$, finite unions of disjoint $C^1$ curves; 
\item a set $\tilde{\omega}$, finite union of cubes;
\item a set of finite perimeter $\widehat{\omega}$;
\item a function $v \in GSBD^p(\Omega) \cap W^{1,p}(\Omega \sm (\Gamma \cup \overline{\tilde{\omega}}); \R^2)$, where $\Gamma:=\Gamma^i \cup \Gamma^\fff$;
\end{itemize}
such that $\{\nabla u\neq \nabla v\} \subset  \tilde{\omega}\cup \widehat{\omega}$, $\Ld(\{u\neq v\})<\varepsilon$, $v=0$ in $\tilde{\omega}$, $[v](x)\in \Z^2$ for $\Huno$-a.e.\ $x\in \Gamma^\iii$, and
\begin{equation*}
\begin{split}
|\Huno(S^\iii_u)-\Huno(\Gamma^\iii)| + \Huno(S^\fff_u \triangle \Gamma^\fff)+\Huno(\partial \tilde{\omega}) + \Huno(\partial^* \widehat{\omega}) \leq \varepsilon, \quad \int_{\Omega\sm \tilde{\omega}} |e(v)|^p \ud x  \leq (1+\varepsilon)  \int_\Omega |e(u)|^p \ud x,
\end{split}
\end{equation*}
where $S^\fff_u:=\{x \in \sal_u \colon [u] \notin \Z^2\}$.
Moreover, $\Huno(\Gamma \cap \{v^\pm \neq u^\pm\}) \leq \varepsilon$, where $v^\pm$ and $u^\pm$ denote the traces of $v$ and $u$ on the two sides of $\Gamma$, and, if $u\in SBD^p(\Omega)$, then also $v\in SBD^p(\Omega)$.
\end{theorem}

\section{Description of the problem}\label{sec:model}
Let $\Omega$ be a bounded and open subset of $\R^2$ with Lipschitz continuous boundary and let $\Omega'\subset\subset\Omega$ be an open set. We introduce
\begin{equation}\label{admissiblebdrysenzaw}
	\mathcal {AD}(\Om,\Omega'):=\{u\in SBV^2(\Om;\Ss^1)\,:\,\overline{S}_u\subset\overline{\Omega'}\},
\end{equation}
where 
$\sal_u$  denotes the jump set of $u$. 
For every $\ep>0$\,, let $\mathcal G_\ep:{SBV^2(\Omega;\Ss^1)}\to [0,\infty]$ be the functional defined by 
\begin{equation}\label{defEne}
	\mathcal G_\ep(u):=\left\{\begin{array}{ll}
		\displaystyle \int_{\Omega}\frac 1 2 |\nabla u|^2\ud x+\frac 1 \ep\Huno(\overline{S}_u)&\textrm{if }u\in \mathcal{AD}(\Omega,\Omega')\\
		+\infty&\textrm{elsewhere in }{SBV^2(\Omega;\Ss^1)}\,.
	\end{array}
	\right.
\end{equation}
{In what follows, we will adopt also localized versions of the functional  $\mathcal G_\ep $;  more precisely, for any $u\in \mathcal{AD}(\Omega,\Omega')$ and for any open set $A$ with  $\Omega'\subset\subset A\subset\subset\Omega$\,, we will denote by $\mathcal G_\ep(u;A)$ the functional in \eqref{defEne} with $\Omega$ replaced by $A$\,.
}

\EEE
Notice that, since  $u\in H^1(\Om\setminus \overline\Om';\mathbb S^1)$, it follows that 
\begin{align}
	\supp{Ju}\subseteq  \overline\Om'\qquad\textrm{for every } u\in \mathcal{AD}(\Om,\Om').\label{supporto_Ju}
\end{align} 
Indeed, let $\varphi\in \Cc^\infty(\Om\setminus \overline\Om')$, and write 
\begin{align*}
	\langle Ju,\varphi\rangle_\Om&=\frac{1}{2}\int_\Om \frac{\partial\varphi}{\partial x_{2}}\ud([u^1\mad_1u^{2}]-[u^{2}\mad_1u^{1}])-\frac{1}{2}\int_\Om\frac{\partial\varphi}{\partial x_{1}}\ud([u^1\mad_{2}u^{2}]-[u^{2}\mad_{2}u^{1}])\nonumber\\
	&=\frac{1}{2}\int_{\Om\setminus \overline\Om'} \frac{\partial\varphi}{\partial x_{2}}\Big(u^1\frac{\partial u^2}{\partial x_1}-u^2\frac{\partial u^1}{\partial x_1}\Big)\ud x-\frac{1}{2}\int_{\Om\setminus \overline\Om'}\frac{\partial\varphi}{\partial x_{1}}\Big(u^1\frac{\partial u^2}{\partial x_2}-u^2\frac{\partial u^1}{\partial x_2}\Big)\ud x\nonumber\\
	&=\langle Ju,\varphi\rangle_{\Om\setminus \overline\Om'}=0,
\end{align*}
where the last equality follows since $u\in H^1(\Om\setminus \overline\Om';\mathbb S^1)$ has null distributional Jacobian determinant in $\Om\setminus \overline\Om'$.

\subsection{$\Gamma$-convergence in the subcritical regime}
We introduce the class of atomic measures, namely
$$X(\Omega):=\Big\{\mu\in \mathcal M(\Om)\,:\,\mu=\sum_{n=1}^Nz^n\delta_{x^n},\;x^n\in \Om,\;z^n\in \mathbb Z\setminus\{0\},\; N\in \mathbb N\Big\}.$$
In \cite[Theorem 3.1]{DLSVG}, the authors show that the rescaled functional $|\log\ep|^{-1}\mathcal G_\ep$ $\Gamma$-converges to the functional 
$\F:X(\Om)\rightarrow \R^+$ defined as $\F(\mu)=\pi |\mu|(\Om)$.
Using a density argument, and in particular Proposition \ref{cor:21041335}, this result can be easily extended to the following setting, where the energy functional does not take into account of the closure of the jump set: We introduce
\begin{equation}\label{defEneG}
	\mathcal F_\ep(u):=\left\{\begin{array}{ll}
		\displaystyle \int_{\Omega}\frac 1 2 |\nabla u|^2\ud x+\frac 1 \ep\Huno({S}_u)&\textrm{if }u\in \mathcal{AD}(\Omega,\Omega')\\
		+\infty&\textrm{elsewhere in }{SBV^2(\Omega;\Ss^1)}\,.
	\end{array}
	\right.
\end{equation}
Then 
the following $\Gamma$-convergence result holds:
\begin{theorem}\label{mainthmVECCHIO}
	Let $\Om$ and $\Om'$ be as above; then there hold
	\begin{itemize}
		\item[(i)] (Compactness) Let $\{u_\ep\}_{\ep}\subset {SBV^2(\Omega;\Ss^1)}$ be such that 
		\begin{equation}\label{enboundVECCHIO}
			\sup_{\ep>0}\frac{ \mathcal F_\ep(u_\varepsilon)}{|\log\ep|}\le C,
		\end{equation}
		for some $C>0$\,. Then there exists $\mu\in X(\Omega)$ with $\supp\mu\subseteq\overline\Om'$ such that, up to a subsequence, $\|Ju_\ep-\pi\mu\|_{\flt,\Omega}\to 0$ (as $\ep\to 0$).
		\item[(ii)] ($\Gamma$-liminf inequality) For every $\mu\in X(\Omega)$ with $\supp\mu\subseteq\overline\Om'$ and for every $\{u_\ep\}_{\ep}\subset {SBV^2(\Omega;\Ss^1)}$ such that $\|Ju_\ep-\pi\mu\|_{\flt,\Omega}\to 0$ (as $\ep\to 0$)\,, it holds
		\begin{equation}\label{liminfVECCHIO}
			\pi|\mu|(\Omega)\le \liminf_{\ep\to 0}\frac{ \mathcal F_\ep(u_\varepsilon)}{|\log\ep|}\,.
		\end{equation}
		\item[(iii)] ($\Gamma$-limsup inequality) For every $\mu\in X(\Omega)$ with $\supp\mu\subseteq\overline\Om'$, there exists  $\{u_\ep\}_{\ep}\subset  {SBV^2(\Omega;\Ss^1)}$ with $\|Ju_\ep-\pi\mu\|_{\flt,\Omega}\to 0$ (as $\ep\to 0$)\,, such that
		\begin{equation}\label{eq:limsupVECCHIO}
			\pi|\mu|(\Omega)\geq\limsup_{\ep\to 0}\frac{\mathcal F_\ep(u_\varepsilon)}{|\log\ep|}\,.
		\end{equation}
	\end{itemize}
\end{theorem}
Actually, by arguing as above and going through the proof of \cite[Theorem 3.1]{DLSVG}, one can prove the following more general result.
\begin{theorem}\label{mainthmVECCHIOquasi}
	Let $\Om$ and $\Om'$ be as above; and let $\{E_\ep\}_\ep\subset(0,+\infty)$ with $c|\log\ep|\le E_\ep\ll |\log\ep|^2$ for some constant $c>0$ (independent of $\ep$). Then the following $\Gamma$-convergence result holds true.
	\begin{itemize}
		\item[(i)] (Compactness) Let $\{u_\ep\}_{\ep}\subset {SBV^2(\Omega;\Ss^1)}$ be such that 
		\begin{equation*}
			\sup_{\ep>0}\frac{ \mathcal F_\ep(u_\varepsilon)}{E_\ep}\le C,
		\end{equation*}
		for some $C>0$. Then there exists $\mu\in X(\Omega)$ with $\supp\mu\subseteq\overline\Om'$ such that, up to a subsequence, $\|\frac{|\log\ep|}{E_\ep}Ju_\ep-\pi\mu\|_{\flt,\Omega}\to 0$ (as $\ep\to 0$).
		\item[(ii)] ($\Gamma$-liminf inequality) For every $\mu\in X(\Omega)$ with $\supp\mu\subseteq\overline\Om'$ and for every $\{u_\ep\}_{\ep}\subset {SBV^2(\Omega;\Ss^1)}$ such that $\|\frac{|\log\ep|}{E_\ep}Ju_\ep-\pi\mu\|_{\flt,\Omega}\to 0$ (as $\ep\to 0$)\,, it holds
		\begin{equation}\label{liminfVECCHIOquasi}
			\pi|\mu|(\Omega)\le \liminf_{\ep\to 0}\frac{ \mathcal F_\ep(u_\varepsilon)}{E_\ep}\,.
		\end{equation}
		\item[(iii)] ($\Gamma$-limsup inequality) For every $\mu\in X(\Omega)$ with $\supp\mu\subseteq\overline\Om'$, there exists  $\{u_\ep\}_{\ep}\subset  {SBV^2(\Omega;\Ss^1)}$ with $\|\frac{|\log\ep|}{E_\ep}Ju_\ep-\pi\mu\|_{\flt,\Omega}\to 0$ (as $\ep\to 0$)\,, such that
		\begin{equation}\label{eq:limsupVECCHIOquasi}
			\pi|\mu|(\Omega)\geq\limsup_{\ep\to 0}\frac{ \mathcal F_\ep(u_\varepsilon)}{E_\ep}\,.
		\end{equation}
	\end{itemize}
\end{theorem}
\begin{proof}[Proof of Theorem \ref{mainthmVECCHIO}]
Although the argument is standard, we briefly discuss how to prove points (i) and (ii), (iii) being identical to the case of \cite{DLSVG}. Assume \eqref{enboundVECCHIO}; by Proposition \ref{cor:21041335} (applied to the domain $\Om'$), for all $\ep>0$ we choose $\widehat u_\ep$ such that 
\begin{align}\label{conditions}
&\int_\Om \frac12|\nabla \widehat u_\ep|^2 dx\leq \int_\Om \frac12|\nabla u_\ep|^2 d x+\ep,\nonumber\\
&\mathcal H^1(\overline S_{\widehat u_\ep})=\mathcal H^1( S_{\widehat u_\ep})\leq \mathcal H^1( S_{ u_\ep})+\ep,\nonumber\\
&\|J\widehat u_\ep-J u_\ep\|_{{\rm flat}, \Om}\leq \ep,
\end{align} 
so that it follows 
$$	\sup_{\ep>0}\frac{ \mathcal G_\ep(\widehat u_\varepsilon)}{|\log\ep|}\le C+1. $$
The compactness result in \cite[Theorem 3.1 (i)]{DLSVG} and the third condition in \eqref{conditions} imply (i). In a similar way also (ii) is a consequence of \cite[Theorem 3.1 (ii)]{DLSVG} and of the same density result.
\end{proof}

We do not discuss the proof of Theorem \ref{mainthmVECCHIOquasi} since it follows from the same result with $\mathcal G_\ep$ in place of $\mathcal F_\ep$, which in turn has the same proof of \cite[Theorem 3.1]{DLSVG}.

\subsection{$\Gamma$-convergence in the critical and supercritical regimes}
Our main results are the following.
\begin{theorem}\label{mainthm}
	The following $\Gamma$-convergence result holds true.
	\begin{itemize}
		\item[(i)] (Compactness) Let $\{u_\ep\}_{\ep}\subset {SBV^2(\Omega;\Ss^1)}$ be such that 
		\begin{equation}\label{enbound}
			\sup_{\ep>0}\frac{\F_\ep(u_\ep)}{|\log\ep|^2}\le C,
		\end{equation}
		for some $C>0$\,. Then there exist a measure $\mu\in\M(\Omega)\cap H^{-1}(\Omega)$ with $\supp\mu\subseteq\overline\Om'$ and a map $\Td\in L^2(\Omega;\R^2)$ with $-\mathrm{Div}\,\Td=\pi\mu$ such that, up to a subsequence, 
		\begin{align}\label{compj}\tag{{FJ}}
			\big\|\frac{Ju_\ep}{\pi|\log\ep|}-\mu\big\|_{\flt,\Omega}\to 0\\ \label{compac}\tag{{ACJ}}
			\frac{T^{D}_{u_\ep}}{|\log\ep|}\weakly \Td\textrm{ in }L^2(\Omega;\R^2)\,.
		\end{align}
		
		\item[(ii)] ($\Gamma$-liminf inequality) For every $(\mu,\Td)\in \big(\M(\Omega)\cap H^{-1}(\Omega)\big)\times L^2(\Omega;\R^{2})$ as in (i) and for every $\{u_\ep\}_{\ep}\subset {SBV^2(\Omega;\Ss^1)}$ satisfying \eqref{compj} and \eqref{compac}, it holds
		\begin{equation}\label{liminf}
			\pi|\mu|(\Omega)+	2\int_{\Omega}|\Td|^2\ud x\le \liminf_{\ep\to 0}\frac{\F_\ep(u_\ep)}{|\log\ep|^2}\,.
		\end{equation}
		\item[(iii)] ($\Gamma$-limsup inequality) For every 
		$(\mu,\Td)\in (\M(\Omega)\cap H^{-1}(\Omega))\times L^2(\Omega;\R^{2})$ as in (i)  there exists $\{u_\ep\}_{\ep}\subset {SBV^2(\Omega;\Ss^1)}$ satisfying \eqref{compj} and \eqref{compac},
		such that
		\begin{equation}\label{eq:limsup}
			\pi|\mu|(\Omega)+2\int_{\Omega}|\Td|^2\ud x\geq\limsup_{\ep\to 0}\frac{\F_\ep(u_\ep)}{|\log\ep|^2}\,.
		\end{equation}
	\end{itemize}
\end{theorem}
\begin{theorem}\label{mainthm_super}
	Let $\{N_\ep\}_{\ep>0}$ be such that $|\log\ep|\ll N_\ep\ll\ep^{-1}$. 
	The following $\Gamma$-convergence result holds true.
	\begin{itemize}
		\item[(i)] (Compactness) Let $\{u_\ep\}_{\ep}\subset {SBV^2(\Omega;\Ss^1)}$ be such that 
		\begin{equation}\label{enboundsuper}
			\sup_{\ep>0}\frac{\F_\ep(u_\ep)}{N^2_\ep}\le C,
		\end{equation}
		for some $C>0$\,. Then there exist a field $\Td\in L^2(\Omega;\R^2)$ such that, up to a subsequence, $\frac{\Td_{u_\ep}}{N_\ep}\weakly \Td$ in $L^2(\Omega;\R^2)$\,.
		\item[(ii)] ($\Gamma$-liminf inequality) For every $\Td\in L^2(\Omega;\R^{2})$ and for every $\{u_\ep\}_{\ep}\subset {SBV^2(\Omega;\Ss^1)}$ with $\frac{\Td_{u_\ep}}{N_\ep}\weakly \Td$ in $L^2(\Omega;\R^2)$\,, it holds
		\begin{equation}\label{liminfsuper}
			2\int_{\Omega}|\Td|^2\ud x\le \liminf_{\ep\to 0}\frac{\F_\ep(u_\ep)}{N^2_\ep}\,.
		\end{equation}
		\item[(iii)] ($\Gamma$-limsup inequality) For every 
		$\Td \in L^2(\Omega;\R^{2})$ there exists $\{u_\ep\}_{\ep}\subset {SBV^2(\Omega;\Ss^1)}$  with $\frac{\Td_{u_\ep}}{N_\ep}\weakly \Td$ in $L^2(\Omega;\R^2)$
		such that
		\begin{equation}\label{eq:limsupsuper}
			2\int_{\Omega}|\Td|^2\ud x\geq\limsup_{\ep\to 0}\frac{\F_\ep(u_\ep)}{N_\ep^2}\,.
		\end{equation}
	\end{itemize}
\end{theorem}
By using the density result in Proposition \ref{cor:21041335}, one can show that Theorems \ref{mainthm} and \ref{mainthm_super} hold true also when replacing $\F_\ep$ with  $\mathcal G_\ep$.

In order to prove Theorems \ref{mainthm} and \ref{mainthm_super}, we will make use of the corresponding core radius approach results that for the sake of completeness we state and prove in Section \ref{cra_section} below.
\section{Core radius approach}\label{cra_section}
We first introduce some notation. 
Let $V\subset\R^2$ be a bounded and open set with Lipschitz continuous boundary.
For every finite family of pairwise (essentially\footnote{That is, whose closures are mutually disjoint.}) disjoint open balls $\B:=\{B^n\}_{n=1,\ldots,N}$ (with $N\in\N$)
we set 
$$V(\B):=V\setminus \bigcup_{n=1}^N\overline{B}^n,$$ and we denote by $\rad(\B)$ the sum of the radii of the balls $B^n$, namely
$$\rad(\B):=\sum_{n=1}^Nr(B^n)\,,$$
where $r(B)$ denotes the radius of the ball $B$\,.
Moreover, for every $\mu\in X(V)$ with $\mu\neq 0$ of the form
\begin{align}\label{mu_B}
	\mu:=\sum_{n=1}^Nz^n\delta_{x(B^n)}\qquad \text{ with }z^n\in{\Z}\setminus\{0\}\,,
\end{align}
we set
\begin{equation}\label{admi}
	\Ad(\B,\mu,V):=\{u\in H^1(V(\B);\Ss^1)\,:\, \deg(u,\partial B^n)=z^n\textrm{ for every }n=1,\ldots,N\}\,.
\end{equation}
Here and below, $x(B)$ denotes the center of the ball $B$\,.

\begin{definition}[Merging procedure]\label{merging_def}
	\rm{Given a finite family $\B=\{B_{r^i}(x^i)\}_{i=1,\ldots,I}$ ($I\in\N$) of balls in $\R^2$, we define a new family $\widehat{\B}$ as follows. If the closures of two balls in $\B$ are not disjoint, then we replace the two balls with a unique ball which contains both of them and has radius less than or equal to the sum of the radii of the original balls. After this, we repeat this replacement recursively, until as all the balls in the family are mutually essentially disjoint. The final family is $\widehat{\B}$.
		The procedure of passing from $\B$ to $\widehat{\B}$ is called {merging procedure} applied to $\B$. Notice that a merging procedure does not increase the sum of all the radii of the balls in the family.
	}
\end{definition}

The following result is proven in \cite[Proposition 2.2]{DLP}.
\begin{proposition}\label{ballconstr}
	{Let $V\subset\R^2$ be a bounded and open set with Lipschitz continuous boundary, let $\B$ be a finite family of pairwise essentially disjoint balls in $\R^2$, let $\mu\in X(V)$ be of the form \eqref{mu_B}, and let $u\in \Ad(\B,\mu,V)$\,.}
	Then, there exists a one-parameter family of open balls $\B(t)$ with $t\ge 0$  such that, setting $U(t):=\bigcup_{B\in\B(t)} B$, the following properties hold true:
	\begin{enumerate}
		\item $\B(0)=\B \, $;
		\item $ U(t_1)\subset U(t_2)$  for any $0\le t_1<t_2 \, $;
		\item the balls in $\B(t)$ are pairwise (essentially) disjoint {for every $t>0$};
		\item for any $0\le t_1<t_2$ and for any open set $A \subseteq V$\,, 
		\begin{equation*}
			\frac 1 2\int_{(U(t_2)\setminus \overline U(t_1))\cap A}|\nabla u|^2\ud x\ge\pi\sum_{\newatop{B\in\B(t_2)}{B\subseteq A}}|\mu(B)|\log\frac{1+t_2}{1+t_1} \, ;
		\end{equation*} 
		\item  {for every $t>0$:} $\displaystyle \sum_{B\in \B(t)}r(B)\le(1+t)\sum_{B\in \B}r(B)$, where $r(B)$ denotes the radius of  $B$\,. 
	\end{enumerate}
\end{proposition}
For every $\B$ and $\mu$ as in Proposition \ref{ballconstr}, for every $t>0$, we set $\Ccal(t):=\{B\in\B(t)\,:\,\overline{B}\subset V\}$ and we define 
\begin{equation}\label{tilde}
	\widetilde\mu:=\sum_{B\in\Ccal(1)}\mu(B)\delta_{x(B)}\,.
\end{equation}
We can now state the crucial result which will be the starting point of the proof of Theorem \ref{mainthm}.
\begin{theorem}\label{alipons_crit}
	Let $V$ be a bounded  open set with Lipschitz boundary. For every $\ep>0$ let $\B_\ep:=\{B_\ep^n\}_{n=1,\ldots,N_\ep}$ (with $N_\ep\in\N$) be a (finite) family of pairwise (essentially) disjoint open balls with	 $\rad(\B_\ep)\to 0$ as $\ep\to 0$\,, $\mu_\ep:=\sum_{n=1}^{N_\ep}z_{\ep}^n\delta_{x(B_\ep^n)}$ with $z_\ep^n\in\Z\setminus\{0\}$ for every $n=1,\ldots,N_\ep$\,. Let moreover $\{u_\ep\}_\ep$ be such that $u_\ep\in \Ad(\B_\ep,\mu_\ep,V)$\,. 
	Assume that 
	\begin{equation}\label{enbound0_crit}
		\sup_{\ep>0}\frac{1}{2|\log\rad(\B_\ep)|^2}\int_{V(\B_\ep)}|\nabla u_\ep|^2\ud x\le C\,,
	\end{equation} 
	for some constant $C>0$ independent of $\ep$\,.
	Then, the following facts hold true.
	\begin{itemize}
		\item[(i)] Let $\widetilde\mu_\ep$ be the measures defined in \eqref{tilde} with $\Ccal(1)=\Ccal_\ep(1)=\{B\in\B_\ep(1)\,:\,\overline{B}\subset V\}$; then $|\widetilde\mu_\ep|(V)\le C|\log\rad(\B_\ep)|^2$ for all $\ep>0$ with a constant $C>0$ independent of $\ep$, and
		there exist a measure $\mu\in \M(V)$ and a function $\lambda\in L^2(V;\R^2)$
		such that,
		up to a subsequence, as $\ep\to 0$
		\begin{align}\label{cra_convmis}
			&\frac{\widetilde\mu_\ep \res V}{|\log\rad(\B_\ep)|}\fla\mu,\\ \label{cra_convcampi}
			&\frac{\lambda_{u_\ep}\chi_{V(\B_\ep)}}{|\log\rad(\B_\ep)|}\rightharpoonup \lambda\qquad \textrm{ weakly in }L^2(V;\R^2)\,;
		\end{align}
			\item[(ii)] $\pi |\mu|(V)+2\int_{V}|\lambda|^2\ud x\le \liminf_{\ep\to 0}\frac{1}{2|\log\rad(\B_\ep)|^2}\int_{V(\B_\ep)}|\nabla u_\ep|^2\ud x$\,.
		\end{itemize}
	\end{theorem}
	Notice that, as $u_\ep\in \Ad(\B_\ep,\mu_\ep,V)$, we have $\lambda_{u_\ep}\in L^2(V(\B_\ep);\R^2)$ (see \eqref{lambda_u}). In formula \eqref{cra_convcampi}, symbol $\lambda_{u_\ep}\chi_{V(\B_\ep)}$ denotes the extension of $\lambda_{u_\ep}$ to the constant $(0,0)$ in $V\setminus \overline{V(\B_\ep)}$. 
	\begin{proof}
		We start by proving (i). Our proof closely resembles that of \cite[Theorem 3.2]{AP} where the compactness result is proven in the energy regime
		$|\log\rad(\B_\ep)|$\,. 
		
		For every $0<p<1$ and for every $\ep>0$ we set
		\begin{equation}\label{tempopi}
			t^p_\ep:=\frac{1}{\rad^{1-p}(\B_\ep)}-1\,,\qquad\qquad\nu^p_\ep:=\nu[t^{p}_\ep]\,,
		\end{equation}
		where we have set, for $t\geq0$, $$\nu[t]:=\sum_{B\in\Ccal_\ep(t)}\mu_\ep(B)\delta_{x(B)}\,.$$
		Fix $0<p<1$\,.
		Then, by applying Proposition \ref{ballconstr}(4) (with $t_1=0$ and $t_2=1$ and $t_2=t^p_\ep$) and by the energy bound \eqref{enbound0_crit}, we have that
		\begin{equation}\label{vartolim}
			|\widetilde\mu_\ep|(V)\le C|\log\rad(\B_\ep)|^2\,,\qquad\qquad |\nu^p_\ep|(V)\le {C(1-p)^{-1}}|\log\rad(\B_\ep)|\,,
		\end{equation}
		whence we deduce the first statement in claim (i) and the existence of a measure $\mu^p\in \M_b(V)$ such that  (up to a not-relabeled subsequence)
		\begin{equation}\label{convedebolemisure}
			\frac{\nu^p_\ep}{|\log\rad(\B_\ep)|}\weakstar\mu^p\qquad\textrm{as $\ep\to 0$}\,.
		\end{equation}
		Now we prove that 
		\begin{equation}\label{class}
			\frac{1}{|\log\rad(\B_\ep)|}\big(\widetilde\mu_\ep-\nu^p_\ep\big)\fla 0\qquad \textrm{for every }0<p<1\,,
		\end{equation}
	from which we deduce also that $\mu^p\equiv \mu$ for any $0<p<1$\,.
		To this purpose, we first observe that $\widetilde\mu_\ep(B)=\nu^p_\ep(B)$ for any $B\in\Ccal_\ep(t^p_\ep)=\{B\in \B_\ep(t^p_\ep):\overline B\subset V\}$\,; therefore, using \eqref{vartolim} and Proposition \ref{ballconstr}(5) together with the very definition of $t^p_\ep$\,, for every sequence $\{\ffi_\ep\}_{\ep}\subset C_{\mathrm{c}}^{0,1}(V)$ with $\|\ffi_\ep\|_{C^{0,1}}\le 1$\,, we have
		\begin{equation*}
			\begin{aligned}
				\frac{1}{|\log\rad(\B_\ep)|}\left|\big\langle\widetilde\mu_\ep-\nu^p_\ep,\ffi_\ep\big\rangle\right|\leq&\,\frac{1}{|\log\rad(\B_\ep)|}\left|\sum_{B\in\Ccal_\ep(t^p_\ep)}\int_{B}\Big(\ffi_\ep-\fint_{B}\ffi_\ep\ud x\Big)\ud (\widetilde\mu_\ep-\nu^p_\ep)\right|\\
				&\,+\frac{1}{|\log\rad(\B_\ep)|}\left|\sum_{B\in\B_\ep(t^p_\ep)\setminus\Ccal_\ep(t^p_\ep)}\int_{B\cap V}\ffi_\ep\ud (\widetilde\mu_\ep-\nu^p_\ep)\right|\\
				\le&\,\frac{1}{|\log\rad(\B_\ep)|}\sum_{B\in\B_\ep(t^p_\ep)}\int_{B}\big(\max_{B}\ffi_\ep-\min_{B}\ffi_\ep\big)\big(|\widetilde\mu_\ep|(B)+|\nu^p_\ep|(B)\big)\\
				\le&\,\frac{1}{|\log\rad(\B_\ep)|}\sum_{B\in\B_\ep(t^p_\ep)}\mathrm{diam}(B)\big(|\widetilde\mu_\ep|(V)+|\nu^p_\ep|(V)\big)\\
				\le&\,C\rad^p(\B_\ep)|\log\rad(\B_\ep)|\,,
			\end{aligned}
		\end{equation*}
		whence \eqref{class} follows.
		
		Moreover, by the very definition of $\lambda_{u_\ep}$ in \eqref{lambda_u} and by the energy bound \eqref{enbound0_crit}, we immediately have that
		\begin{equation}\label{limcorr}
			\frac 1 2\int_{V(\B_\ep)}|2\lambda_{u_\ep}|^2\ud x=\frac 1 2\int_{V(\B_\ep)}|\nabla u_\ep|^2\ud x\le C|\log\rad(\B_\ep)|^2\,,
		\end{equation}
		thus, up to extracting a further subsequence, there exists $\lambda\in L^2(V;\R^2)$ such that \eqref{cra_convcampi} holds.
		{Notice that, for $p\in(0,1)$ fixed, since $|U(t^p_\ep)|\rightarrow0$ as $\ep\rightarrow0$, we also deduce 
			\begin{equation}\label{serviva}\frac{\lambda_{ u_\ep}}{|\log\rad(\B_\ep)|}\chi_{V(\B_\ep(t^p_\ep))}\rightharpoonup\lambda \qquad \text{ weakly in }L^2(V;\R^2).
			\end{equation}}
		Now we prove (ii).
		To this end, let $p\in (0,1)$ be fixed;
		by \eqref{convedebolemisure} and by
		Proposition \ref{ballconstr}(4), we get
		\begin{equation}\label{coreen}
			\begin{aligned}
				\liminf_{\ep\to 0}\frac{1}{2|\log\rad(\B_\ep)|^2}\int_{(U(t_\ep^p)\setminus U(0))\cap V}|\nabla u_\ep|^2\ud x
				\ge&\,\pi {(1-p)}\liminf_{\ep\to 0}\frac{|\nu_\ep^p|(V)}{|\log\rad(\B_\ep)|}\\
				\ge&\,\pi {(1-p)}|\mu|(V)\,.
			\end{aligned}
		\end{equation}
		Furthermore, by \eqref{serviva}, we have that
		\begin{equation*}
			\begin{aligned}
				\liminf_{\ep\to 0}\frac{1}{2|\log\rad(\B_\ep)|^2}\int_{V(\B_\ep(t_\ep^p))}|\nabla u_\ep|^2\ud x=&\liminf_{\ep\to 0}\frac 1 2\int_{V}\Big|\frac{2\lambda_{u_\ep}}{|\log\rad(\B_\ep)|}\Big|^2\chi_{V(\B_\ep(t_\ep^p))}\ud x\\
				\ge& \, 2\int_{V}|\lambda|^2\ud x\,,
			\end{aligned}
		\end{equation*}
		which, together with \eqref{coreen} (letting $p\to 0$), yields (ii).
	\end{proof}
		\section{Proof of Theorem \ref{mainthm}}\label{sec:5}
		This section is devoted to the proof of Theorem \ref{mainthm}. 
		Using the density argument as in the proof of Theorem \ref{mainthmVECCHIO}, we can prove Theorem \ref{mainthm} with $\mathcal F_\ep$ replaced by $\mathcal G_\ep$ in \eqref{defEne}.

		The proofs of the compactness and of the lower bound are addressed in Subsection \ref{proofmainthmi} and closely resemble those in the $|\log\ep|$ regime treated in \cite{DLSVG}, whereas the construction of the recovery sequence is provided in Subsection \ref{proofmainthmii}.

		\subsection{Proof of compactness and lower bound}\label{proofmainthmi} 
		By the energy bound \eqref{enbound}, together with Corollary~\ref{thm:densitySBV}, we have that
		\begin{equation}\label{saltolim}
			\Huno(\overline S_{u_\ep})\le C\ep|\log\ep|^2\,,
		\end{equation}
		for every $\ep>0$\,. By the very definition of Hausdorff measure, since $\overline S_{u_\ep}$ is compact, there exists a finite family $\B_\ep$ of open balls (in $\R^2$) such that $\overline S_{u_\ep}\subset \bigcup_{B\in\B_\ep}B$ and $\rad(\B_\ep)\le C\ep|\log\ep|^2$\,. Notice that we can always assume (just by enlarging an arbitrarily chosen ball in $\B_\ep$) that $\rad(\B_\ep)\ge \ep$\,, so that, from now on
		\begin{equation}\label{sommaraggi}
			\ep\le\rad(\B_\ep)\le C\ep|\log\ep|^2\,,
		\end{equation}
		for some $C>0$\,. Moreover, by construction,
		\begin{equation}\label{uhuno}
			u_\ep\in H^1(\Omega(\B_\ep);\Ss^1)\,,
		\end{equation}
		where we recall that $\Omega(\B_\ep):=\Omega\setminus\bigcup_{B\in\B_\ep}\overline B$\,. By \eqref{sommaraggi} and recalling that $\overline S_{u_\ep}\subset\overline{\Omega'}$\,, we can assume that, for $\ep$ small enough, 
		\begin{equation}\label{palledentro}
			\bigcup_{B\in\B_\ep}\overline B\subset\Omega\,.
		\end{equation}
		Up to applying a merging procedure (as described in  Definition \ref{merging_def}) to the balls in $\B_\ep$\,, we can assume without loss of generality that these balls are mutually (essentially) disjoint, and still satisfy \eqref{sommaraggi} and \eqref{palledentro}. 
		For $\ep>0$ small enough  we set
		\begin{equation}\label{defmisin}
			\mu_\ep:=\sum_{B\in\B_\ep}\deg(u_\ep,\partial B)\delta_{x(B)}\,.
		\end{equation}
		By \eqref{enbound}, \eqref{sommaraggi}, and \eqref{uhuno}, for $\ep$ small enough it holds
		\begin{equation}\label{enboundmeglio}
			\frac 1 2\int_{\Omega(\B_\ep)}|\nabla u_\ep|^2\,\ud x \le\mathcal G_\ep(u_\ep)\le C|\log\ep|^2\le C|\log\rad(\B_\ep)|^2\,.
		\end{equation} 
		Therefore we can apply Theorem \ref{alipons_crit} to the family $\{(\B_\ep;\mu_\ep)\}_\ep$\,.
		Notice that, in view of the very definition of $\mathcal G_\ep$\,, we have that also the family $\B_\ep(1)$ satisfies \eqref{palledentro} (for $\ep$ small enough), so that $\Ccal_\ep(1)\equiv \B_\ep(1)$\,.
		Setting
		$$
		\widetilde\mu_\ep:=\sum_{B\in\B_\ep(1)}\mu_\ep(B)\delta_{x(B)}\,,
		$$
		by Theorem \ref{alipons_crit}(i) (more precisely, by \eqref{cra_convmis}), using \eqref{sommaraggi}, we have that, up to a subsequence,
		\begin{equation}\label{dacorerad}
			\frac{\widetilde\mu_\ep}{|\log\ep|}\fla \mu\,,
		\end{equation} 
		for some $\mu\in\M(\Omega)$\,. By construction, $\supp\mu\subseteq \overline{\Omega'}$\,.
		Moreover, by arguing verbatim as in the proof of \cite[formula (3.17)]{DLSVG} one can prove that
		\begin{equation*}
			\frac{1}{|\log\ep|}\|{Ju_\ep}-\pi\widetilde\mu_\ep\|_{\flt,\Omega}\to 0\,,
		\end{equation*}
		which, combined with \eqref{dacorerad}, yields \eqref{compj}.
		Furthermore, by \eqref{perlift} and by \eqref{enbound}, we have that
		\begin{equation*}
			\frac{1}{|\log\ep|^2}\int_{\Omega}|2\Td_{u_\ep}|^2\ud x=\frac{1}{|\log\ep|^2}\int_{\Omega}|\nabla u_\ep|^2\ud x\le C\,,
		\end{equation*}
		so that, up to a further subsequence,
		\begin{equation}\label{convebetaep}
			\frac{\Td_{u_\ep}}{|\log\ep|}\weakly \Td\qquad\textrm{in }L^2(\Omega;\R^2)\,,
		\end{equation}
		for some field $\Td\in L^2(\Omega;\R^2)$\,. This proves \eqref{compac}.
		
		It remains to show that $-\mathrm{Div}\,\Td=\pi\mu$\,, which will imply also that $\mu\in H^{-1}(\Omega)$\,.
		To this end, let $\ffi\in\Cc^{\infty}(\Omega)$\,; then, by \eqref{perlift},
		\begin{equation}\label{aczero}
			\langle Ju_\ep,\ffi\rangle = \langle T_{u_\ep},\nabla\ffi\rangle=\langle \Td_{u_\ep},\nabla\ffi\rangle+\langle T^S_{u_\ep},\nabla\ffi\rangle\,,
		\end{equation}
		for every $\ep>0$\,.
		By \eqref{compac}, we have that
		\begin{equation}\label{acuno}
			\frac{1}{|\log\ep|}\langle \Td_{u_\ep},\nabla\ffi \rangle \to \langle\Td,\nabla \ffi\rangle\qquad \textrm{as }\ep\to 0\,;
		\end{equation}
		{moreover, by \eqref{perlift} and \eqref{enbound}, we have that
			\begin{equation}\label{acduesper}
				\big|\langle T^{S}_{u_\ep},\nabla\ffi\rangle\big|\le C\ep\|\nabla\ffi\|_{L^\infty}|\log\ep|^2\,.
			\end{equation}
		}
		By, \eqref{compj}, \eqref{aczero}, \eqref{acuno}, and \eqref{acduesper}, 
		\begin{equation*}
			\langle \pi\mu,\ffi\rangle=\lim_{\ep\to 0}\frac{1}{|\log\ep|}\langle Ju_\ep,\ffi\rangle=\lim_{\ep\to 0}\frac{1}{|\log\ep|}\langle \Td_{u_\ep}, \nabla \ffi\rangle=\langle\Td,\nabla\ffi\rangle=\langle -\mathrm{Div}\Td,\phi\rangle\,,
		\end{equation*}
		thus concluding the proof of (i).
		\medskip
		
		Now we prove the lower bound (ii). We can assume without loss of generality that \eqref{enbound} holds true. By the first inequality in \eqref{enboundmeglio} and by Theorem \ref{alipons_crit}(ii) we have immediately that
		\begin{equation*}
			\liminf_{\ep\to 0}\frac{\mathcal G_\ep(u_\ep)}{|\log\ep|^2}\ge \pi|\mu|(\Omega)+2\int_{\Omega}|\Td|^2\ud x\,,
		\end{equation*}
		where we have used also that the function $\Td$ coincides with the field $\lambda$ in Theorem \ref{alipons_crit}. The claim thus follows.

\subsection{Proof of the upper bound}\label{proofmainthmii}
In order to construct the recovery sequence, we first introduce some notation.
Let $r>0$ be fixed\,. For every finite sum of Dirac deltas $\mu:=\sum_{n=1}^{N}\delta_{x^n}$ with $|x^{n_1}-x^{n_2}|\geq 2r$ for $n_1\neq n_2$ and for every $0<\rho<r$ we set
\begin{equation}\label{forsesola}
\widehat\mu^\rho:=\frac{1}{2\pi\rho}\sum_{n=1}^{N}\Huno\res\partial B_{\rho}(x^n)\,,\qquad\widetilde f^\rho:=\frac{1}{\pi\rho^2}\sum_{n=1}^{N}\chi_{B_{\rho}(x^n)}\,,\qquad\textrm{and}\qquad \widetilde\mu^\rho:=\widetilde f^{\rho}\ud x\,.
\end{equation}
For every $r>0$ and for every $z\in\R^2$\,, we recall that  $Q_r(z)$ denotes the (open) square centered at $z$ with sides parallel to the cartesian axes and side-length equal to $2r$\,, i.e., $Q_r(z):=z+Q_r(0)$\,, with $Q_r(0):=(-r,r)^2$\,. 

\begin{lemma}	\label{ovvieta}
Let $\mu:=\sum_{l=1}^L m^l\chi_{\omega^l}\mathrm{d} x$, where $L\in\N$, $m^l\in\R$ 
and $\{\omega^l\}_{l=1,\ldots,L}$ is a partition of $\Omega$ into sets with Lipschitz continuous boundary. 
Let $N_\ep\to +\infty$ as $\ep\rightarrow0$\,. For every $\ep>0$ and for every $l=1,\ldots,L$ with $m^l\neq 0$,
set 
\begin{equation}\label{defrepl}
r^l_\ep:=\frac{1}{2\sqrt{N_\ep|m^l|}}.
\end{equation}
For every $l=1,\ldots,L$ with $m^l\neq 0$, let  
$\mathcal{Z}_\ep^{l}:=\{z\in 2r^l_\ep\Z^2\,:\,{{\ol Q}_{r^l_\ep}(z)}\subset\omega^l\}$ and $N_\ep^l:=\sharp \mathcal{Z}_\ep^{l}$\,.
Then, for every $l=1,\ldots,L$ with $m^l\neq 0$,
\begin{equation}\label{sistema}
\frac{N^l_\ep}{N_\ep}\to |m^l| |\omega^l| \qquad\textrm{as }\ep\to 0\,.
\end{equation}
Moreover, setting
$\mu_\ep^l:=\sum_{n=1}^{N_\ep^l}\delta_{x_\ep^{l,n}}$ (where $\{x_\ep^{l,n}\}_{n=1}^{N_\ep^l}$ is the set of points in $\mathcal{Z}_\ep^{l}$) for every $l=1,\ldots,L$ with $m^l\neq 0$,
$\mu_\ep^l\equiv 0$ whenever $m^l=0$, and 
$\mu_\ep:=\sum_{l=1}^{L}\mu_\ep^l$, we have that the sequence $\{\mu_\ep\}_\ep\subset\M(\Omega)$ satisfies
\begin{itemize}
\item[(a)]  for every $\ep>0$ and for every $l=1,\ldots,L$ with $m^l\neq 0$ and $n=1,\ldots,N_\ep^{l}$\,, the balls $B_{r^l_\ep}(x^{l,n}_\ep)$ are pairwise disjoint and contained in $\omega^l$\,;
\item[(b)] $\frac{\mu_\ep}{N_\ep}\weakstar \mu$ in $\M(\Omega)$ as $\ep\to 0$\,;
\item[(c)] $\big\|\frac{\widetilde\mu_\ep}{N_\ep}-\mu\big\|_{H^{-1}(\Omega)}\le{C}{{N_\ep^{-\frac 1 4}}}$ (for $\ep$ small enough),
\end{itemize}
where $\widetilde\mu_\ep:=\sum_{l=1}^{L}\widetilde\mu_\ep^{l,r^l_\ep}$, with $\widetilde\mu_\ep^{l,r^l_\ep}$ defined as in \eqref{forsesola} starting from $\mu_\ep^l$\,.
\end{lemma}
\begin{proof}
For every $l=1,\ldots,L$ with $m^l\neq 0$ we set $\overline{\omega}^l_\ep:=\bigcup_{z\in\mathcal{Z}_\ep^l}\overline{Q}_{r^l_\ep}(z)$ and we denote by $\omega_\ep^l:=\text{{\rm int}}({\overline\omega^l_\ep})$ the set of interior points of $\overline\omega_\ep^l$\,.
We set $r_\ep:=\min\{r_\ep^l:l=1,\dots,L\}$ and $R_\ep:=\max\{r_\ep^l:l=1,\dots,L,\;m^l\neq0\}$, and we notice that $r_\ep\rightarrow 0$ as $\ep\rightarrow0$.

Property \eqref{sistema} is straightforward. Indeed, let $l\in\{1,\ldots,L\}$ with $m^l\neq 0$; setting $(\partial \om^l)_{2r^l_\ep}:=\{x\in \om^l:\dist(x, \partial\om^l)<2r^l_\ep\}$\,, it is sufficient to observe that $\left|(\omega^l\setminus \ol\omega^l_\ep)\setminus(\partial \om^l)_{2r^l_\ep}\right|=0$\,, and hence,
by the Lipschitz continuity of $\partial\omega^l$\,, 
\begin{equation}\label{area_piccola}
|\om^l\setminus \ol\om^l_\ep|=\mathrm{O}(r^l_\ep)=\mathrm{O}(R_\ep)\,,
\end{equation}
where $\limsup_{\ep\rightarrow0}\mathrm{O}(R_\ep)R_\ep^{-1}<C<+\infty$. On the other hand, $|\ol\omega^l_\ep||m^l|=4(r_\ep^l)^2N_\ep^l|m^l|=\frac{N_\ep^l}{N_\ep}$, which, together with \eqref{area_piccola}, yields 
$|\omega^l|=|\om^l\setminus \ol\om^l_\ep|+|\ol\om^l_\ep|=\frac{\mathrm{O}(R_\ep)}{|m^l|}+\frac{N_\ep^l}{|m^l|N_\ep}$\,, and hence \eqref{sistema}.

Property (a) follows by construction.
Now we prove property (b). Let $\phi\in \Cc(\Omega)$\,, and let $m_\phi$ be the modulus of continuity of $\phi$\,. Then, setting 
\begin{equation}\label{defemmegrande}
M:=\max\{|m^l|:l=1,\dots,L\},
\end{equation}
we have that, as $\ep\to 0$,
\begin{equation*}
\begin{aligned}
\Big|\Big\langle \frac{\mu_\ep}{N_\ep}-\mu,\phi\Big\rangle_{\Omega}\Big|\le&\,\sum_{\newatop{l=1}{m^l\neq 0}}^Lm^l\sum_{n=1}^{N^l_\ep}\int_{Q_{r^l_\ep}(x_\ep^{l,n})}|\phi(x_\ep^{l,n})-\phi(x)|\,\ud x+\sum_{\newatop{l=1}{ m^l\neq 0}}^Lm^l\int_{\omega^l\setminus\ol\omega^l_\ep}|\phi(x)|\,\ud x\\
\le&\,4M\sum_{\newatop{l=1}{ m^l\neq 0}}^LN^l_\ep (r_\ep^l)^2 \max_{t\in [0,2\sqrt{2} r^l_\ep]}m_\phi(t)+M\|\phi\|_{L^\infty}\sum_{\newatop{l=1}{ m^l\neq 0}}^L|\omega^l\setminus\ol\om^l_\ep|\to 0,
\end{aligned}
\end{equation*}
where we have used \eqref{defrepl}, \eqref{sistema}, and \eqref{area_piccola};
this proves (b). 

We are left with the proof of (c). 
We set 
\begin{align*}
	\omega:=\bigcup_{\newatop{l=1}{ m^l\neq 0}}^{L}\omega^l,\qquad \qquad 
	\omega_\ep:=\bigcup_{\newatop{l=1}{ m^l\neq 0}}^{L}\omega_\ep^l,\qquad \qquad \eta_\ep:=(\frac{\widetilde\mu_\ep^{r_\ep}}{N_\ep}-\mu)\chi_{\om_\ep},
\end{align*}
 so that 
\begin{equation}\label{decomis}
	\frac{\widetilde\mu_\ep^{r_\ep}}{N_\ep}-\mu=\eta_\ep-\mu\chi_{\omega\setminus \om_\ep}\,.
\end{equation}
Let $\bar v\in H^1(Q_1(0))$ be  {a} solution to 
\begin{equation*}
	\left\{
	\begin{array}{ll}
		\Delta v=\frac{4}{\pi}\chi_{B_1}-1&\textrm{ in }Q_1(0)\\
		\partial_\nu v=0&\textrm{ on }\partial Q_1(0)\,.
	\end{array}
	\right.
\end{equation*}
We then define for every $\ep>0$ and for every $l=1,\ldots,L,$ and $n=1,\ldots,N_\ep^l$\,,
\begin{align*}
v_\ep^{l,n}(x):={\frac{1}{2N_\ep}} \bar v(\frac{x-x_\ep^{l,n}}{r^l_\ep});
\end{align*}
notice that $\nabla  v_\ep^{l,n}(x)=\sqrt{\frac{m^l}{N_\ep}}\nabla \overline v(\frac{x-x_\ep^{l,n}}{r^l_\ep})$, so
\begin{equation}\label{riscala}
	\|\nabla v_\ep^{l,n}\|^2_{L^2(Q_{r^l_\ep}(x^{l,n}_\ep);\R^2)}=\frac{1}{4N_\ep^2}\|\nabla \bar v\|^2_{L^2(Q_1(0);\R^2)}
\end{equation}
and 
\begin{equation*}
	-\Delta v^{l,n}_\ep=\eta_\ep\qquad \text{in}\;\;\;Q_{r^l_\ep}(x^{l,n}_\ep)\,
	,\qquad\qquad\partial_\nu v_\ep^{l,n}=0\;\;\text{ on }\;\;\partial Q_{r^l_\ep}(x^{l,n}_\ep).
\end{equation*}
Integrating by parts, using H\"older inequality, \eqref{riscala}, and Young inequality, it follows that 
\begin{equation}\label{quasilast}
	\begin{aligned}
		\|\eta_\ep\|_{H^{-1}(\Omega)} =&\,\sup_{\newatop{\phi\in H^1_{0}(\Om)}{\|\phi\|_{H^1(\Om)}\le 1}}\int_\Om\phi \ud\eta_\ep=\sup_{\newatop{\phi\in H^1_{0}(\Om)}{\|\phi\|_{H^1(\Om)}\le 1}}
		\sum_{\newatop{l=1}{m^l\neq 0}}^{L}\sum_{n=1}^{N^l_\ep}	\int_{Q_{r^l_\ep}(x_\ep^{l,n})}\phi \Delta v^{l,n}_\ep \ud x\\
		=&\,\sup_{\newatop{\phi\in H^1_{0}(\Om)}{\|\phi\|_{H^1(\Om)}\le 1}}\sum_{\newatop{l=1}{m^l\neq 0}}^L \left(\sum_{n=1}^{N^l_\ep}	\int_{Q_{r^l_\ep}(x_\ep^{l,n})}\nabla\phi\cdot \nabla v^{l,n}_\ep \ud x\right)\\
		\le&\, \sup_{\newatop{\phi\in H^1_{0}(\Om)}{\|\phi\|_{H^1(\Om)}\le 1}}\sum_{\newatop{l=1}{m^l\neq 0}}^{L}\sum_{n=1}^{N^l_\ep}	\|\nabla\phi\|_{L^2(Q_{r^l_\ep}(x_\ep^{l,n});\R^2)}\| \nabla v^{l,n}_\ep\|_{L^2(Q_{r^l_\ep}(x_\ep^{l,n});\R^2)}\\
		=&\,\sup_{\newatop{\phi\in H^1_{0}(\Om)}{\|\phi\|_{H^1(\Om)}\le 1}}\sum_{\newatop{l=1}{m^l\neq 0}}^{L}\frac{1}{2N_\ep}\sum_{n=1}^{N^l_\ep}	\|\nabla\phi\|_{L^2(Q_{r^l_\ep}(x_\ep^{l,n});\R^2)}\| \nabla \bar v\|_{L^2(Q_{1}(0);\R^2)}\\
		\le &\,\sup_{\newatop{\phi\in H^1_{0}(\Om)}{\|\phi\|_{H^1(\Om)}\le 1}}\sum_{\newatop{l=1}{m^l\neq 0}}^{L}\frac{1}{4N_\ep}\sum_{n=1}^{N_\ep^l}\Big(\frac{1}{N_\ep^{\frac12}}\|\nabla \bar v\|^2_{L^2(Q_1(0);\R^2)}+N_\ep^{\frac12}\|\nabla\phi\|^2_{L^2(Q_{r^l_\ep}(x_\ep^{l,n});\R^2)}\Big)\\
		\le &\,\frac{1}{2}\sum_{\newatop{l=1}{m^l\neq 0}}^L\frac{N_\ep^l}{N_\ep^{\frac32}}\|\nabla \bar v\|^2_{L^2(Q_1(0);\R^2)}+ \frac{1}{4N_\ep^{\frac 1 2}} \le \frac{C}{N_\ep^{\frac 1 2}}\,,
	\end{aligned}
\end{equation}
where the last inequality follows from \eqref{sistema} and \eqref{defrepl}.
Finally, by H\"older inequality, \eqref{defrepl} and \eqref{area_piccola}, we obtain
\begin{equation}\label{last2}
	\|\mu\chi_{\omega\setminus\omega_\ep}\|_{H^{-1}(\Omega)}=\sup_{\newatop{\phi\in H^1_{0}(\Omega)}{\|\nabla\phi\|_{L^2(\Omega;\R^2)}\le 1}}\int_{\omega\setminus\omega_\ep}\phi\ud \mu
	\le \sup_{\newatop{\phi\in H^1_{0}(\Omega)}{\|\nabla\phi\|_{L^2(\Omega;\R^2)}\le 1}}M \|\phi\|_{L^2(\Omega)}|\omega\setminus\omega_\ep|^{\frac{1}{2}}\le \frac{C}{N_\ep^{\frac 1 4}}\,,
\end{equation}
with $M$ defined in \eqref{defemmegrande}; this, combined with \eqref{decomis} and \eqref{quasilast}, yields (c).

\end{proof}

\vskip3mm
\begin{proof}[Proof of Theorem \ref{mainthm}(iii)] 
We divide the proof into two cases. 
\medskip

{\it Case 1: $\mu=\sum_{l=1}^{L}m^l\chi_{\omega^l}\mathrm{d} x$\,, 
where $L\in\N$\,, $m^l\in\R$\,, and $\{\omega^l\}_{l=1,\ldots,L}$ is a partition of $\Omega$ into  sets with Lipschitz continuous boundary.
}

We divide the proof into two steps. In the first one we construct the recovery sequence $\{(\mu_\ep;\overline\beta_\ep)\}_\ep$ for the core-radius  problem; in the second step, we exploit the structure of  $\{(\mu_\ep;\overline\beta_\ep)\}_\ep$ to build up  the recovery sequence $\{u_\ep\}_\ep$ for the functional $\mathcal G_\ep$\,.
\vskip3pt
{\it First step: Construction of the discrete measure $\mu_\ep$ and of the core-radius field $\overline\beta_\ep$\,.}	 
For every $\ep>0$\,, we set 
\begin{equation}\label{sceltadiNep}
N_\ep:=\lfloor |\log\ep|\rfloor
\end{equation}
and let
\begin{equation}\label{defmuep}
\mu_\ep:=\sum_{l=1}^{L}\mu^l_\ep=\sum_{\newatop{l=1}{m^l\neq 0}}^{L}\sum_{n=1}^{N_\ep^l}\delta_{x_{\ep}^{l,n}}
\end{equation}
be the measure provided by Lemma \ref{ovvieta}.
Set  
\begin{equation}\label{inssalto}
S:=\{(0;x_2)\,:\,x_2<0\}
\end{equation}
and let 
$\vartheta\in C^{\infty}(\R^2\setminus \overline{S})$ be the function defined by
\begin{equation}\label{deftheta0}
\vartheta(x):=\left\{\begin{array}{ll}
	\arctan\frac{x_2}{x_1}&\textrm{if }x_1>0\\
	\frac{\pi}{2}&\textrm{if }x_1=0 \textrm{ and }x_2>0\\
	\pi+\arctan\frac{x_2}{x_1}&\textrm{if }x_1<0\\
	\frac{3}{2}\pi&\textrm{if }x_1=0  \textrm{ and }x_2<0\,.
\end{array}
\right.
\end{equation}
For every $l=1,\ldots,L$ with $m^l\neq0$ and for every $n=1,\ldots,N_\ep^l$\,, let $\widehat K_\ep^{l,n}\in L^2_{\loc}(\R^{2}\setminus\{x_\ep^{l,n}\};\R^2)$ and  $\widetilde K^{l,n}_\ep\in L^2_{\loc}(\R^2;\R^2)$  be the functions defined by
\begin{equation*}
\widehat{K}_\ep^{l,n}(x):=\frac{1}{2\pi}\nabla\vartheta(x-x_\ep^{l,n})=\frac{1}{2\pi|x-x_\ep^{l,n}|^2}(-(x_2-x_{\ep,2}^{l,n}); x_1-x_{\ep,1}^{l,n})\,,
\end{equation*}
and
\begin{equation*}
\widetilde{K}_\ep^{l,n}(x):=\frac{|x-x_\ep^{l,n}|^2}{2\pi (r^l_\ep)^2}\nabla\vartheta(x-x_\ep^{l,n})=\frac{1}{2\pi (r^l_\ep)^2}(-(x_2-x_{\ep,2}^{l,n}); x_1-x_{\ep,1}^{l,n})\,,
\end{equation*}
respectively.
Recalling that $A_{r,R}(x):=B_R(x)\setminus\overline{B}_{r}(x)$ (for every $0<r<R$ and for every $x\in\R^2$)\,, we define
\begin{equation*}
\widehat K_\ep:=\sum_{\newatop{l=1}{ m^l\neq 0}}^{L}\sum_{n=1}^{N^l_\ep}\widehat K_\ep^{l,n}\chi_{A_{\ep,r^l_\ep}(x_\ep^{l,n})},	
\qquad\qquad
\widetilde K_\ep:=\sum_{\newatop{l=1}{ m^l\neq 0}}^{L}\sum_{n=1}^{N^l_\ep}\widetilde K_\ep^{l,n}\chi_{B_{r^l_\ep}(x_\ep^{l,n})}\,.
\end{equation*}	
Now, as in \eqref{forsesola}, for all $l=1,\dots,L$ with $m^l\neq 0$ and for every $0<\rho<r_\ep^l$ we set
$$\widetilde\mu^{l,\rho}_\ep:=\frac{1}{\pi\rho^2}\sum_{n=1}^{N^l_\ep}\chi_{B_{\rho}(x^{l,n}_\ep)}\ud x\,,\qquad \text{and }\qquad\widetilde\mu_\ep:=\sum_{l=1}^{L}\widetilde\mu_\ep^{l,r^l_\ep},$$
and similarly
$$\widehat\mu^{l,\rho}_\ep:=\frac{1}{2\pi\rho}\sum_{n=1}^{N}\Huno\res\partial B_{\rho}(x^{l,n}_\ep)\,,\qquad \text{and }\qquad\widehat\mu_\ep:=\sum_{l=1}^{L}\widehat\mu_\ep^{l,r^l_\ep}.$$
Eventually, we introduce
$$\widehat\mu^\ep_\ep:=\sum_{l=1}^{L}\widehat\mu_\ep^{l,\ep}$$
Then, using the notation just introduced, 
we have
\begin{equation}\label{icurl}
\cu\widehat K_\ep=\widehat\mu_\ep^\ep-\widehat\mu_\ep\qquad\textrm{and}\qquad\cu\widetilde K_\ep=\widetilde\mu_\ep-\widehat\mu_\ep\,.
\end{equation}
Let $v_\ep\in H^1(\Omega)$ be {the} solution to 
\begin{equation}\label{ellipt}
\left\{
\begin{array}{ll}
	-\Delta v=\widetilde\mu_\ep-N_\ep \mu&\textrm{in }\Omega\\
	{v= 0}&\textrm{on }\partial\Omega\,.
\end{array}
\right.
\end{equation} 
{
Then, by Poincar\'e inequality and by Lemma \eqref{ovvieta}(c), we get
\begin{equation*}
	\|\nabla v_\ep\|_{L^2(\Omega;\R^2)}^2\le \|\widetilde\mu_\ep-N_\ep\mu\|_{H^{-1}(\Omega)}\|v_\ep\|_{H^1_0(\Omega)}\le C(\Omega)N_\ep^{\frac 3 4}\|\nabla v_\ep\|_{L^2(\Omega;\R^2)}\,,
\end{equation*}
whence, recalling \eqref{sceltadiNep} we get 
\begin{equation}\label{tendeazero}
	\frac{\nabla v_\ep}{\sqrt{N_\ep|\log\ep|}}\to 0\qquad\textrm{in }L^2(\Omega;\R^2)\,.
\end{equation}
}
Let 
\begin{equation}\label{defbeta}
\beta:=-{\frac{1}{\pi}}(\Td)^{\perp}\in L^2(\Omega;\R^2)\,,
\end{equation}
and set
\begin{equation}\label{betafuori}
\overline\beta_\ep:=N_\ep\beta+\widehat K_\ep-\widetilde K_\ep+\nabla^{\perp} v_\ep\,.
\end{equation}
By \eqref{icurl} and \eqref{ellipt},
\begin{equation}\label{gradogiusto}
\cu\overline\beta_\ep\res\Omega=N_\ep\mu+\widehat\mu_\ep^\ep-\widehat\mu_\ep-\widetilde\mu_\ep+\widehat\mu_\ep+\widetilde\mu_\ep-N_\ep \mu=\widehat\mu_\ep^\ep\,,
\end{equation}
so that $\cu\overline\beta_\ep=0$ in $\Omega_\ep(\mu_\ep):=\Omega\setminus\bigcup_{\newatop{l=1}{ m^l\neq 0}}^L\bigcup_{n=1}^{N_\ep^l}\overline{B}_\ep(x_\ep^{l,n})$\,.
Furthermore, by \eqref{gradogiusto}, for any $l=1,\ldots,L$ with $m^l\neq 0$ and for any $n=1,\ldots,N_\ep^l$, we have
\begin{equation*}
\int_{\partial B_\rho(x_\ep^{l,n})}\overline\beta_\ep\cdot \ttau\ud\Huno={1}
\qquad\textrm{for a.e. }\rho\in (\ep,r^l_\ep)\,.
\end{equation*}
Again by \eqref{gradogiusto}, setting $S_\ep^{l,n}:=x_\ep^{l,n}+S$ (with $S$ defined in \eqref{inssalto}) for every $l=1,\ldots,L$ with $m^l\neq 0$ and $n=1,\ldots, N^l_\ep$\,, there exists a function $\overline\vartheta_\ep\in H^1\Big(\Omega_\ep(\mu_\ep)\setminus\bigcup_{\newatop{l=1}{ m^l\neq 0}}^L\bigcup_{n=1}^{N^l_\ep}S_\ep^{l,n}\Big)$ such that
\begin{align}\label{def_thetaeps}
\overline\beta_\ep=\nabla\overline\vartheta_\ep\qquad\text{a.e. on }\Omega_\ep(\mu_\ep)\setminus\bigcup_{\newatop{l=1}{ m^l\neq 0}}^L\bigcup_{n=1}^{N^l_\ep}S_\ep^{l,n}.
\end{align}
In what follows, with a little abuse of notation, we denote by $\overline\vartheta_\ep$ and $\overline\beta_\ep$ the zero-extensions of $\overline\vartheta_\ep$ and $\overline\beta_\ep$ to $\bigcup_{\newatop{l=1}{m^l\neq 0}}^L\bigcup_{n=1}^{N_\ep^l}B_\ep(x_\ep^{l,n})$\,, respectively.
We now prove that 
\begin{equation}\label{convbeta}
\frac{\overline\beta_\ep}{\sqrt{N_\ep|\log\ep|}}\weakly \beta\qquad\textrm{in }L^2(\Omega;\R^{2})\,.	
\end{equation}
On the one hand, 
\begin{equation*}
\begin{aligned}
	\frac{1}{\sqrt{N_\ep|\log\ep|}}\int_{\Omega}|\widehat K_\ep|\ud x=&\,\frac{1}{\sqrt{N_\ep|\log\ep|}}\sum_{\newatop{l=1}{m^l\neq 0}}^L\sum_{n=1}^{N_\ep^l}\int_{A_{\ep,r^l_\ep}(x_\ep^{l,n})}|\widehat K_\ep^{l,n}|\ud x\\
	=&\,\frac{2\pi}{\sqrt{N_\ep|\log\ep|}}\sum_{\newatop{l=1}{m^l\neq 0}}^{L}(r^l_\ep-\ep)N_\ep^l\to 0\,.
\end{aligned}
\end{equation*}
On the other hand, recalling \eqref{sistema}, we also have
\begin{equation}\label{limitecore0}
\begin{aligned}
	\lim_{\ep\to 0}\frac{1}{N_\ep|\log\ep|}\int_{\Omega}|\widehat K_\ep|^2\ud x=&\,
	\lim_{\ep\to 0}\frac{1}{N_\ep|\log\ep|}\sum_{\newatop{l=1}{m^l\neq 0}}^L\sum_{n=1}^{N^l_\ep}\int_{A_{\ep,r^l_\ep}(x_\ep^{l,n})}|\widehat K^{l,n}_\ep|^2\ud x\\
	=&\,\lim_{\ep\to 0}\frac{\sum_{\newatop{l=1}{m^l\neq 0}}^LN_\ep^l\log\frac{r^l_\ep}{\ep}}{N_\ep|\log\ep|}\frac{1}{2\pi}= \frac{1}{2\pi}{{|\mu|(\Omega)}}\,,
\end{aligned}
\end{equation}
so that 
\begin{equation}\label{convdebkcapp}
\frac{\widehat K_\ep}{\sqrt{N_\ep|\log\ep|}}\weakly 0\qquad\textrm{in }L^2(\Omega;\R^2)\,.
\end{equation}
Moreover, by construction,
\begin{multline}\label{convfortektilde}
\frac{1}{N_\ep|\log\ep|}\int_{\Omega}|\widetilde K_\ep|^2\ud x\\
=\frac{1}{N_\ep|\log\ep|}\sum_{\newatop{l=1}{m^l\neq 0}}^{L}\frac{1}{4\pi^2 (r_\ep^l)^2}\sum_{n=1}^{N_\ep^l}\int_{B_{r^l_\ep}(x_\ep^{l,n})}|x-x_\ep^{l,n}|^2\ud x
\to 0\quad\textrm{in }L^2(\Omega;\R^2)\,;
\end{multline}
therefore, by the very definition of $\overline\beta_\ep$ in \eqref{betafuori}\,,
using \eqref{convdebkcapp}, \eqref{convfortektilde}, and \eqref{tendeazero}, we deduce \eqref{convbeta}.
Moreover, by \eqref{limitecore0},  \eqref{convfortektilde} and \eqref{tendeazero}, we easily get
\begin{equation}\label{stimaenbeta}
\lim_{\ep\to 0}
\frac{1}{N_\ep|\log\ep|}\int_{\Omega}|\overline\beta_\ep|^2\ud x=\frac{1}{2\pi}|\mu|(\Omega)+\int_{\Omega}|\beta|^2\ud x\,.
\end{equation}
\vskip3pt
{\it Second step: Construction of the recovery sequence $\{u_\ep\}_\ep$\,.}
{
Let $\bar l\in\{1,\ldots,L\}$ be such that $m^{\bar l}\neq 0$ and $\bar n\in\{1,\ldots,N_\ep^{\bar l}\}$ be fixed. The set $A_{\ep,2\ep}(x_\ep^{\bar l,\bar n})\setminus \bigcup_{\newatop{l=1}{m^l\neq 0}}^{L}\bigcup_{n=1}^{N_\ep^l}S_\ep^{l,n}$ is either connected or is given by the union of the two sets 
$$A^\pm_{\ep,2\ep}(x_\ep^{\bar l,\bar n}):=A_{\ep,2\ep}(x_\ep^{\bar l,\bar n})\cap\{x_1\gtrless x_{\ep,1}^{\bar l,\bar n}\}.$$
We set $a_\ep^{\bar l,\bar n,+}:=\fint_{A^+_{\ep,2\ep}(x_\ep^{\bar l,\bar n})}\overline\vartheta_\ep\ud x$ and $a_\ep^{\bar l,\bar n,-}:=\fint_{A^-_{\ep,2\ep}(x_\ep^{\bar l,\bar n})}\overline\vartheta_\ep\ud x$\,, where $\ol\vartheta_\ep$ is the function in \eqref{def_thetaeps}.
By construction $\overline\vartheta_\ep\in H^1(A^+_{\ep,2\ep}(x_\ep^{\bar l,\bar n}))$ and $\overline\vartheta_\ep\in H^1(A^-_{\ep,2\ep}(x_\ep^{\bar l,\bar n}))$\,, so that, since the sets $A^\pm_{\ep,2\ep}(x_\ep^{\bar l,\bar n})$ have Lipschitz continuous boundary, we can apply the Poincar\'e-Wirtinger inequality in $H^1(A^+_{\ep,2\ep}(x_\ep^{\bar l,\bar n}))$ and $H^1(A^-_{\ep,2\ep}(x_\ep^{\bar l,\bar n}))$\,, thus getting
\begin{equation}\label{usopoinwirt}
	\|\overline\vartheta_\ep-a_\ep^{\bar l,\bar n,+}\|^2_{L^2(A^+_{\ep,2\ep}(x_\ep^{\bar l,\bar n}))}+\|\overline\vartheta_\ep-a_\ep^{\bar l,\bar n,-}\|^2_{L^2(A^-_{\ep,2\ep}(x_\ep^{\bar l,\bar n}))}\le C\ep^2 \|\overline{\beta}_\ep\|^2_{L^2(A_{\ep,2\ep}(x_\ep^{\bar l,\bar n});\R^2)}\,,
\end{equation}
for some universal constant $C>0$\,.
}

Let $\sigma_\ep\in C^\infty(B_{2\ep}(0);[0,1])$ be such that $\sigma_\ep\equiv 0$ in $B_{\ep}(0)$\,, $\sigma_\ep\equiv 1$ in $A_{\frac 3 2\ep, 2\ep}(0)$ and that
\begin{equation}\label{stimacutoff}
|\nabla\sigma_\ep(x)|\le \frac{C}{\ep}\quad\textrm{for every  }x\in B_{2\ep}(0)\,,
\end{equation} 
for some constant $C>0$ independent of $\ep$ (and of $x$)\,.
For every $\ep>0$ we set
\begin{equation*}
\vartheta_\ep(x):=\left\{\begin{array}{ll}
	\sigma_\ep(x-x^{\bar l,\bar n}_\ep)\overline\vartheta_\ep(x)+(1-\sigma_\ep(x-x_\ep^{\bar l,\bar n}))a_\ep^{\bar l,\bar n,-} &\textrm{if }x\in B^{-}_{2\ep}(x^{\bar l,\bar n}_\ep)\textrm{ for some }\bar l=1,\ldots,L,\,\, \bar n=1,\ldots,N^{\bar l}_\ep\\
	\sigma_\ep(x-x^{\bar l,\bar n}_\ep)\overline\vartheta_\ep(x)+(1-\sigma_\ep(x-x_\ep^{\bar l,\bar n}))a_\ep^{\bar l,\bar n,+} &\textrm{if }x\in B^{+}_{2\ep}(x^{\bar l,\bar n}_\ep)\textrm{ for some }\bar l=1,\ldots,L,\,\, \bar n=1,\ldots,N^{\bar l}_\ep\\
	\overline\vartheta_\ep(x)&\textrm{if }x\in \Omega\setminus\bigcup_{\newatop{l=1}{m^l\neq 0}}^L\bigcup_{n=1}^{N^{l}_\ep}\overline{B}_{2\ep}(x^{l,n}_\ep)\,,
\end{array}
\right.
\end{equation*}
(where $\overline{B}^\pm_{2\ep}(x^{l,n}_\ep):=\overline{B}_{2\ep}(x^{l,n}_\ep)\cap \{x_1\gtrless x^{l,n}_{\ep,1}\}$ and $\bar l$ is such that $m^{\bar l}\neq 0$)
and we define $u_\ep:\Omega\to \Ss^1$ as 
\begin{equation}\label{recovery}
u_\ep(\cdot):=e^{2\pi\imath\vartheta_\ep(\cdot)}\,.
\end{equation}	
By construction,
\begin{equation}\label{uno}
	\overline S_{u_\ep}\subset {\bigcup_{\newatop{l=1}{m^l\neq 0}}^L\bigcup_{n=1}^{N^{l}_\ep}\overline{B}_{\frac32\ep}(x^{l,n}_\ep)}\qquad \text{and }\qquad
	\Huno(\overline S_{u_\ep})\le \sum_{\newatop{l=1}{m^l\neq 0}}^LN^l_\ep 4\ep \,,
\end{equation}
which, in view of \eqref{sistema}, implies
\begin{equation}\label{uno_bis}
	\lim_{\ep\to 0}\frac{1}{N_\ep|\log\ep|}\frac 1 \ep\Huno(\overline S_{u_\ep})=0\,.
\end{equation}
We claim that
\begin{equation}\label{convgradienti}
	\frac{\nabla\vartheta_\ep}{\sqrt{N_\ep|\log\ep|}}\weakly \beta\qquad\textrm{in }L^2(\Omega;\R^{2})\,,	
\end{equation}	
which, together with \eqref{uno_bis}, in view of \eqref{perlift}, will imply \eqref{compac}.
To show \eqref{convgradienti}, we start by observing that, by the very definition of  $\vartheta_\ep$ in $\bigcup_{\newatop{l=1}{m^l\neq 0}}^L\bigcup_{n=1}^{N^l_\ep}A_{\ep,2\ep}(x_\ep^{l,n})$ and by \eqref{usopoinwirt} (applied to every $l=1,\ldots,L$ with $m^l\neq 0$ and $n=1,\ldots,N^l_\ep$)\,, we get
\begin{multline}\label{convl1}
\frac{1}{N_\ep|\log\ep|}\|\nabla\vartheta_\ep\|^2_{L^2(\bigcup_{\newatop{l=1}{m^l\neq 0}}^L\bigcup_{n=1}^{N^l_\ep}A_{\ep,2\ep}(x_\ep^{l,n});\R^2)}\le \,\frac{2}{N_\ep|\log\ep|}\sum_{\newatop{l=1}{m^l\neq 0}}^L\sum_{n=1}^{N^l_\ep}\|\overline{\beta}_\ep\|^2_{L^2(A_{\ep,2\ep}(x_\ep^{l,n});\R^2)}\\
		\qquad\qquad\qquad\qquad\qquad\qquad\qquad\qquad\,+\frac{2}{N_\ep|\log\ep|}\frac{C}{\ep^2}\sum_{\newatop{l=1}{m^l\neq 0}}^L\sum_{n=1}^{N^l_\ep}\|\overline\vartheta_\ep-a_\ep^{l,n,+}\|^2_{L^2(A^{+}_{\ep,2\ep}(x_\ep^{l,n}))}\\
		\qquad\qquad\qquad\qquad\qquad\qquad\qquad\qquad\,+\frac{2}{N_\ep|\log\ep|}\frac{C}{\ep^2}\sum_{\newatop{l=1}{m^l\neq 0}}^L\sum_{n=1}^{N^l_\ep}\|\overline\vartheta_\ep-a_\ep^{l,n,-}\|^2_{L^2(A^{-}_{\ep,2\ep}(x_\ep^{l,n}))}\\
	\qquad\qquad\qquad\qquad\qquad	\le\,\frac{C}{N_\ep|\log\ep|}\sum_{\newatop{l=1}{m^l\neq 0}}^L\sum_{n=1}^{N^l_\ep}\|\overline{\beta}_\ep\|^2_{L^2(A_{\ep,2\ep}(x_\ep^{l,n});\R^2)}\\
			\qquad\qquad\qquad\qquad\qquad\qquad\qquad\le
		\frac{C}{|\log\ep|}+C\|\beta\|^2_{L^2(\bigcup_{\newatop{l=1}{m^l\neq 0}}^L\bigcup_{n=1}^{N^l_\ep}A_{\ep,2\ep}(x_\ep^{l,n});\R^2)}+\mathrm{o}(1)\\
		=\mathrm{o}(1)\to 0\,,
\end{multline}
where in the last inequality we used \eqref{convfortektilde}, \eqref{limitecore0} (with $r_\varepsilon$ replaced by $2\varepsilon$), and \eqref{tendeazero}, to deduce that
\begin{equation*}
	\begin{aligned}
		\frac{1}{N_\ep|\log\ep|}\sum_{\newatop{l=1}{m^l\neq 0}}^L\sum_{n=1}^{N^l_\ep}\|\overline{\beta}_\ep\|^2_{L^2(A_{\ep,2\ep}(x_\ep^{l,n});\R^2)}\le&\, \frac{2}{N_\ep|\log\ep|}\sum_{\newatop{l=1}{m^l\neq 0}}^L \sum_{n=1}^{N^l_\ep}\|\widehat K^{l,n}_\ep\|^2_{L^2(A_{\ep,2\ep}(x_\ep^{l,n});\R^2)}\\
		&\,+\frac{N_\ep}{|\log\ep|}\sum_{\newatop{l=1}{m^l\neq 0}}^L\sum_{n=1}^{N^l_\ep}\|\beta\|^2_{L^2(A_{\ep,2\ep}(x_\ep^{l,n});\R^2)}+\mathrm{o}(1)\\
		\le&\,\frac{C}{|\log\ep|}+C\|\beta\|^2_{L^2\big(\bigcup_{\newatop{l=1}{m^l\neq 0}}^L\bigcup_{n=1}^{N_\ep^l}A_{\ep,2\ep}(x_\ep^{l,n});\R^2\big)}+\mathrm{o}(1)\,.
	\end{aligned}
\end{equation*}
Therefore, by \eqref{convl1} and \eqref{convbeta}, using the very definition of $\vartheta_\ep$\,, we deduce \eqref{convgradienti}.
Furthermore, using \eqref{stimaenbeta} and again \eqref{convl1}, recalling \eqref{defbeta}, we get
\begin{equation}\label{stimaenteta}
	\begin{aligned}
		\lim_{\ep\to 0}
		\frac{1}{N_\ep|\log\ep|}\frac 1 2\int_{\Omega}|\nabla u_\ep|^2\ud x=&\,\lim_{\ep\to 0}
		\frac{1}{N_\ep|\log\ep|}2\pi^2\int_{\Omega}|\beta_\ep|^2\ud x\\
		=&\,
		\pi|\mu|(\Omega)+2\pi^2\int_{\Omega}|\beta|^2\ud x=\pi|\mu|(\Omega)+2\int_{\Omega}|\Td|^2\ud x\,,
	\end{aligned}
\end{equation}
which, combined with \eqref{uno_bis}, 
implies that 
the sequence $\{u_\ep\}_{\ep}$ satisfies \eqref{eq:limsup}.

Now, in order to conclude the proof of (iii) of Theorem \ref{mainthm} in the case $\mu=\sum_{l=1}^Lm^l\chi_{\omega^l}\ud x$\,, it remains to prove that also \eqref{compj} is satisfied.
To this end, we first observe that, by H\"older inequality and the very definition of $u_\ep$ in \eqref{recovery}, for every $l=1,\ldots,L$ with $m^l\neq0$ and for every $n=1,\ldots,N_\ep^l$
\begin{equation*}
\begin{aligned}
	\int_{A_{\frac{3}{2}\ep,2\ep}(x_\ep^{l,n})}|\nabla u_\ep|\ud x\le C \ep\|\nabla\vartheta_\ep\|_{L^2(A_{\frac{3}{2}\ep,2\ep}(x_\ep^{l,n});\R^2)}
\end{aligned}
\end{equation*}
which, by Fubini Theorem and by the Mean Value Theorem, implies that (for every $l=1,\ldots,L$ with $m^l\neq 0$ and for every  $n=1,\ldots,N^l_\ep$) there exists $\frac 3 2 \ep< \rho_\ep^{l,n}<2\ep$ such that
\begin{equation*}
\int_{\partial B_{\!\!\rho^{l,n}_\ep}(x_\ep^{l,n})}|\nabla u_\ep|\ud\Huno  \le C \|\nabla\vartheta_\ep\|_{L^2(A_{\frac{3}{2}\ep,2\ep}(x_\ep^{l,n});\R^2)}\,.
\end{equation*}
{
Therefore, recalling \eqref{uno}, by \eqref{convl1}, for $\ep$ small enough we get
\begin{align}\label{espera_1}
	\sum_{\newatop{l=1}{m^l\neq 0}}^L	\sum_{n=1}^{N^l_\ep}\int_{\partial B_{\!\!\rho^{l,n}_\ep}(x_\ep^{l,n})}\ud |T_{u_\ep}|&= \sum_{\newatop{l=1}{m^l\neq 0}}^L\sum_{n=1}^{N^l_\ep}\int_{\partial B_{\!\!\rho^{l,n}_\ep}(x_\ep^{l,n})}|\Td_{u_\ep}|\ud\Huno\nonumber\\
	& \le C\sqrt{N_\ep}\|\nabla\vartheta_\ep\|_{L^2\big(\bigcup_{\newatop{l=1}{m^l\neq 0}}^L\bigcup_{n=1}^{N^l_\ep}A_{\ep,2\ep}(x_\ep^{l,n});\R^2\big)}\nonumber\\
	&\leq CN_\ep|\log\ep|^{\frac12}\,.
\end{align}
}
Analogously,
again by \eqref{convl1} and by \eqref{uno}, using also H\"older inequality, we have
\begin{multline*}
		{|T_{u_\ep}|}\Big(\bigcup_{\newatop{l=1}{m^l\neq 0}}^{L}\bigcup_{n=1}^{N^l_\ep}B_{\!\rho_\ep^{l,n}}(x_\ep^{l,n})\Big)\le\, 2{|\mathrm{D}u_\ep|}\Big(\bigcup_{\newatop{l=1}{m^l\neq 0}}^{L}\bigcup_{n=1}^{N^l_\ep}B_{\!\rho_\ep^{l,n}}(x_\ep^{l,n})\Big)\\
		\qquad\qquad\qquad\le\,2\ep\|\nabla\vartheta_\ep\|_{L^2(\bigcup_{\newatop{l=1}{m^l\neq 0}}^L\bigcup_{n=1}^{N^l_\ep}B_{\rho_\ep^{l,n}}(x_\ep^{l,n});\R^{2})}\sqrt{\sum_{\newatop{l=1}{m^l\neq 0}}^LN^l_\ep\pi}+C\sum_{\newatop{l=1}{m^l\neq 0}}^{L}N^l_\ep\ep\\
		\leq \,C\ep N_\ep|\log\ep|^{\frac12}\,.
\end{multline*}
Recalling the definition of $\vartheta$ in \eqref{deftheta0}, we define
\begin{equation*}
	v_\ep(\cdot):= \exp\Big({\imath\sum_{\newatop{l=1}{m^l\neq 0}}^L\sum_{n=1}^{N^l_\ep}\vartheta(\cdot-x_\ep^{l,n})}\Big);
\end{equation*}
by construction, $v_\ep\in W^{1,p}(\Omega;\Ss^{1})$ for any $1\le p<2$\,, and $$Jv_\ep=\pi\mu_\ep\qquad \text{in }\mathcal M(\Om).$$ Moreover 
for every $x\in\Omega$ 
\begin{equation}\label{stimapuntuale}
	\begin{aligned}
		|T_{v_\ep}(x)|=&\,|\Td_{v_\ep}(x)|=2|\lambda_{v_\ep}(x)|\le 2\pi\sum_{\newatop{l=1}{m^l\neq 0}}^L\sum_{n=1}^{N^l_\ep}|\nabla\vartheta(x-x_\ep^n)|
		=2\pi\sum_{\newatop{l=1}{m^l\neq 0}}^L \sum_{n=1}^{N^l_\ep}\frac{1}{|x-x_\ep^{l,n}|}\,.
	\end{aligned}
\end{equation}
Therefore, for every $\bar l=1,\ldots,L$ with $m^{\bar l}\neq 0$ and for every $\bar n=1,\ldots,N^{\bar l}_\ep$\,, and for $\ep$ small enough we have
\begin{equation*}
	\begin{aligned}
		|T_{v_\ep}|(\partial B_{\!\rho^{\bar l,\bar n}_\ep}(x_\ep^{\bar l,\bar n}))\le&\,\sum_{\newatop{l=1}{m^l\neq 0}}^{L} \sum_{n=1}^{N^l_\ep}\int_{\partial B_{\!\!\rho_\ep^{\bar l,\bar n}}(x_\ep^{\bar l,\bar n})}\frac{2\pi}{|x-x_\ep^{ l, n}|}\ud\Huno\\
		\le&\,4\pi^2+\sum_{\newatop{n=1}{n\neq \bar n}}^{N_\ep^{\bar l}}\int_{\partial B_{\!\!\rho_\ep^{\bar l,\bar n}}(x_\ep^{\bar l,\bar n})}\frac{2\pi}{|x-x_\ep^{\bar l, n}|} 
		+\sum_{\newatop{l=1}{\newatop{l\neq\bar l}{m^l\neq 0}}}^{L}\sum_{n=1}^{N_\ep^l}\frac{2\pi}{|x-x_\ep^{ l, n}|}\ud\Huno\\
		\le &\, 4\pi^2+4\pi\ep \frac{N^{\bar l}_\ep}{r^{\bar l}_\ep}+\sum_{\newatop{l=1}{\newatop{l\neq\bar l}{m^l\neq 0}}}^{L}N^l_\ep \frac{4\pi\ep}{r^l_\ep}\le 4\pi^2+4\pi\ep \sum_{l=1}^L|m^l||\omega^l|N_\ep^{\frac 3 2} \le C\,,
	\end{aligned}
\end{equation*}
where in the last but one inequality we have used \eqref{sistema} together with the fact that 
$$
\inf_{\newatop{n=1,\ldots,N_\ep^{\bar l}}{n\neq \bar n}}|x_\ep^{\bar l, n}-x_\ep^{\bar l,\bar n}|\geq r_\ep^{\bar l},\qquad \qquad\inf_{\newatop{l=1,\ldots,L}{l\neq \bar l}}\inf_{n=1,\ldots,N_\ep^l}|x_\ep^{l, n}-x_\ep^{\bar l,\bar n}|\geq r^{\bar l}_\ep,
$$
so that  (for $\ep$ small enough) $\di(\partial B_{\rho_\ep^n}(x_\ep^{\bar l,\bar n}),x_\ep^{ l, n})\geq r^{\bar l}_\ep$\,.
It follows that
\begin{equation}\label{espera_2}
	\sum_{\newatop{l=1}{m^l\neq 0}}^{L}\sum_{n=1}^{N_\ep^l}\int_{\partial B_{\!\!\rho^{l,n}_\ep}(x_\ep^{ l, n})}|T_{v_\ep}|\ud\Huno\le CN_\ep=C|\log\ep|\,.
\end{equation}
Analogously, by \eqref{stimapuntuale}, for every $\bar l=1,\ldots,L$ with $m^{\bar l}\neq 0$, for every $\bar n=1,\ldots,N^{\bar l}_\ep$ and for $\ep$ small enough we have
\begin{equation*}
	\begin{aligned}
		|T_{v_\ep}|(B_{\!\rho^{\bar l,\bar n}_\ep}(x_\ep^{\bar l,\bar n}))\le&\,2\pi\int_{B_{\!\!\rho^{\bar l,\bar n}_\ep}(x_\ep^{\bar l,\bar n})}\frac{1}{|x-x_\ep^{\bar l,\bar n}|}\ud x+2\pi\sum_{\newatop{n=1}{n\neq \bar n}}^{N_\ep^{\bar l}}\int_{B_{\!\!\rho^{\bar l,\bar n}_\ep}(x_\ep^{\bar l,\bar n})}\frac{1}{|x-x_\ep^{\bar l, n}|}\ud x\\
		&\,\phantom{2\pi\int_{B_{\rho^{\bar n}_\ep}(x_\ep^{\bar n})}\frac{1}{|x-x_\ep^{\bar n}|}\ud x}+2\pi\sum_{\newatop{l=1}{\newatop{l\neq\bar l}{m^l\neq 0}}}
		^{L}\sum_{n=1}^{N_\ep^l}\int_{B_{\!\!\rho^{\bar l,\bar n}_\ep}(x_\ep^{\bar l,\bar n})}\frac{1}{|x-x_\ep^{l, n}|}\ud x\\
		\le&\,C\ep+CN^{\bar l}_\ep\frac{\ep}{r^{\bar l}_\ep}+C\sum_{\newatop{l=1}{\newatop{l\neq\bar l}{m^l\neq 0}}}
		^{L}N^l_\ep\frac{\ep}{r^l_\ep}\le CN_\ep^{\frac 3 2}\ep\,,
	\end{aligned}
\end{equation*}
whence we deduce that
\begin{equation}\label{nellepalle2}
	\frac{1}{|\log\ep|}{|T_{v_\ep}|}\Big(\bigcup_{\newatop{l=1}{m^l\neq 0}}^{L}\bigcup_{n=1}^{N^l_\ep}B_{\!\rho_\ep^{l, n}}(x_\ep^{ l, n})\Big)\to 0\qquad\textrm{as }\ep\to 0\,.
\end{equation}
Let finally $\ffi\in C_{\mathrm{c}}^1(\Omega)$ be such that  $\|\ffi\|_{C^{0,1}_{\mathrm{c}}(\Omega)}\le 1$\,. By the very definition of distributional Jacobian,
integrating by parts and using that
$$
Ju_\ep\res\Big(\Omega\setminus\bigcup_{\newatop{l=1}{m^l\neq 0}}^L\bigcup_{n=1}^{N^l_\ep}B_{\!\rho_\ep^{l,n}}(x_\ep^{l,n})\Big)=Jv_\ep\res\Big(\Omega\setminus\bigcup_{\newatop{l=1}{m^l\neq 0}}^L\bigcup_{n=1}^{N_\ep^l}B_{\!\rho_\ep^{l,n}}(x_\ep^{l,n})\Big)=0,
$$ 
we obtain
\begin{equation*}
	\begin{aligned}
		\langle Ju_\ep,\ffi\rangle_{\Omega}=&\,\sum_{\newatop{l=1}{m^l\neq 0}}^L\sum_{n=1}^{N^l_\ep}\int_{B_{\!\!\rho_\ep^{l,n}}(x_\ep^{l,n})}\nabla\ffi\cdot\ud T_{u_\ep}+\int_{\Omega\setminus\bigcup_{\newatop{l=1}{m^l\neq 0}}^L\bigcup_{n=1}^{N^l_\ep}B_{\!\!\rho_\ep^{l,n}}(x_\ep^{l,n})}\nabla\ffi\cdot\ud T_{u_\ep}\\
		=&\,\sum_{\newatop{l=1}{m^l\neq 0}}^L\sum_{n=1}^{N^l_\ep}\int_{B_{\!\!\rho_\ep^{l,n}}(x_\ep^{l,n})}\nabla\ffi\cdot\ud T_{u_\ep}-\sum_{\newatop{l=1}{m^l\neq 0}}^L\sum_{n=1}^{N^l_\ep}\int_{\partial B_{\!\!\rho_\ep^{l,n}}(x_\ep^{l,n})}\ffi T_{u_\ep}\cdot \nnu\ud\Huno\,,
	\end{aligned}
\end{equation*}
and
\begin{equation*}
	\begin{aligned}
		\langle Jv_\ep,\ffi\rangle_{\Omega}=&\,\sum_{\newatop{l=1}{m^l\neq 0}}^L\sum_{n=1}^{N^l_\ep}\int_{B_{\!\!\rho_\ep^{l,n}}(x_\ep^{l,n})}\nabla\ffi\cdot\ud T_{v_\ep}+\int_{\Omega\setminus\bigcup_{\newatop{l=1}{m^l\neq 0}}^L\bigcup_{n=1}^{N^l_\ep}B_{\!\!\rho_\ep^{l,n}}(x_\ep^{l,n})}\nabla\ffi\cdot\ud T_{v_\ep}\\
		=&\,\sum_{\newatop{l=1}{m^l\neq 0}}^L\sum_{n=1}^{N^l_\ep}\int_{B_{\!\!\rho_\ep^{l,n}}(x_\ep^{l,n})}\nabla\ffi\cdot\ud T_{v_\ep}-\sum_{\newatop{l=1}{m^l\neq 0}}^L\sum_{n=1}^{N^l_\ep}\int_{\partial B_{\!\!\rho_\ep^{l,n}}(x_\ep^{l,n})}\ffi T_{v_\ep}\cdot \nnu\ud\Huno\,.
	\end{aligned}
\end{equation*}
Therefore, by \eqref{espera_1}-\eqref{nellepalle2}, and using that {
	$$
	Ju_\ep(B_{\!\rho_\ep^{l,n}}(x_\ep^{l,n}))=Jv_\ep(B_{\!\rho_\ep^{l,n}}(x_\ep^{l,n})),\qquad \textrm{for every }l=1,\ldots,L\textrm{ with }m^l\neq 0,\,\, n=1,\dots, N_\ep^l,
	$$}
we obtain
\begin{equation*}
	\begin{aligned}
		\big|\langle Ju_\ep-\pi\mu_\ep,\ffi\rangle_{\Omega}\big|\le&\,\sum_{\newatop{l=1}{m^l\neq 0}}^{L}\sum_{n=1}^{N^l_\ep}\Big|\int_{\partial B_{\!\!\rho_\ep^{l,n}}(x_\ep^{l,n})}\ffi(T_{u_\ep}-T_{v_\ep})\cdot\nnu\ud\Huno\Big|\\
		&\,+|T_{u_\ep}|\Big(\bigcup_{\newatop{l=1}{m^l\neq 0}}^{L}\bigcup_{n=1}^{N^l_\ep}B_{\!\rho_\ep^{l,n}}(x_\ep^{l,n})\Big)+|T_{v_\ep}|\Big(\bigcup_{\newatop{l=1}{m^l\neq 0}}^L\bigcup_{n=1}^{N^l_\ep}B_{\!\rho_\ep^{l,n}}(x_\ep^{l,n})\Big)\,\\
		\le&\,\sum_{\newatop{l=1}{m^l\neq 0}}^{L}\sum_{n=1}^{N^l_\ep}\mathrm{osc}_{B_{\!\!\rho_\ep^{l,n}}(x_\ep^{l,n})}(\ffi)\int_{\partial B_{\!\!\rho^{l,n}_\ep}(x_\ep^{l,n})}\big(|T_{u_\ep}|+|T_{v_\ep}|\big)\ud\Huno\\
		&\,+ |\log\ep|\mathrm{o}(1)\,.
	\end{aligned}
\end{equation*}
Since $\mathrm{osc}_{B_{\!\!\rho_\ep^{l,n}}(x_\ep^{l,n})}(\ffi)\leq C\ep$, it follows that
\begin{equation*}
	\frac{1}{|\log\ep|}\left\|{Ju_\ep}-\pi\mu_\ep\right\|_{\flt,\Omega}\to 0\qquad\textrm{as }\ep\to 0\,,
\end{equation*}
whence \eqref{compj} follows by Lemma \ref{ovvieta}(b).

\medskip
{\it Case 2: General case.}
We argue by density, namely we show that for every $(\mu;\Td)\in (\M(\Omega)\cap H^{-1}(\Omega))\times L^2(\Omega;\R^2)$ with $\supp\mu\subset\subset\Omega$ and
$-\mathrm{Div}\, \Td=\pi\mu$ there exists a sequence $\{(\mu_k;\Td_k)\}_{k\in\N}\subset (\M(\Omega)\cap H^{-1}(\Omega))\times L^2(\Omega;\R^2)$ with $\supp\mu_k\subset\subset\Omega$ and $-\mathrm{Div} \Td_k=\pi\mu_k$ for every $k\in\N$ such that $\mu_k$ is locally constant for every $k$ (and takes the form as in Step 1), and
\begin{equation}\label{loccon}
	\mu_k\weakstar\mu\,,\qquad |\mu_k|(\Omega)\to|\mu|(\Omega)\,, \qquad \Td_k\to\Td\quad\textrm{in }L^2(\Omega;\R^2)\qquad\textrm{as }k\to +\infty\,.
\end{equation}
First, let $\{\rho_h\}_{h>0}$ be a sequence of standard mollifiers. We define
\begin{equation*}
	f_h:=\mu\ast \rho_h\,,\qquad\mu_h:=f_h\ud x\,,\qquad \Td_h:=(\Td\ast \rho_h)\res\Omega
\end{equation*}
By construction,  $-\mathrm{Div}\,\Td_h=2\pi\mu_h$ for every $h>0$ and, for $h$ small enough, $|\mu_h|(\partial\Omega)=0$\,. Moreover,
\begin{equation*}
	\mu_h\weakstar\mu\,,\qquad |\mu_h|(\Omega)\to |\mu|(\Omega)\,,\qquad\Td_h\to \Td\textrm{ in }L^2(\Omega;\R^2)\,.
\end{equation*}
Furthermore, since $\{f_h\}_{h>0}\subset C^{\infty}(\Omega)$\,, for every $h>0$\,, there exists a sequence $\{f_h^j\}_{j\in\N}$ with $f^j_h$ locally constant as in Step 1, such that
\begin{equation*}
	\|f^j_h-f_h\|_{L^\infty(\Omega)}\to 0\qquad\textrm{and}\qquad\int_{\Omega}(f^j_h-f_h)\ud x=0\,.
\end{equation*}
For every $h>0$ and $j\in\N$\,, let $w_h^j$ be the solution to 
\begin{equation*}	
	\left\{
	\begin{array}{ll}
		-\Delta w=f^j_h-f_h&\textrm{in }\Omega\\
		w=0&\textrm{on }\partial\Omega\,.
	\end{array}
	\right.
\end{equation*}
By standard elliptic estimates, we have
\begin{equation*}
	\|\nabla w^j_h\|_{L^2(\Omega;\R^2)}\le C\|f^j_h-f_h\|_{L^2(\Omega;\R^2)}\,.
\end{equation*}
Finally, for every $h>0$ and for any $j\in\N$\,, we set
$(\Td)_h^j:=\Td_h+2\pi\nabla w^j_h$\,, so that  $-\mathrm{Div}\,(\Td)_h^j=2\pi\mu_h^j$\,, and, for every $h>0$\,,
\begin{equation*}
	(\Td)_h^j\to \Td_h\qquad\textrm{in }L^2(\Omega;\R^2)\qquad\textrm{(as $j\to +\infty$)}\,.
\end{equation*}
Using a standard diagonal argument one can find a sequence $\{(\mu_k;\Td_k)\}_{k\in\N}$ satisfying \eqref{loccon}.
Finally, by arguing as in the second step of Case 1, we can construct the recovery sequence for the functional $\mathcal G_\ep$\,.
	\end{proof}
\section{Proof of Theorem \ref{mainthm_super}}\label{sec:6}
This section is devoted to the proof of Theorem \ref{mainthm_super}.

\begin{proof}[Proof of Theorem \ref{mainthm_super}]
	The compactness statement follows immediately by \eqref{enboundsuper} and \eqref{perlift}. 
	
	Analogously, the lower bound \eqref{liminfsuper} is a consequence of \eqref{perlift} and of the lower semicontinuity of the $L^2$ norm with respect to the weak convergence.
	
	Therefore, it  remains to prove only the upper bound.
	The proof is fully analogous to that of Theorem \ref{mainthm}(iii) in Subsection \ref{proofmainthmii}.
	We briefly sketch it.
	Let $\beta\in L^2(\Omega;\R^{2})$ be such that $\beta^\perp=\Td$ and we set $\mu:=\frac{1}{2\pi}\cu\beta=-\frac{1}{2\pi}\Div\Td$\,. Then $\mu\in\M(\Omega)\cap H^{-1}(\Omega)$\,. Moreover, by construction, $\supp\mu\subset\subset\Omega$\,.
	
	We show how to prove the claim only in the case $\mu=\chi_{\omega}\ud x$ for some $\omega\subset\Omega$ with Lipschitz continuous boundary, since the other cases follow by the former by arguing verbatim as in the proof of Theorem \ref{mainthm}(iii).
	Let 
	\begin{equation}\label{defmuep_bis}
		\mu_\ep:=\sum_{n=1}^{N_\ep^\omega}\delta_{x_\ep^n}
	\end{equation}
	be the measure provided by Lemma \ref{ovvieta} and let $\overline\beta_\ep\in L^2(\Omega;\R^2)$ be the field defined in \eqref{betafuori}. 
	By arguing verbatim as in Case 1 (first step) of the proof of Theorem \ref{mainthm}(iii), and using that here $N_\ep\gg |\log\ep|$\,, we have that
	\begin{equation*}
		\lim_{\ep\to 0}\frac{1}{N^2_\ep}\int_{\Omega}|\widehat K_\ep|^2\ud x=\lim_{\ep\to 0}\frac{1}{N_\ep^2}\sum_{n=1}^{N_\ep}\int_{A_{\ep,r_\ep}(x_\ep^n)}|\widehat K_\ep|^n\ud x=\lim_{\ep\to 0}\frac{1}{2\pi}\frac{\log r_\ep+|\log\ep|}{N_\ep}=0\,;
	\end{equation*}
	analogously, by arguing as in  \eqref{convfortektilde} and \eqref{tendeazero}, we get
	\begin{equation*}
		\lim_{\ep\to 0}\frac{1}{N^2_\ep}\int_{\Omega}|\widetilde K_\ep|^2\ud x=\lim_{\ep\to 0}\frac{1}{N^2_\ep}\int_{\Omega}|\nabla v_\ep|^2\ud x=0\,.
	\end{equation*}
	Therefore,
	\begin{equation}\label{unicaimp}
		\frac{\overline\beta_\ep}{N_\ep}\to 0\qquad\textrm{strongly in }L^2(\Omega;\R^2)\,.
	\end{equation}
	Finally, defining the sequence $\{u_\ep\}_\ep$ as in \eqref{recovery}, by \eqref{uno} and \eqref{convl1}, we have that
	\begin{equation*}
		\lim_{\ep\to 0}\frac{1}{N_\ep^2}\mathcal G_\ep(u_\ep)=\frac 1 2\int_{\Omega}|\beta|^2\ud x=\frac 1 2\int_{\Omega}|\Td|^2\ud x\,,
	\end{equation*}
	which concludes the proof of the claim in this case.
\end{proof}

\end{document}